\documentclass[francais,11pt]{smfart}

\usepackage[utf8]{inputenc}
\usepackage{amssymb,footnote}
\usepackage{multicol}
\usepackage{multirow}
\usepackage{dsfont}
\usepackage{mathrsfs}
\usepackage{dsfont}

\usepackage{stmaryrd}
\usepackage{color}
\usepackage[all]{xy}
\usepackage{enumerate}
\usepackage[russian,francais]{babel}
\usepackage{tikz}
\usepackage{url}
\usepackage{tabularx}

\usepackage{rotating}

\makeatletter
\def\@thm#1#2#3{%
  \ifhmode\unskip\unskip\par\fi
  \normalfont

  \trivlist
  \let\thmheadnl\relax
  \let\thm@swap\@gobble
  \thm@notefont{\fontseries\mddefault\upshape}%
  \thm@headpunct{.}
  \thm@headsep 5\p@ plus\p@ minus\p@\relax
  \thm@space@setup
  \setlength{\parskip}{0mm}
  #1
  \@topsep \thm@preskip               
  \@topsepadd \thm@postskip           
  \def\@tempa{#2}\ifx\@empty\@tempa
    \def\@tempa{\@oparg{\@begintheorem{#3}{}}[]}%
  \else
    \refstepcounter{#2}%
    \def\@tempa{\@oparg{\@begintheorem{#3}{\csname the#2\endcsname}}[]}%
  \fi
  \@tempa
}
\renewenvironment{proof}[1][\proofname]{\par
  \ifx@pushQED \pushQED{\qed}\fi
  \setlength{\parskip}{0mm}
  \normalfont
  \topsep6\p@\@plus6\p@ \trivlist \itemindent\z@ 
  \def\@proofhead{\normalfont\itshape #1}%
  \sbox\@tempboxa{\@proofhead}%
  \ifdim\wd\@tempboxa>0.7\linewidth \smf@skippttrue\fi
  \ifsmf@skippt
    \global\smf@skipptfalse
    \item[]{\@proofhead\@@par}
    \nobreak
  \else
    \item[\hskip\labelsep
          \unhbox\@tempboxa\pointrait]%
  \fi
  \ignorespaces
}{%
  \MakeQed
  \endtrivlist
  \@endpefalse
}
\makeatother

\newtheorem{thm''}{Théorème}
\newtheorem{cor''}[thm'']{Corollaire}
\newtheorem{conj''}[thm'']{Conjecture}

\newtheorem{prop'}{Proposition}[section]
\newtheorem{conj'}[prop']{Conjecture}
\newtheorem{lem'}[prop']{Lemme}
\newtheorem{thm'}[prop']{Théorème}
\newtheorem{cor'}[prop']{Corollaire}
\theoremstyle{definition}
\newtheorem{hyp'}[prop']{Hypothèse}
\newtheorem{definit'}[prop']{Définition}
\newtheorem{ex'}[prop']{Exemple}
\newtheorem{ques'}[prop']{Question}
\newtheorem{rem'}[prop']{Remarque}

\theoremstyle{plain}
\newtheorem{prop}{Proposition}[subsection]
\newtheorem{conj}[prop]{Conjecture}
\newtheorem{cor}[prop]{Corollaire}
\newtheorem{lem}[prop]{Lemme}
\newtheorem{thm}[prop]{Théorème}
\theoremstyle{definition}
\newtheorem{definit}[prop]{Définition}
\newtheorem{ex}[prop]{Exemple}
\newtheorem{rem}[prop]{Remarque}

\def\={\buildrel {\rm d\acute ef}\over =}

\DeclareMathOperator{\Id}{Id}

\DeclareMathOperator{\Spec}{Spec}
\DeclareMathOperator{\Spf}{Spf}
\DeclareMathOperator{\Spm}{Spm}

\DeclareMathOperator{\Ind}{Ind}

\DeclareMathOperator{\Gal}{Gal}

\DeclareMathOperator{\GL}{GL}
\newcommand{\nr}{\text{\rm nr}}

\DeclareMathOperator{\ab}{\mathrm{ab}}
\DeclareMathOperator{\tr}{\mathrm{tr}}

\newcommand{\calV}{\mathcal{V}}
\newcommand{\calX}{\mathcal{X}}
\newcommand{\calY}{\mathcal{Y}}
\newcommand{\calZ}{\mathcal{Z}}
\newcommand{\calD}{\mathcal{D}}
\renewcommand{\P}{\mathbb{P}}

\newcommand{\bbM}{\mathbb{M}}

\newcommand{\SK}{\mathfrak{S}}
\newcommand{\MK}{\mathfrak{M}}

\renewcommand{\sp}{\mathrm{sp}}

\newcommand{\oubgenre}{\text{\foreignlanguage{russian}{\it ж}}}

\newcommand{\I}{\mathrm{I}}
\newcommand{\II}{\mathrm{II}}
\newcommand{\oF}{{\mathcal O}_F}

\newcommand{\oE}{{\mathcal O}_E}
\newcommand{\oEp}{{\mathcal O}_{E'}}

\newcommand{\ocE}{{\mathcal O}_{\mathcal{E}}}
\newcommand{\ocEnr}{{\mathcal O}_{\mathcal{E}^{\nr}}}

\newcommand{\oK}{{\mathcal O}_K}

\renewcommand{\1}{\mathds 1}
\newcommand{\Z}{{\mathbb Z}} 
\newcommand{\N}{{\mathbb N}} 
\newcommand{\Q}{{\mathbb Q}} 
 
\newcommand{\Qp}{\Q_{p}} 
 
\newcommand{\Zp}{\Z_{p}}

\newcommand{\F}{\mathbb F}
\newcommand{\Fp}{{\mathbb F}_{p}}
\newcommand{\Fq}{{\mathbb F}_{q}}

\newcommand{\Qbar}{\overline\Q}
\newcommand{\Qpbar}{\Qbar_p}

\newcommand{\rhobar}{\overline\rho}

\newcommand{\pr}{\mathrm{pr}}

\newcommand{\vv}{{\bf v}}

\newcommand{\ttt}{{\rm t}}

\newcommand{\m}{{\mathfrak m}}

\newcommand{\cL}{\mathscr{L}}
\newcommand{\GR}{\mathscr{G\!R}}
\newcommand{\vK}{\overline{\GR}}
\newcommand{\gA}{\text{\tt A}}
\newcommand{\gB}{\text{\tt B}}
\newcommand{\gAB}{\text{\tt AB}}
\newcommand{\gO}{\text{\tt O}}
\definecolor{colA}{rgb}{0,0,0.5}
\definecolor{colB}{rgb}{0,0.4,0}
\definecolor{colstr}{rgb}{0.5,0.1,0.1}
\definecolor{ptmarque}{rgb}{0.4,0.3,0}
\definecolor{colC}{rgb}{0.7,0,0}
\definecolor{colD}{rgb}{0.2,0.2,0.2}
\definecolor{colE}{rgb}{0.4,0.4,0.4}
\definecolor{colF}{rgb}{0.25,0.25,0}
\definecolor{colG}{rgb}{0.85,0.85,0.9}

\def \Ieta{\mathrm I_{\eta}}
\def \Ietap{\mathrm I_{\eta'}}

\def \II{\mathrm{II}}

\renewcommand{\epsilon}{\varepsilon}
\renewcommand{\leq}{\leqslant}
\renewcommand{\geq}{\geqslant}

\author[X. Caruso]{Xavier Caruso}
\address{IRMAR,
Université de Rennes 1, UMR 6625,
Campus de Beaulieu,
35042 Rennes Cedex, France}
\email{xavier.caruso@normalesup.org}

\author[A. David]{Agnès David}
\address{Laboratoire de Mathématiques de Besançon, UMR 6623,
Université de Franche-Comté,
16 route de Gray,
25030 Besançon Cedex, France}
\email{Agnes.David@math.cnrs.fr}

\author[A. Mézard]{Ariane Mézard}
\address{Institut de Mathématiques de Jussieu Paris Rive-Gauche,  UMR 7586,  LabEx SMP
Université Pierre et Marie Curie,
75005 Paris, France}
\email{ariane.mezard@upmc.fr}

\title[Variétés de Kisin stratifiées]{Variétés de Kisin stratifiées et\\déformations potentiellement Barsotti--Tate}

\begin{document} 

\begin{abstract}
Soient $F$ une extension finie non ramifiée de $\Qp$ et $\rhobar$ une 
représentation modulo $p$ irréductible de dimension $2$ du groupe de 
Galois absolu de $F$.
L'objet de ce travail est la détermination de la variété de Kisin 
qui paramètre les modules de Breuil--Kisin associés à certaines 
familles de déformations potentiellement Barsotti--Tate de $\rhobar$. 
Nous démontrons que cette variété est une réunion finie de produits de 
$\P^1$ qui s'identifie à une sous-variété explicite connexe de 
$(\P^1)^{[F:\Qp]}$. Nous définissons une stratification de la variété 
de Kisin en sous-schémas localement fermés et expliquons enfin comment
la variété de Kisin ainsi stratifiée peut aider à déterminer l'anneau
des déformations potentiellement Barsotti--Tate de $\rhobar$.
\end{abstract}

\begin{altabstract}
Let $F$ be a unramified finite extension of $\Qp$ and $\rhobar$ be
an irreducible mod $p$ two-dimensional representation of the absolute
Galois group of $F$.
The aim of this article is the explicit computation of the Kisin
variety parameterizing the Breuil--Kisin modules associated to certain families
of potentially Barsotti--Tate deformations of $\rhobar$. We prove that
this variety is a finite union of products of $\P^1$. Moreover,  it
appears as an explicit closed subvariety of $(\P^1)^{[F:\Qp]}$. We
define a stratification of the Kisin variety by locally closed
subschemes and explain how the Kisin variety equipped with its
stratification may help in determining the ring of Barsotti--Tate
deformations of $\rhobar$.
\end{altabstract}

\maketitle

\setcounter{tocdepth}{1}
\tableofcontents

\section*{Introduction}\label{intro}

Soient $p$ un nombre premier et $F$ une extension finie de $\Qp$ dont le 
groupe de Galois absolu est noté $G_F$. Nous savons depuis Mazur 
\cite{Ma} que l'ensemble des déformations de déterminant fixé d'une 
représentation absolument irréductible $\rhobar$ de $G_F$ à coefficients 
dans un corps fini de caractéristique $p$ est muni d'une structure 
géométrique. \'Etant donnés une extension finie $E$ de $\Qp$ de corps 
résiduel $k_E$ et d'anneau des entiers $\oE$ et un caractère $\psi$ de 
$G_F$ dans $\oE^\times$, Mazur a construit dans \emph{loc. cit.} une 
$\oE$-algèbre $R^{\psi}(\rhobar)$ dont l'ensemble des points à valeurs 
dans une $\oE$-algèbre locale complète noetherienne $R$ de corps 
résiduel $k_E$ s'identifie fonctoriellement à l'ensemble des 
$R$-représentations de $G_F$ de déterminant $\psi$ qui se réduisent sur 
$\rhobar$ modulo l'idéal maximal de $R$.

Récemment, Kisin \cite{Ki1,Ki4} a démontré que certaines conditions 
issues de la théorie de Hodge $p$-adique définissent des sous-schémas fermés de 
$\Spec R^{\psi} (\rhobar)$. Plus précisément, étant donnés un type de 
Hodge $\vv$ et un type galoisien  $\ttt$, le résultat de 
Kisin établit l'existence d'un unique quotient $R^{\psi}(\vv, \ttt, 
\rhobar)$ de $R^{\psi}(\rhobar)$ qui est réduit, sans $p$-torsion et 
vérifie la condition suivante : pour toute extension finie $E'$ de $E$ 
d'anneaux des entiers $\oEp$, un morphisme de $R^{\psi}(\rhobar)$ dans $\oEp$ se 
factorise par $R^{\psi}(\vv, \ttt, \rhobar)$ si et seulement si la 
représentation qui lui est associée est potentiellement cristalline de 
poids de Hodge--Tate $\vv$ et la représentation de Weil--Deligne qui 
lui est associée par Fontaine (\emph{cf} \cite{Fo1}) est isomorphe à 
$\ttt$. Pour démontrer ce théorème, Kisin construit un schéma 
muni d'un morphisme vers $\Spec R^{\psi}( \rhobar)$ dont l'adhérence 
schématique de l'image se trouve être le spectre de l'anneau 
$R^{\psi}(\vv, \ttt, \rhobar)$\footnote{\emph{Stricto sensu}, cette description rapide de 
la démonstration de Kisin n'est correcte que dans les cas les plus
simples, à savoir le cas Barsotti--Tate discuté par la suite. Il est
toutefois possible de reformuler la preuve de Kisin dans les autres cas en suivant les
grandes lignes que nous esquissons ici.}.
Ce schéma, que nous notons ici $\GR^{\psi}(\vv, \ttt, 
\rhobar)$, est obtenu comme espace de modules de \emph{réseaux de 
Breuil--Kisin}.

Le cas initialement étudié par Kisin dans \cite{Ki1} est le cas  
Barsotti--Tate pour lequel le type de Hodge ne fait intervenir 
que les entiers $0$ et $1$ et le type galoisien est trivial. Les 
représentations correspondantes sont alors  associées aux groupes 
de Barsotti--Tate (éventuellement tronqués). Sous ces 
hypothèses, Kisin s'intéresse aux fibres
$$\vK^{\psi}(\vv, \ttt, \rhobar) = \Spec k_E \times_{\Spec 
R^{\psi}(\rhobar)} \GR^{\psi}(\vv, \ttt, \rhobar)$$
qui sont définies purement en caractéristique $p$ et correspondent à un 
problème de modules plus facile à appréhender.

Dans leur article \cite{PR}, Pappas et Rapoport ont donné le nom de 
\emph{variétés de Kisin} aux variétés $\vK^{\psi}(\vv, \ttt, \rhobar)$. Depuis, 
plusieurs auteurs se sont intéressés aux propriétés géométriques des 
variétés de Kisin. En lien avec la connexité, Hellmann a notamment 
démontré dans \cite{He} que les variétés de Kisin correspondant aux 
déformations Barsotti--Tate d'une représentation irréductible de 
dimension $2$ sont connexes, ce qui implique la connexité de la fibre 
générique de l'espace de déformations associé. Les dimensions de 
certaines variétés de Kisin ont également été étudiées par Hellmann 
\cite{He1}, Imai \cite{Im} et Caruso \cite{Ca3}.

Pour la première fois dans \cite{CDM} est apparue une variété de Kisin 
associée à un type galoisien non trivial. Cet article faisait suite à 
\cite{BM} et se plaçait dans la situation particulière suivante : 
l'extension $F$ est \emph{non ramifiée} de degré $f$ sur $\Qp$ avec $p \geq 5$, le 
type de Hodge $\vv$ est égal à $(0,1)$ en toutes les places, le type galoisien 
$\ttt$ est la somme directe de deux caractères \emph{modérément 
ramifiés} de niveau $f$ de $G_F$ et la représentation galoisienne $\rhobar$ est 
absolument irréductible de dimension $2$. Les résultats obtenus dans 
\cite{CDM} --- qui concernent essentiellement le cas où $F$ est de degré 
$2$ sur $\Qp$ --- suggèrent un lien encore plus étroit que celui 
exprimé par Kisin entre la variété de Kisin 
$\vK^{\psi}(\vv, \ttt, \rhobar)$ et (la fibre 
générique de) l'espace de déformations $\Spec R^{\psi}(\vv, \ttt, 
\rhobar)$.

Afin d'étudier ce lien hypothétique, la première étape est de calculer 
les variétés de Kisin. C'est l'objet du présent article. Nous nous 
plaçons dans le cadre de \cite{CDM} (pour tout degré $f$) et, sous ces 
hypothèses, nous donnons une description complète des variétés 
$\vK^{\psi}(\vv, \ttt,\rhobar)$. Nous montrons le théorème suivant, où 
$\varepsilon$ désigne le caractère cyclotomique $p$-adique de $G_F$ dans 
$\Z_p^\times$ (nous renvoyons le lecteur au théorème 
\ref{thequationsKisin} dans le corps du texte pour un énoncé détaillé).

\begin{thm''}[Théorème \ref{thequationsKisin}]
\label{thmintro:varKisin}
Nous supposons  que
\begin{itemize}
\item le corps $F$ est une extension non ramifiée de degré $f$ de $\Qp$ avec $p \geq 5$,
\item le type de Hodge $\vv$ vaut $(0,1)$ en chaque place,
\item le type galoisien $\ttt=\eta\oplus \eta'$ est la somme directe de deux caractères
modérément ramifiés de niveau $f$ du sous-groupe d'inertie de $G_F$ à valeurs dans
$\oE^\times$,
\item la représentation $\rhobar : G_F \to \GL_2(k_E)$ est absolument
irréductible,
\item le caractère $\psi$ restreint au sous-groupe d'inertie coïncide 
avec le produit $\eta \eta'\varepsilon$.
\end{itemize}
Alors la variété de Kisin $\vK^{\psi}(\vv, \ttt,\rhobar)$ apparaît comme 
un sous-schéma fermé réduit de $(\P^1_{k_E})^{[F:\Qp]}$ défini par une famille 
d'équations explicites.
\end{thm''}

Comme conséquence du théorème \ref{thmintro:varKisin}, nous déduisons 
diverses propriétés géométriques des variétés de Kisin et des espaces de 
déformations, notamment :

\begin{cor''}[Corollaire \ref{corcon}]
\label{corintro:connexite}
Sous les hypothèses du théorème \ref{thmintro:varKisin}, la variété de 
Kisin $\vK^{\psi}(\vv, \ttt, \rhobar)$ est connexe.
\end{cor''}

Nous 
donnons également des conditions nécessaires et suffisantes sur le 
couple $(\rhobar, t)$ pour que la variété $\vK^{\psi}(\vv, \ttt, 
\rhobar)$ soit non vide (\emph{cf} corollaire \ref{corvide}). Notons que 
ceci se produit si et seulement si l'anneau $R^{\psi}(\vv, \ttt, \rhobar)$ est non nul.

Nous déduisons de notre étude qu'il est trop optimiste de penser 
que la variété de Kisin $\vK^{\psi}(\vv, \ttt, \rhobar)$ à elle 
seule détermine $\Spec R^{\psi}(\vv, \ttt, \rhobar)$ ou même 
seulement sa fibre générique. Toutefois, nous proposons une version 
raffinée plausible de cette idée. Pour ce faire, 
reprenant les techniques de \cite{BM,CDM}, nous définissons une 
stratification de $\vK^{\psi}(\vv, \ttt, \rhobar)$ par des 
sous-schémas localement fermés et formulons la conjecture suivante.

\begin{conj''}[Conjecture \ref{conj:D}]
\label{conj'':D}
Si le type galoisien $\ttt$ est non dégénéré\footnote{C'est
une hypothèse faible de généricité précisée dans la définition  \ref{hypnondegenere}.},
la variété de Kisin $\vK^{\psi}(\vv, \ttt, \rhobar)$ munie de sa
stratification détermine l'anneau $R^\psi(\vv,\ttt,\rhobar)[1/p]$.
\end{conj''}

Cette conjecture est vraie si $f=2$ (hormis peut-être pour quelques cas 
très particuliers) par les travaux de \cite{CDM}. Elle l'est également 
lorsque la variété de Kisin est réduite à un point, ce qui se produit 
pour une représentation générique. Nous concluons cet article en 
proposant un candidat pour (la variété rigide ayant pour anneau) 
$R^\psi(\vv,\ttt,\rhobar)[1/p]$ qui, conformément à la conjecture 
\ref{conj'':D}, est construit uniquement à partir de la variété de 
Kisin stratifiée.

Pour démontrer le théorème \ref{thmintro:varKisin}, notre méthode 
consiste d'abord à associer à $\rhobar$ et $\ttt$, une donnée combinatoire que nous appelons le \emph{gène} de 
$(\rhobar, \ttt)$ et que nous notons $X$. Concrètement, 
il s'agit d'une suite de $2 \cdot [F:\Qp]$ symboles de l'ensemble $\{\gA, 
\gB, \gAB, \gO\}$ qu'il est commode de représenter sur un ruban de Moebius. 
Ensuite, nous démontrons que les équations de la variété de 
Kisin se lisent sur $X$ à l'aide de manipulations 
combinatoires élémentaires. En d'autres termes, le gène $X$ est une donnée épurée et extrêmement simple qui capture 
intégralement la géométrie de la variété de Kisin $\vK^{\psi}(\vv, 
\ttt, \rhobar)$ --- qui, elle, peut être compliquée.
À partir de là, le corollaire \ref{corintro:connexite} découle d'une 
étude combinatoire portant sur les gènes $X$. La stratification sur
$\vK^{\psi}(\vv,\ttt, \rhobar)$ s'obtient, elle aussi, aisément, à
partir du gène, de même que le candidat que nous proposons pour 
$R^\psi(\vv,\ttt,\rhobar)[1/p]$.

Le plan de l'article est le suivant. Dans le \S \ref{sec:Definitions}, 
nous rappelons toutes les notions utiles à la définition rigoureuse des 
variétés de Kisin $\vK^{\psi}(\vv, \ttt, \rhobar)$ dans le cadre qui nous 
intéresse. Le \S \ref{sec:Genes} est consacré à la définition des gènes 
$X$ et à l'énoncé précis du théorème \ref{thmintro:varKisin}. Pour familiariser 
le lecteur avec les gènes, nous détaillons en outre quelques exemples que 
nous pensons représentatifs. La démonstration du théorème 
\ref{thmintro:varKisin} est reportée au \S \ref{sec:Demonstrations} et 
les conséquences géométriques, et en particulier le corollaire 
\ref{corintro:connexite}, sont discutées au \S \ref{sec:Connexite}. 
Enfin, dans le \S \ref{sec:Stratification}, nous définissons la 
stratification par le genre sur les variétés de Kisin et expliquons en 
quoi nous pensons qu'elle est liée à la géométrie des espaces de 
déformations potentiellement Barsotti--Tate.

Les auteurs remercient Bernard Le Stum et Alberto Vezzani pour leurs
explications et leurs réponses toujours pertinentes sur les questions
de géométrie $p$-adique.

Les recherches menant aux présents résultats ont bénéficié d'un soutien financier du septième programme-cadre de l'Union européenne (7ePC/2007-2013) en vertu de la convention de subvention numéro 266638.

\section{Variétés de Kisin avec donnée de descente}
\label{sec:Definitions}

Toutes les extensions de $\Qp$ considérées sont supposées 
contenues une clôture algébrique $\Qpbar$ fixée de $\Qp$. Pour $K$ une 
telle extension, nous notons $G_K=\Gal(\Qpbar/K)$, $\oK$ son anneau 
d'entiers, $\pi_K$ une uniformisante et $k_K$ le corps r\'esiduel de 
$\oK$.

\subsection{Les données : représentation et type galoisiens}
\label{ssec:donnees}

Soient $E$ une extension finie de $\Qp$ et $F$ une extension finie non 
ramifiée de $\Qp$ de degré $f\geq 2$. Posons $q = p^f$ et $e=p^f-1$. 
Soit $F'$ l'unique extension quadratique non ramifiée de $F$ dans 
$\Qpbar$. Nous supposons $F'\subset E$. Nous fixons un plongement 
$\tau_0$ de $F$ dans $E$ et pour $0\leq i\leq f-1$, nous notons $\tau_i$ 
le plongement $\tau_0 \circ \varphi^{-i}$ où $\varphi$ désigne 
l'endomorphisme de Frobenius sur $F$. Nous notons $F^\nr$ la plus grande 
extension de $F$ non ramifiée dans $\Qpbar$ et $G_F^{\ab}$ le plus grand 
quotient abélien de $G_F$.

Soit $L$ le corps obtenu en adjoignant à $F$ une racine $e$-ième de 
$-p$, notée $\sqrt[e]{-p}$. Il s'agit d'une extension totalement ramifiée de $E$ et la projection à gauche sur $\Gal(L/F)$ 
et à droite sur les repr\'esentants multiplicatifs $[\Fq^{\times}]\cong 
\Fq^{\times}$ (en envoyant $p^{\Z}(1+p\oF)$ sur $1$) induit 
l'isomorphisme 
\begin{eqnarray}\label{fondamental}
\nonumber \omega_f :  \Gal(L/F) & \buildrel\sim\over\longrightarrow & (\oF/p)^{\times}=k_F^{\times}\\
g &\longmapsto & \overline{\frac{g(\sqrt[e]{-p})}{\sqrt[e]{-p}}}
\end{eqnarray}
par lequel nous voyons tout caract\`ere de $k_F^{\times}$ comme un 
caract\`ere de $\Gal(L/F)$ et r\'eciproquement.
Notons $\omega_f: G_F\rightarrow k_E^\times$ le caractère fondamental de 
niveau $f$ induit sur $G_F$ par (\ref{fondamental}) et le plongement 
$\tau_0$. De fa\c{c}on analogue, nous notons 
$\omega_{2f}:G_{F'}\rightarrow k_E^\times$ le caractère fondamental de 
niveau $2f$ de $F'$ (après avoir choisi un plongement 
$\tau'_0:F'\rightarrow E$ qui prolonge $\tau_0$ à $F'$).
Soit $\varepsilon:G_F\longrightarrow \Z_p^\times$ le caractère 
cyclotomique $p$-adique et $\omega$ sa réduction modulo $p$. Pour 
$\theta$ dans $k_E^\times$, nous notons enfin 
$\nr'(\theta):G_{F'}\rightarrow k_E^\times$ l'unique caract\`ere non 
ramifié qui envoie le Frobenius arithmétique de $G_{F'}$ sur $\theta$.
Soit $\rhobar:G_F\longrightarrow\GL_2(k_E)$ la représentation 
galoisienne (continue) irréductible
$$\rhobar \simeq 
\Ind_{G_{F'}}^{G_F} \Big( \omega_{2f}^h \cdot \nr'(\theta)\Big),$$
avec $h$ entre $0$ et $p^{2f}-2$ et $\theta$ dans $k_E^\times$. Comme 
$\rhobar$ est supposée irréductible, l'entier $h$ n'est pas un multiple 
de $q+1$. Par conséquent, il existe des entiers $h_i$ ($0 \leq i \leq 
f-1$) dans l'intervalle $\llbracket 0, p-1 \rrbracket$ uniquement 
déterminés tels que :
\begin{equation}
h \equiv 1 + \sum_{i=0}^{f-1} h_i p^{f-1-i} \pmod{q+1}.
\end{equation}

Fixons deux caractères modérés de niveau $f$ distincts $\eta,\eta': I_F 
\rightarrow \oE^\times$ qui s'étendent à~$G_F$. Le type galoisien 
$\ttt=\eta\oplus\eta'$ est une représentation de noyau ouvert $I_F 
\rightarrow\GL_2(E)$.
Notons $\bar{\eta}$ (resp. $\bar{\eta}'$) la réduction modulo $p$ de 
$\eta$ (resp. $\eta'$).  Il existe donc $c\in\llbracket 0, 
p^f-2\rrbracket$ tel que $\bar{\eta}\cdot (\bar{\eta}')^{-1}=\tau_0^{c}$ et pour 
$0\leq i<f$ nous notons $c_i\in\llbracket 0,p-1\rrbracket$ les entiers 
définis par l'égalité $c=\sum_{i=0}^{f-1}c_ip^i$. Nous introduisons une 
notion de dégénérescence pour les représentations et les types 
galoisiens.

\begin{definit} 
\label{hypnondegenere}
La représentation $\rhobar$ est dite \emph{non dégénérée} s'il existe 
un entier $i_0\in\llbracket 0,f-1\rrbracket$ tel que 
$h_{i_0}\not\in\{0,p-1\}$.

Le type galoisien $\ttt$ est dit \emph{non dégénéré} s'il existe 
un entier $j_0\in\llbracket 0,f-1\rrbracket$ tel que 
$c_{j_0}\not\in\{0,1,p-2,p-1\}$.
\end{definit}

Pour l'instant, nous ne supposons aucune propriété de non dégénérescence 
ni sur $\rhobar$, ni sur $\ttt$. Nous explicitons ces hypothèses 
 dans les énoncés lorsqu'elles sont nécessaires. Remarquons que, lorsque 
$p$ est fixé et $f$ tend vers l'infini, la proportion de représentations 
dégénérées (resp. de types galoisiens dégénérés) tend vers $0$. Ceci 
contraste avec la notion de représentations génériques qui avait été 
considérée dans \cite{Br} puisque, lorsque $p$ est fixé et $f$ tend vers 
l'infini, la proportion de représentations non génériques tend, elle, 
vers $1$.

Nous fixons le type de Hodge $\vv = (0,1)^f$ ainsi qu'un 
caractère continu $\psi:G_F\rightarrow \oE^\times$ tel que 
$\psi_{|I_F}=\varepsilon\det \ttt.$ Une représentation continue de $G_F$ 
sur un $E$-espace vectoriel de dimension $2$ est dite potentiellement 
Barsotti--Tate de type $(\vv,\ttt)$ si elle est potentiellement 
cristalline avec poids de Hodge--Tate $\vv$ et si la 
représentation de Weil--Deligne qui lui est attachée par \cite{Fo1} est 
isomorphe à $\ttt$ en restriction à $I_F$.
Une condition nécessaire d'existence d'un relèvement potentiellement 
Barsotti--Tate de type $(\vv,\ttt)$ de $\rhobar$ est donc
\begin{eqnarray}
\det\rhobar_{|I_F}=(\bar{\eta}\bar{\eta}'\omega)_{|I_F}.
\end{eqnarray}
Nous supposons désormais satisfaite cette hypothèse sur le type 
galoisien $\ttt$.

\subsection{Les $\varphi$-modules étales} 
\label{secdescente}

Par commodité pour le lecteur, nous rappelons ici les notions de théorie 
de Hodge $p$-adique indispensables pour définir les modules de 
Breuil--Kisin qui paramètrent les déformations géométriques de 
$\rhobar$. Pour une version plus détaillée, nous renvoyons à \cite{Fo2}, 
\cite{Ki2}, 
\cite{Ca1}. Posons $W = \oF$ et notons $\varphi$ l'endomorphisme de 
Frobenius agissant sur $W$ et $F$. Fixons également un système 
compatible $(\pi_s)_{s \in \N}$ de racines $p^s$-ièmes de $(-p)$ dans 
$\Qpbar$ et, pour tout entier $s$, posons $F_s = F(\pi_s)$ et $L_s = 
L(\pi_s)$. Définissons également $F_\infty = \bigcup_s F_s$ et $L_\infty 
= \bigcup_s L_s$. Les groupes de Galois correspondants sont notés
$G_{F_\infty}$ et $G_{L_\infty}$. Le quotient $G_{L_\infty} /
G_{F_\infty} = \Gal(L_\infty/F_\infty)$ s'identifie naturellement
à $\Gal(L/F)$.
Introduisons enfin les anneaux $\SK = W[[u]]$ et
$$\ocE = \left\{ \sum_{i \in \Z} a_i u^i \quad \Big|\quad a_i \in W, 
\lim_{i \to - \infty} a_i = 0 \right\}$$
le complété $p$-adique de $\SK[1/u]$. Ils sont, tous deux, munis d'un 
endomorphisme de Frobenius $\varphi$ défini par
$\varphi\big(\sum_i a_i u^i\big) = \sum_i \varphi(a_i) u^{pi}$.
Ils sont également munis d'une action de $\Gal(L/F)$ définie par la
formule $g \cdot \big(\sum_i a_i u^i\big) = \sum_i [\omega_f(g)]^i \: 
u^i$ où $[\cdot]$ désigne le représentant de Teichmüller.

\begin{definit}\label{defi-phi-module}
Soit $R$ une $\Zp$-algèbre locale complète noetherienne.
Un \emph{$\varphi$-module} sur $R \hat\otimes_{\Zp} \ocE$ est un
$(R \hat\otimes_{\Zp} \ocE)$-module $M$ libre de rang fini muni d'une 
application $\varphi : M \to M$ qui est $\varphi$-semi-linéaire par 
rapport à $\ocE$ et linéaire par rapport à $R$.

Le $\varphi$-module $M$ est dit \emph{étale} si l'image de $\varphi$ 
engendre $M$ comme $(R \hat\otimes_{\Zp} \ocE)$-module.
\end{definit}

Soit $R$ une $\Zp$-algèbre locale noethérienne. Par la théorie de 
Fontaine et Wintenberger \cite{Fo2,Ki2}, nous savons associer à toute 
$R$-représentation $V$ de $G_{F_\infty}$ un $\varphi$-module étale sur 
$(R \hat\otimes_{\Zp} \ocE)$ défini par la formule
$\bbM(V) = ( V(-1)  \hat\otimes_{\Zp}  \ocEnr )^{G_{L_{\infty}}}$.
Il est muni d'une action de $\Gal(L_\infty/F_\infty) \simeq \Gal(L/F)$, 
qui est semi-linéaire par rapport à l'action de $\Gal(L/F)$ sur $\ocE$.

Supposons à présent que $R$ soit une $W$-algèbre. Nous avons une
décomposition canonique de l'anneau $R \hat\otimes_{\Zp} W$ :
\begin{eqnarray*}\label{decomposition}
R \hat\otimes_{\Zp} W & \simeq & \textstyle \prod^{f-1}_{i \in 0} R \\
x \otimes y & \mapsto & (x \cdot \varphi^{-i}(y) )_{0 \leq i < f}.
\end{eqnarray*}
En tensorisant par $\ocE$ sur $W$, nous en déduisons un isomorphisme 
canonique :
\begin{equation}
\label{eq:decompocER}
R  \hat\otimes_{\Zp} \ocE  \simeq 
\textstyle \prod^{f-1}_{i \in 0} R \hat\otimes_{W, \iota \circ \varphi^{-i}}  \ocE,
\end{equation}
o\`u $\iota$ désigne le morphisme structurel faisant de $R$ une $W$-algèbre.
Concrètement, le $i$-ième facteur $R \hat\otimes_{W, \iota \circ 
\varphi^{-i}} \ocE$ admet la description explicite suivante :
\begin{equation}
\label{eq:ocERexpl}
R \hat\otimes_{W, \iota \circ \varphi^{-i}} \ocE \,\, \simeq \,\,
\Bigg\{ \sum_{j \in \Z} a_j u^j \quad  \Big|\quad  a_j \in R, \lim_{j \to - \infty} a_j = 0 \Bigg\},
\end{equation}
l'identification faisant correspondre le tenseur pur $\lambda \otimes  (\sum a_j u^j) $ avec la série $\lambda \sum \varphi^{-i} (a_j) \:u^j$. 
Avec ce choix, l'isomorphisme \eqref{eq:ocERexpl} est compatible à la fois à la
multiplication par les éléments de $W$ agissant sur le facteur $\ocE$
et à la multiplication par les éléments de $R$ agissant sur le facteur
$R$.

En notant $e_i$ l'idempotent de $R \hat\otimes_{\Zp} \ocE$ correspondant 
au $i$-ième facteur de la décomposition \eqref{eq:decompocER}, nous 
voyons que tout module $M$ sur $R \hat\otimes_{\Zp} \ocE$ se décompose 
canoniquement comme une somme directe :
\begin{equation}\label{decompositionmodule}
M = M^{(0)} \oplus M^{(1)} \oplus \cdots \oplus M^{(f-1)}
\end{equation}
où $M^{(i)} = e_i M$ peut être considéré comme un module sur l'anneau
$R \hat\otimes_{W, \iota \circ \varphi^{-i}} \ocE$.

Nous pouvons décrire explicitement le $\varphi$-module correspondant à 
la représentation $\rhobar = \Ind_{G_{F'}}^{G_F} \big( \omega_{2f}^h 
\cdot \nr'(\theta)\big)$  : pour 
tout $i$ dans~$\Z/f\Z$, il existe des bases $(\varepsilon_0^{(i)}$, 
$\varepsilon_1^{(i)})$ du $k_E((u))$-espace vectoriel 
$\bbM(\rhobar)^{(i)}$ qui sont fixées par $\Gal(L/F)$ et dans 
lesquelles, en posant $v = u^e$, les matrices du Frobenius $\varphi : 
\bbM(\rhobar)^{(i)}\longrightarrow\bbM(\rhobar)^{(i+1)} $ sont :
$$
\begin{array}{rl} \Id& \mbox{  pour } 0\leq i\leq f-2, \medskip \\
\left(\begin{array}{ccc} 
0&\theta^{-1} v^h \cr
1& 0\cr\end{array}\right) 
\cdot v^{\frac{e}{p-1} - h - k}&\mbox{ pour } i=f-1 .\cr\end{array}$$

\subsection{Modules de Breuil--Kisin et variétés de Kisin}

Si $\rho : G_F \to \GL_2(\oE)$ est une déformation de $\rhobar$, il est 
possible de lire sur le $\varphi$-module étale $\bbM(\rho)$ si elle est 
potentiellement Barsotti--Tate de type $(\vv,\ttt)$. Le théorème 
\ref{th:modBK} (dont la démonstration est une conséquence simple des 
résultats de Kisin \cite{Ki-Crys} et est détaillée dans 
\cite[Proposition 3.1.4]{CDM}) précise ce point.

\begin{thm}
\label{th:modBK}
Soit $R$ l'anneau des entiers d'une extension finie de $E$ et soit
$\rho$ une déformation de $\rhobar$ à $R$ de déterminant $\psi$. 
Alors $\rho$ est potentiellement Barsotti--Tate de type $(\vv,\ttt)$ si, 
et seulement s'il existe un sous-$(R \hat\otimes_{\Zp} \SK)$-module 
$\MK \subset \bbM(\rho)$ tel que :

\begin{enumerate}[(i)]
\item
$\MK$ est projectif de rang $2$ sur $R \hat\otimes_{\Zp} \SK$ ;
\item
$\MK$ engendre $\bbM(\rho)$ comme $(R \hat\otimes_{\Zp} \ocE)$-module ;
\item
$\MK$ est stable par les actions de $\varphi$ et $\Gal(L/F)$ ;
\item
l'idéal déterminant de $\varphi : \MK \to \MK$ est l'idéal principal engendré 
par $u^e + p$ ;
\item
l'action de $\Gal(L/F)$ sur le quotient $\MK / u \MK$ est donnée par 
$\ttt$.
\end{enumerate}

\noindent
De plus, si un tel $\MK$ existe, il est unique.
\end{thm}

D'après le théorème \ref{th:modBK}, lorsque $R$ est l'anneau des entiers 
d'une extension finie de $E$, l'étude des $R$-points de 
$R^\psi(\vv,\ttt,\rhobar)$ se ramène à un problème de classification 
d'algèbre semi-linéaire qui paraît plus abordable. Afin de rendre cette 
idée plus précise, fixons momentanément une $\oE$-algèbre locale 
artinienne $R$ de corps résiduel $k_E$ ainsi qu'une déformation $\rho_R 
: G_F \to \GL_2(R)$ de $\rhobar$ de déterminant $\psi$. Pour définir le 
schéma formel qui nous intéresse (\ref{eq:GRhat}), nous introduisons la définition 
suivante.

\begin{definit}\label{defmoduleBK}
Soit $S$ une $R$-algèbre.
Un \emph{réseau de Breuil--Kisin de type $(\vv,\ttt)$} de $S \otimes_R
\bbM(\rho_R)$ est la donnée d'un sous-$(S \otimes_{\Zp} \SK)$-module 
$\MK$ de $S \otimes_R \bbM(\rho_R)$ tel que 
\begin{enumerate}[(i)]
\item le module $\MK$ est projectif de rang $2$ ;
\item $\MK$ engendre
$S \otimes_R \bbM(\rho_R)$ comme $(S \otimes_{\Zp} \ocE)$-module ;
\item $\MK$ est stable par $\varphi$ et $\Gal(L/F)$ ;
\item l'idéal déterminant de $\varphi : \MK \to \MK$ est 
l'idéal principal engendré par $u^e + p$ ;
\item pour tout $g \in \Gal(L/F)$ :
\begin{align*}
\tr(g \,|\, \MK / u \MK) & = \big(\eta(g) +\eta'(g)\big) \otimes 1 
\in S \otimes_{\Zp} W \\
\det(g \,|\, \MK / u \MK) & = \eta(g)\eta'(g) \otimes 1
\in S \otimes_{\Zp} W.
\end{align*}
\end{enumerate}
\end{definit}

\begin{rem}
Dans la définition \ref{defmoduleBK}, le 
produit tensoriel $S \otimes_{\Zp} \SK$ n'est \emph{pas} complété. Il
s'agit donc d'une variante algébrique --- par opposition à formelle ---
des plus traditionnels modules de Breuil--Kisin définis lorsque $S$
est une $\Zp$-algèbre locale noethérienne complète.
\end{rem}

\begin{rem}
Lorsque $R$ est une $W$-algèbre, tout réseau de Breuil--Kisin 
$\MK$ de type $(\vv,\ttt)$ de $S \otimes_R \bbM(\rho_R)$ admet une
décomposition similaire à \eqref{decompositionmodule} qui s'écrit :
\begin{equation}\label{decompositionmoduleBK}
\MK = \MK^{(0)} \oplus \MK^{(1)} \oplus \cdots \oplus \MK^{(f-1)}.
\end{equation}
Pour tout indice $i$ dans $\Z / f \Z$, le facteur $\MK^{(i)}$ est un 
module projectif de rang $2$ sur $S \otimes_{W,\iota{\circ}\varphi^{-i}} 
\SK$ et $\varphi$ envoie $\MK^{(i)}$ dans $\MK^{(i+1)}$.
\end{rem}

Nous pouvons à présent considérer le foncteur $L_{\rho_R}$ qui, à une 
$R$-algèbre $S$ associe l'ensemble des réseaux de type $(\vv,\ttt)$ dans 
$S \otimes_R \bbM(\rho_R)$. 
Nous avons alors un résultat de représentabilité.

\begin{prop}
\label{prop:repLVR}
\begin{enumerate}[(1)]
\item Le foncteur $L_{\rho_R}$ est représentable par un $R$-schéma 
projectif $\cL_{\rho_R}$.
\item Si $R'$ est une $R$-algèbre locale artinienne de corps résiduel 
$k_E$ et $\rho_{R'}$ est une $R'$-représentation qui étend $\rho_R$ 
alors il existe un isomorphisme canonique de $\Spec R' \times_{\Spec R}
\cL_{\rho_R}$ dans $\cL_{\rho_{R'}}$.
\item Le schéma $\cL_{\rho_R}$ est muni, fonctoriellement en $R$, d'un 
fibré en droite très ample canonique.
\end{enumerate}
\end{prop}

\begin{proof}
Il s'agit de l'analogue dans notre contexte de la proposition~1.3 de 
\cite{Ki4}. La démonstration est similaire, l'unique 
différence étant que nous devons justifier en outre que les conditions (iv) 
et (v) qui apparaissent dans la définition \ref{defmoduleBK} sont 
fermées. La 
condition (iv) est équivalente à la
conjonction des deux conditions suivantes :
\begin{itemize}
\item[(iv-a)] le conoyau de $\det (\text{id} \otimes \varphi) : 
\SK \otimes_\varphi \bigwedge^2 \MK \to \bigwedge^2 \MK$ est 
annulé par $u^e + p$ ;
\item[(iv-b)] $u^e + p$ appartient à l'idéal déterminant de $\varphi :
\MK \to \MK$.
\end{itemize}
La 
condition (iv-a) est fermée d'après 
\cite{Ki4}. Les conditions (iv-b) et (v) sont clairement fermées. 
\end{proof}

Nous faisons maintenant varier $R$ (et $\rho_R$) et appliquons la proposition \ref{prop:repLVR} aux algèbres 
artiniennes $R_n = R^\psi(\vv,\ttt, \rhobar)/\m^n$ où $\m$ désigne 
l'idéal maximal de $R^\psi(\vv,\ttt, \rhobar)$. La représentation
$\rho_{R_n}$ que nous considérons est le pullback de la déformation 
universelle $\rho^{\text{univ}} : G_F \to \GL_2(R^\psi(\rhobar))$ par le 
morphisme naturel $R^\psi(\rhobar) \to R_n$. La proposition 
\ref{prop:repLVR} fournit, pour tout entier $n$, un schéma 
projectif $\cL_{\rho_{R_n}}$ sur $\Spec R_n$ qui représente le foncteur
$L_{\rho_{R_n}}$. Il découle des énoncés de la proposition 
\ref{prop:repLVR} que ces schémas se recollent en un schéma formel :
\begin{equation}
\label{eq:GRhat}
\widehat{\GR}^\psi(\vv,\ttt,\rhobar) \longrightarrow \Spf 
R^\psi(\vv,\ttt,\rhobar).
\end{equation}
Il résulte de GAGA formel (voir corollaire 1.5.1 de \cite{Ki4}) que 
$\widehat{\GR}^\psi(\vv,\ttt,\rhobar)$ est algébrisable dans le sens 
où il s'identifie au complété formel d'un schéma $\GR^\psi(\vv,\ttt,
\rhobar)$ sur $\Spec R^\psi(\vv,\ttt,\rhobar)$ 
le long de sa fibre au-dessus de $\Spec k_E$. Par ailleurs, si $R$ est 
l'anneau des entiers d'une extension finie de $E$, il résulte du
théorème \ref{th:modBK} que la donnée d'un $R$-point de 
$\widehat{\GR}^\psi(\vv,\ttt,\rhobar)$ est équivalente à celle d'une 
déformation de $\rhobar$ à $R$ qui est potentiellement Barsotti--Tate de 
type $(\vv,\ttt)$. En copiant la démonstration de la proposition 1.6.4 
de \cite{Ki4}, nous déduisons le théorème suivant.

\begin{thm}
\label{th:isomgeneric}
Le morphisme structurel $\GR^\psi(\vv,\ttt,\rhobar) \to \Spec
R^\psi(\vv,\ttt,\rhobar)$ induit un isomorphisme après avoir
inversé $p$.
\end{thm}

\begin{definit}
 La \emph{variété de Kisin} $\vK^\psi(\vv,\ttt,\rhobar)$ associée à la donnée $(\vv,\ttt,\rhobar)$ est la fibre spéciale de 
$\widehat{\GR}^\psi(\vv,\ttt,\rhobar)$ :
\begin{align*}
\vK^\psi(\vv,\ttt,\rhobar) 
& = 
\Spec k_E \times_{\Spf R^\psi(\vv,\ttt,\rhobar)} \widehat{\GR}^\psi(\vv,\ttt,\rhobar) \\
& =
\Spec k_E \times_{\Spec R^\psi(\vv,\ttt,\rhobar)} \GR^\psi(\vv,\ttt,\rhobar).
\end{align*}
Il s'agit d'une variété projective sur $\Spec k_E$ qui s'identifie 
canoniquement à $\cL_{\rhobar}$.
\end{definit}

\section{Gènes et équations de la variété de Kisin}
\label{sec:Genes}

\def\ph{\vphantom{$A^0_0$}}
\definecolor{fond}{rgb}{0.95,0.95,0.85}

L'objectif de cette partie est d'énoncer une version précise du théorème 
\ref{thmintro:varKisin} de l'introduction qui fournit une description 
complète et explicite des variétés de Kisin $\vK^{\psi} (\vv, \ttt, 
\rhobar)$. Pour ce faire, nous commençons par introduire une donnée 
combinatoire associée au couple $(\ttt, \rhobar)$ que nous appelons le 
\emph{gène}.

Le paragraphe \S \ref{secrep} est consacré à la définition des gènes. Il inclut 
également plusieurs résultats combinatoires que 
nous serons amenés à utiliser fréquemment dans la suite. L'énoncé du 
théorème \ref{thequationsKisin} apparaît au \S \ref{secDetermination}. 
Il est suivi de plusieurs exemples qui donnent un aperçu complet des 
différentes situations qui peuvent se produire.
La démonstration du théorème \ref{thequationsKisin}, quant à elle, est 
reportée au \S \ref{sec:Demonstrations} alors que ses applications 
font l'objet du \S \ref{sec:Connexite}.

\subsection{Gènes}\label{secrep}

Posons :
$$\nu = \frac{e}{p-1} - 1 = p + p^2 + \cdots + p^{f-1}.$$
Rappelons que $\rhobar \simeq \Ind_{G_{F'}}^{G_F} \Big( \omega_{2f}^h \cdot 
\nr'(\theta)\Big)$ pour un certain entier $h$ défini modulo $p^{2f}-1
= q^2 - 1$ et qu'il existe des 
entiers $h_i$ ($0 \leq i \leq f-1$) dans l'intervalle $\llbracket 0, p-1 
\rrbracket$ uniquement déterminés tels que
\begin{equation}\label{defhachi}
h \equiv 1 + \sum_{i=0}^{f-1} h_i p^{f-1-i} \pmod{q+1}.
\end{equation}
Nous étendons la suite des $h_i$, dans un premier temps, à $\llbracket 
f, 2f-1 \rrbracket$ en posant $h_i = p - 1 - h_{i-f}$ pour tout entier 
$i$ dans cet intervalle puis, dans un second temps, à $\N$ tout entier 
par $(2f)$-périodicité.

La réduction modulo $\pi_E$ du type galoisien $\ttt=\eta\oplus\eta'$ 
est, quant à elle, encodée par deux éléments $\gamma, \gamma' \in 
\Z/e\Z$ définis par $\bar\eta =\tau_0^{\gamma}$ et $\bar\eta' = 
\tau_0^{\gamma'}$
où nous rappelons que $\tau_0$ désigne le plongement de $k_F$ dans
$k_E$ que nous avons fixé. 

\begin{definit}\label{defalphai}
Avec les notations précédentes, pour tout entier $i\geq 0$, nous 
définissons :
\begin{enumerate}[1)]
\item l'entier $\alpha_i$ comme l'unique élément de $\llbracket 
0, e-1\rrbracket$ vérifiant la congruence :
$$\alpha_i \equiv 
\left\lfloor \frac{p^i h}{q+1} \right\rfloor - p^i \gamma' \pmod e,$$
\item le symbole $X_i \in \{\gA, \gB, \gAB, \gO\}$
par :
$$\begin{array}{rcll}
X_i & = & \gA & \text{si } 
\alpha_i \in \big[0, \, \frac 1 p \: \nu + \epsilon_{i+f} \big[ 
\medskip \\
& = & \gAB & \text{si }
\alpha_i \in \big[\frac 1 p \: \nu + \epsilon_{i+f}, 
\, \frac {p-1} p \: \nu - \epsilon_i \big] \medskip \\
& = & \gB & \text{si }
\alpha_i \in \, \big]\frac {p-1} p \: \nu - 
\epsilon_i, \, \nu\big] \medskip \\
& = & \gO & \text{si }
\alpha_i \in \, \big]\nu, e \big[
\end{array}$$
où $\epsilon_i=\left\{\begin{array}{cl}1 &\mbox{si } h_i = p-1,\cr
 0&\mbox{sinon}.\cr
\end{array}\right.$
\end{enumerate}
La suite $(X_i)_{i \in \N}$ est appelée le \emph{gène} du triplet $(h, 
\gamma, \gamma')$.
\end{definit}

Lorsqu'une confusion est possible, nous précisons la dépendance 
des $\alpha_i$ et des $X_i$ vis-à-vis des paramètres $h$, $\gamma$ et 
$\gamma'$ en notant $\alpha_i(h, \gamma, \gamma')$ et $X_i(h, \gamma,
\gamma')$.

\begin{rem} \label{remgammagamma'}
Le type galoisien $\ttt=\eta\oplus\eta'$ satisfait 
$\det\rhobar_{|I_F}=(\bar{\eta}\bar{\eta}'\omega)_{|I_F}$. Ainsi
$\bar{\eta}_{|I_F}$ détermine $\bar{\eta}'_{|I_F}$. Ceci justifie la 
dissymétrie apparente dans les rôles de $\eta$ et $\eta'$ dans les 
notations introduites dans ce paragraphe. Nous renvoyons au lemme
\ref{lem:depchoix} pour de plus amples précisions.
\end{rem}

Le gène $(X_i)_{i \in \N}$ est une suite $(2f)$-périodique.
Dans la suite, nous le représentons systématiquement sur un ruban de 
Moebius comme le montre la figure \ref{fig:moebius}.
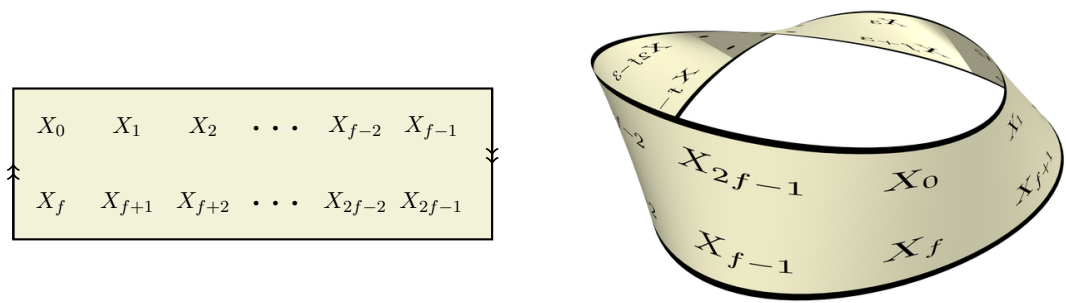
\begin{figure}
\raisebox{1.5cm}{
\begin{tikzpicture}
\draw[thick, fill=fond] (-0.5,-0.5) rectangle (5.8,1.5);
\draw[thick,->>] (-0.5,-0.5)--(-0.5,0.5);
\draw[thick,->>] (5.8,1.5)--(5.8,0.5);

\node[scale=0.8] at (0,1) { \ph $X_0$ };
\node[scale=0.8] at (1,1) { \ph $X_1$ };
\node[scale=0.8] at (2,1) { \ph $X_2$ };
\node[scale=1.5] at (3,1) { \ph $\cdots$ };
\node[scale=0.8] at (4,1) { \ph $X_{f-2}$ };
\node[scale=0.8] at (5,1) { \ph $X_{f-1}$ };

\node[scale=0.8] at (0,0) { \ph $X_f$ };
\node[scale=0.8] at (1,0) { \ph $X_{f+1}$ };
\node[scale=0.8] at (2,0) { \ph $X_{f+2}$ };
\node[scale=1.5] at (3,0) { \ph $\cdots$ };
\node[scale=0.8] at (4,0) { \ph $X_{2f-2}$ };
\node[scale=0.8] at (5,0) { \ph $X_{2f-1}$ };
\end{tikzpicture}
}
\hfill
\includegraphics[width=7cm]{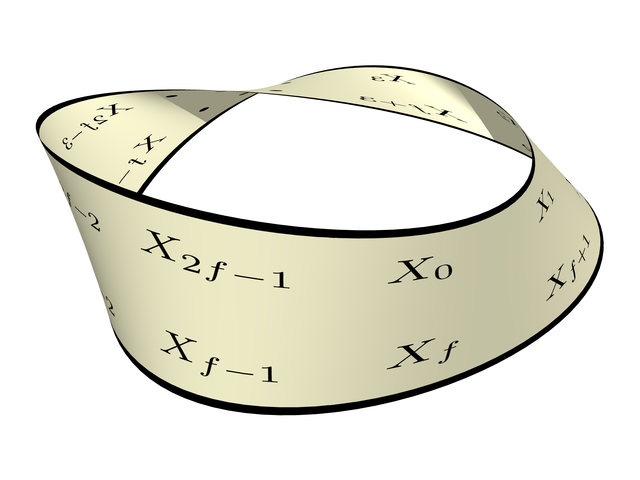}
\vspace{-0.8cm}
\caption{Représentation d'un gène sur un ruban de Moebius}
\label{fig:moebius}
\end{figure}
Ce faisant, les termes $X_i$ qui composent le gène restent écrits dans 
l'ordre les uns à la suite des autres et, de plus, la coordonnée $X_i$ 
\og tombe en face \fg\ de $X_{i+f}$ pour tout $i$. De même, lorsque nous 
considérons le couple, dit \emph{couple d'allèles}, $(X_i, X_{i+f})$, nous le 
notons $\binom{X_i}{X_{i+f}}$. Cette écriture a un intérêt pour notre 
propos car elle nous permet d'avoir une lecture immédiate des 
équations de la variété de Kisin à partir du gène correspondant (voir théorème \ref{thequationsKisin}).

Les suites $(\alpha_i)_{i\in\N}$ et $(X_i)_{i \in \N}$ sont $(2f)$-périodiques. Pour des raisons de commodité, nous 
les prolongeons à $\Z$ par $(2f)$-périodicité. Elles sont soumises à des 
propriétés combinatoires qui contraignent leur structure. Les lemmes 
suivants en sont des exemples.  Le lemme \ref{lem:relalphai} résulte 
d'un calcul direct.

\begin{lem}
\label{lem:relalphai}
Avec les notations précédentes, nous avons :
$$
\begin{array}{rcll}
\displaystyle
\left\lfloor \frac{p^i h}{q+1} \right\rfloor & \equiv & 
\displaystyle
\sum_{j=0}^{i-1}h_{i-j-1}p^j \pmod e
& \text{si } i\in\llbracket 0,f-1\rrbracket \medskip \\
& \equiv & \displaystyle
p^{i-f}-1+\sum_{j=i-f}^{f-1}h_{i-j-1}p^j \pmod e
& \text{si } i\in\llbracket f,2f-1\rrbracket.
\end{array}$$
De plus, $\alpha_{i+1} \equiv p \alpha_i + h_{i} \pmod e$ pour tout 
entier $i$ et cette congruence est une égalité si et seulement 
$X_i \in \{\gA, \gAB\}$. Enfin, dans le cas où $X_i = \gB$, nous
avons $\alpha_{i+1} = p \alpha_i + h_{i} - e$.
\end{lem}

\begin{lem}
\label{proprelgene}
En conservant les notations ci-dessus, le gène $X$ vérifie les 
propriétés suivantes :
\begin{enumerate}[i)]
\item Pour tout entier $i$, l'égalité $X_i = \gAB$ implique 
$X_{i+1} = \gO$.
\item Pour tout entier $i$, l'égalité $X_{i+1} = \gO$ implique 
$X_i\in\{\gO,\gAB\}$.
\end{enumerate}
\end{lem}

\begin{proof}
Supposons $X_i=\gAB$. Par définition, nous obtenons
$$\frac 1 p \: \nu + \epsilon_{i+f} \leq \alpha_i\leq 
 \frac {p-1} p \: \nu - \epsilon_i. $$
où $\epsilon_i$ vaut $1$ si $h_{i} = p-1$ et $0$ sinon. Du lemme 
\ref{lem:relalphai}, nous déduisons
$\nu + p\epsilon_{i+f} + h_i \leq \alpha_{i+1}\leq 
(p-1)\nu -p \epsilon_i+h_i$ et donc $\nu<\alpha_{i+1}\leq e-1$. Il en
résulte que $X_{i+1}=\gO$.

La démonstration de la propriété \emph{ii)} est analogue.
\end{proof}

\begin{lem}
\label{lem:ilyaO}
Nous conservons les notations ci-dessus et supposons en outre que
$\rhobar$ n'est pas dégénérée. Alors le gène $X$ 
contient au moins une occurrence du symbole $\gO$.
\end{lem}

\begin{proof}
Supposons que le gène $X$ ne contienne pas d'occurrence du symbole 
$\gO$. D'après le lemme \ref{proprelgene}, $X$ ne contient que des 
symboles $\gA$ et $\gB$. Ainsi, d'après le lemme \ref{lem:relalphai}, 
pour tout $i\in\llbracket 0,2f-1\rrbracket$, nous avons
$\alpha_{i+1}=p\alpha_i+h_i-e \1_B(X_i)$
où $\1_B$ désigne la fonction indicatrice de $B$.
En résolvant le système précédent, nous obtenons pour tout $i\in\llbracket0,f-1\rrbracket$,
$$\alpha_i=\frac{1}{p^{2f}-1}\sum_{j=0}^{2f-1}p^{2f-1-j}(e\1_B(X_{i+j})-h_{i+j}).$$
En regroupant les termes en $j$ et $j+f$ et en utilisant $h_{i+j+f}=p-1-h_{i+j}$, il vient :
\begin{align*}
\alpha_i 
 & =\frac{1}{p^{f}+1}\Big(-1+\sum_{j=0}^{f-1}p^{f-1-j}(p^f\1_B(X_{i+j})+\1_B(X_{i+j+f})-h_{i+j})\Big) \\
 & =\sum_{j=0}^{f-1}p^{f-1-j}\1_B(X_{i+j}) - \frac{1}{p^{f}+1}
    \underbrace{\Big(1+\sum_{j=0}^{f-1}p^{f-1-j}(\1_B(X_{i+j})-\1_B(X_{i+j+f})+h_{i+j})\Big)}_{A_i}.
\end{align*}
Comme $-1\leq \1_B(X_{i+j})-\1_B(X_{i+j+f})+h_{i+j}\leq p$, nous avons
$$-(p^f+1)<1-\frac{p^f-1}{p-1}\leq A_i\leq 1-\frac{p}{p-1}(p^f-1)<2(p^f+1)$$
et, par conséquent, $A_i\in\{0,p^f+1\}$.
Si $A_{i_0}=p^f+1$, alors
$$p^f-p^{f-1}h_{i_0}=\sum_{j=0}^{f-1}p^{f-1-j}\Big(\1_B(X_{i_0+j})-\1_B(X_{i_0+j+f})\Big)+\sum_{j=1}^{f-1}p^{f-1-j}h_{i_0+j},$$
d'où on déduit que $|p-h_{i_0}|\leq \frac{1}{p^f-1}\Big(\frac{p^f-1}{p-1}+p^{f-1}-1\Big)<2$ et $h_{i_0}=p-1$, ce qui est exclu.
De la même manière, si $A_{i_0}=0$, nous trouvons :
$$p^{f-1}h_{i_0}=-1-\sum_{j=0}^{f-1}p^{f-1-j}\Big(\1_B(X_{i_0+j+f})-\1_B(X_{i_0+j})\Big)-\sum_{j=1}^{f-1}p^{f-1-j}h_{i_0+j},$$
puis, en utilisant que $h_{i_0}<p-1$, $|h_{i_0}|\leq \frac{1}{p^f-1}\Big(\frac{p^f-1}{p-1}+p^{f-1}-p^{f-1}\Big)<1$ et $h_{i_0}=0$.
Dans tous les cas, nous aboutissons alors à une contradiction. 
Le gène $X$ contient donc au moins une occurrence du symbole $\gO$.
\end{proof}

\begin{lem}\label{pasABBA}
Le cas où $\{X_i, X_{i+f}\} = \{\gA, \gB\}$ pour tout 
$i\in\llbracket 0,f-1\rrbracket$ est exclu.
\end{lem}

\begin{proof}
De $\{X_i, X_{i+f}\} = \{\gA, \gB\}$ pour tout $i$, nous déduisons que :
$$|\alpha_i+\alpha_{f+i}-\nu|<\nu/p+1$$
pour tout $i$. De plus, nous avons
$\alpha_{i+1}+\alpha_{i+f+1}-\nu=p(\alpha_i+\alpha_{i+f}-\nu)$
d'après le lemme \ref{lem:relalphai}.
Par induction, nous obtenons alors :
$$|\alpha_i+\alpha_{i+f}-\nu| < \frac 1 {p^n} \cdot (\nu/p+1)$$ 
pour tout $i$ et tout $n\in\N^\star$. Donc $\alpha_{i+f}+\alpha_i=\nu$.
Un calcul élémentaire utilisant $\det\rhobar_{|I_F} = 
(\bar{\eta}\bar{\eta}'\omega)_{|I_F}$ nous amène alors à l'égalité
$\alpha_{i}(h,\gamma,\gamma') = \alpha_i(h,\gamma',\gamma)$ valable
pour tout $i$. Il en résulte que $\gamma = \gamma'$, et enfin
$\eta = \eta'$, ce que nous avons exclu.
\end{proof}

\subsubsection*{Transformations du gène par isomorphisme}

La défintion du gène ne dépend pas uniquement 
du couple $(\rhobar, \ttt)$ mais du triplet d'entiers $(h, \gamma, 
\gamma')$. Or la donnée de $(h, \gamma, \gamma')$ n'est pas équivalente 
à celle de $(\rhobar, \ttt)$. En effet, le 
premier détermine le second mais la réciproque n'est pas valable. 
Précisément, changer $h$ en $qh$ ou échanger $\gamma$ et $\gamma'$ sont 
deux modifications (involutives) qui laissent inchangés $\rhobar$ et 
$\ttt$. Le lemme \ref{lem:depchoix} précise comment se comporte 
le gène suite à ces modifications.

\begin{lem}
\label{lem:depchoix}
Soit $\tau$ la transposition de l'ensemble $\{\gA, \gB, \gAB, \gO\}$ qui
échange $\gA$ et $\gB$. Alors, nous avons :
\begin{enumerate}[i)]
\item $X_i(h, \gamma, \gamma') = X_{i+f}(qh, \gamma, \gamma')$
\item $X_i(h, \gamma', \gamma) = \tau\big(X_{i+f}(h, \gamma, 
\gamma')\big)$
\end{enumerate}
pour tout entier $i$.
\end{lem}

\begin{proof}
\emph{i)} s'obtient en remarquant que 
$$qh\equiv -h\equiv p^f- \sum_{i=0}^{f-1}h_{i}p^{f-1-i}\equiv 1+\sum_{i=0}^{f-1}(p-1-h_{i})p^{f-1-i} \pmod{q+1}.$$
L'hypothèse sur le type galoisien $\det\rhobar_{|I_F}=\omega_f^h=(\bar{\eta}\bar{\eta}'\omega)_{|I_F}$ impose la congruence
 $\alpha_{i}(h,\gamma,\gamma')+\alpha_{i+f}(h,\gamma',\gamma)\equiv \nu \pmod e$. 
De là découle \emph{ii)}.
\end{proof}

Un autre choix que nous avons fait est celui du plongement $\tau_0$. Si 
nous en choisissons un autre, disons $\sigma_0$, il existe un entier $n$ 
tel que $\sigma_0 = \tau_0 \circ \varphi^{-n}$ où $\varphi$ est le 
Frobenius arithmétique. Un calcul simple montre alors que, si $(h, 
\gamma, \gamma')$ encode le couple $(\rhobar, \ttt)$ par rapport au 
plongement $\tau_0$, alors ce même couple est encodé par $(p^n h, p^n 
\gamma, p^n \gamma')$ relativement au plongement $\sigma_0$. Nous en 
déduisons que le gène relatif à $\sigma_0$ s'obtient à partir du gène 
relatif à $\tau_0$ simplement en décalant les indices de $n$.

En conclusion, nous pouvons dire que le gène lui-même ne dépend pas 
uniquement du couple $(\rhobar, \ttt)$ mais également de quelques choix 
supplémentaires. Toutefois, cette dépendance est facile à cerner et
n'influence aucunement les équations de la variété de Kisin que nous
obtenons (voir théorème \ref{thequationsKisin}). 

\subsubsection*{Décoration des gènes}

À partir d'un gène, nous définissons des données supplémen\-taires, que 
nous appelons \emph{décorations}, qui sont utiles pour lire les 
équations de la variété de Kisin.

\begin{definit} \label{defdominance}
Soit $X$ un gène pour lequel il existe $i_0\in\llbracket 0,f-1\rrbracket$ 
avec 
$$\binom{X_{i_0}}{X_{i_0+f}}\not\in\left\{\binom{\gO}{\gO},\binom{\gO}{\gAB},\binom{\gAB}{\gO},\binom{\gAB}{\gAB},\binom{\gA}{\gB},\binom{\gB}{\gA}\right\}.$$
Pour $i$ prenant successivement les valeurs $i_0,i_0-1\ldots,0, f-1,
\ldots,i_0+1$ :

\smallskip

\begin{itemize}
\item[$\bullet$] si $\binom{X_i}{X_{i+f}} \in \left\{\binom{\gA}{\gAB},
\binom{\gAB}{\gA}, \binom{\gA}{\gA}, \binom{\gA}{\gO},
\binom{\gO}{\gA}\right\}$, nous disons que $\gA$ est \emph{dominant} 
en $i$ dans $X$, \smallskip
\item[$\bullet$] si $\binom{X_i}{X_{i+f}} \in \left\{\binom{\gB}{\gAB},
\binom{\gAB}{\gB}, \binom{\gB}{\gB}, \binom{\gB}{\gO},
\binom{\gO}{\gB}\right\}$, nous disons que $\gB$ est \emph{dominant}
en $i$ dans $X$, \smallskip
\item[$\bullet$] sinon, si $\gA$ (resp. $\gB$) est dominant dans $(i+1)$ dans $X$, 
nous disons que $\gA$ (resp. $\gB$) est \emph{dominant} en $i$ dans $X$.
\end{itemize}
\end{definit}

Quand l'hypothèse de la définition \ref{defdominance} est vérifiée, nous 
disons que $X$ a un caractère dominant. Il se peut que le gène associé à 
un couple $(\rhobar, \ttt)$ n'ait pas de caractère dominant ; toutefois, 
en vertu des lemmes \ref{proprelgene} et \ref{pasABBA}, ces cas sont 
très circonscrits. Plus précisément, le gène $X(h, \gamma, \gamma')$ n'a 
pas de caractère dominant uniquement dans les deux cas suivants :
\begin{itemize}
\item[$\bullet$] soit, il contient un couple d'allèles $\binom \gO \gO$
\item[$\bullet$] soit, quitte à échanger les rôles de $\eta$ et $\eta'$, il est de 
la forme :

\smallskip

\begin{center}
\begin{tikzpicture}[xscale=0.8,yscale=0.7]
\draw [fill=fond, thick] (-0.5,-0.5) rectangle (6,1.5);
\draw[thick,->>] (-0.5,-0.5)--(-0.5,0.5);
\draw[thick,->>] (6,1.5)--(6,0.5);
\node at (0, 0) { $\gAB$ };
\node at (1, 0) { $\gO$ };
\node at (2, 0) { $\gAB$ };
\node[scale=1.5] at (3.25,0) { \ph$\cdots$ };
\node at (4.5, 0) { $\gO$ };
\node at (5.5, 0) { $\gAB$ };
\node at (0, 1) { $\gO$ };
\node at (1, 1) { $\gAB$ };
\node at (2, 1) { $\gO$ };
\node[scale=1.5] at (3.25,1) { \ph$\cdots$ };
\node at (4.5, 1) { $\gAB$ };
\node at (5.5, 1) { $\gO$ };
\end{tikzpicture}
\end{center}

\noindent (et donc, en particulier, il est de longueur impaire).
\end{itemize}
\smallskip
Pour démontrer cette alternative, supposons que $X(h, \gamma, \gamma')$ ne
contienne pas de couple d'allèles $\binom \gO \gO$. Alors, d'après le
lemme \ref{proprelgene}, il ne contient pas non plus le couple d'allèles
$\binom \gAB \gAB$. Ainsi, quitte à translater, il contient le couple
$\binom \gO \gAB$. Mais, étant donné que $\gAB$ est nécessairement suivi
de $\gO$ (lemme \ref{proprelgene}), cela impose que le couple 
d'allèles suivant est $\binom \gAB \gO$. En répétant l'argument, nous
aboutissons à la conclusion annoncée.

\begin{definit}\label{defdecoration}
Soit $X$ un gène ayant un caractère dominant.
Pour un entier $i\in\llbracket 0, f-1\rrbracket$ tel que $Y \in \{\gA, 
\gB\}$ soit dominant dans $\binom{X_i}{X_{i+f}}$ et 
$\binom{X_{i+1}}{X_{i+f+1}}$ :
\begin{itemize}
\item[$\bullet$] nous relions $X_i$ et $X_{i+1+f}$ si $X_i=Y$, et
\item[$\bullet$] nous relions $X_{i+f}$ et $X_{i+1}$ si $X_{i+f}=Y$.
\end{itemize}

\smallskip

\noindent
La \emph{décoration} de $X$ est l'ensemble des liens ci-dessus.
\end{definit}

\begin{ex}\label{exMoebius}
Nous considèrons $f = 11$ (\emph{i.e.} $F = \Q_{p^{11}}$) et 
les paramètres 
$$\begin{array}{lcl}
\rhobar = \Ind_{G_{F'}}^{G_F}(\omega_{22}^h) & \text{avec} & h = \frac{p-1} 2 + (p-1) p^3 + (p-1) p^5 + p^8 + p^9 + p^{10} \medskip \\
\bar\eta = \tau_0^{\gamma} & \text{avec} & \gamma = -\frac{p+3} 2 - p^2 - 2p^3 - p^5 \medskip \\
\bar\eta' = \tau_0^{\gamma'} & \text{avec} & \gamma' = -p^5 - p^7 
\end{array}$$
Pour ces données, les valeurs prises par les $\alpha_i$ ainsi 
que les \emph{allèles} $X_i$ correspondants sont représentés dans le 
tableau de la figure \ref{fig:tableexemple}.
\begin{figure}
\renewcommand{\arraystretch}{1.3}
$$\scriptsize
\begin{array}{|c|c|c|c|c|c|c|}
\cline{1-3} \cline{5-7}
i & \alpha_i & X_i &&
i & \alpha_i & X_i \\
\cline{1-3} \cline{5-7}
0 & p^5 + p^7 & \gA &&
11 & \frac{p-3} 2 + (p-1)p^3 + p^6 + p^7 + p^8 + p^9 + p^{10} & \gB \\
1 & 1 + p^6 + p^8 & \gA &&
12 & (p-1) + \frac{p-3}2 \: p + (p-1)p^4 + p^7 + p^8 + p^9 + p^{10} & \gB \\
2 & 1 + p + p^7 + p^9 & \gA &&
13 & (p^2-1) + \frac{p-3}2 \: p^2 + (p-1)p^5 + p^8 + p^9 + p^{10} & \gB \\
3 & 1 + p + p^2 + p^8 + p^{10} & \gB &&
14 & (p^3-1) + \frac{p-3}2 \: p^3 + (p-1)p^6 + p^9 + p^{10} & \gB \\
4 & 1 + p + p^2 + p^3 + p^9 & \gA &&
15 & \frac{p-1} 2 \: p^4 + (p-1) p^7 + p^{10} & \gB \\
5 & p + p^2 + p^3 + p^4 + p^{10} & \gB &&
16 & p + \frac{p-1} 2 \: p^5 + (p-1) p^8 & \gA \\
6 & p + p^2 + p^3 + p^4 + p^5 & \gA &&
17 & p^2 + \frac{p-1} 2 \: p^6 + (p-1) p^9 & \gAB \\
7 & p^2 + p^3 + p^4 + p^5 + p^6 & \gA &&
18 & (p-1) + p^3 + \frac{p-1} 2 \: p^7 + (p-1) p^{10} & \gO \\
8 & (p-1) + p^3 + p^4 + p^5 + p^6 + p^7 & \gA &&
19 & (p^2-1) + p^4 + \frac{p-1} 2 \: p^8 & \gA \\
9 & (p-1)p + p^4 + p^5 + p^6 + p^7 + p^8 & \gA &&
20 & (p^3-1) + p^5 + \frac{p-1} 2 \: p^9 & \gAB \\
10 & (p-1)p^2 + p^5 + p^6 + p^7 + p^8 + p^9 & \gA &&
21 & (p^4-1) + p^6 + \frac{p-1} 2 \: p^{10} & \gO \\
\cline{1-3} \cline{5-7}
\end{array}$$
\renewcommand{\arraystretch}{1}
\caption{Valeurs des $\alpha_i$ et des $X_i$ dans le cas
de l'exemple \ref{exMoebius}}
\label{fig:tableexemple}
\end{figure}
En revenant aux définitions, nous en déduisons 
le gène associé ainsi que sa décoration. Voici comment il se représente 
sur un ruban de Moebius :

\medskip

\begin{center}
\begin{tikzpicture}[scale=0.8]
\draw [fill=fond, thick] (-0.5,-0.5) rectangle (10.5,1.5);
\draw[thick,->>] (-0.5,-0.5)--(-0.5,0.5);
\draw[thick,->>] (10.5,1.5)--(10.5,0.5);
\node[color=colB] at (0, 0) { $\gB$ };
\node[color=colB] at (0, 1) { $\gA$ };
\node[color=colB] at (1, 0) { $\gB$ };
\node[color=colB] at (1, 1) { $\gA$ };
\node[color=colB] at (2, 0) { $\gB$ };
\node[color=colB] at (2, 1) { $\gA$ };
\node[color=colB] at (3, 0) { $\gB$ };
\node[color=colB] at (3, 1) { $\gB$ };
\node[color=colA] at (4, 0) { $\gB$ };
\node[color=colA] at (4, 1) { $\gA$ };
\node[color=colA] at (5, 0) { $\gA$ };
\node[color=colA] at (5, 1) { $\gB$ };
\node[color=colA] at (6, 0) { $\gAB$ };
\node[color=colA] at (6, 1) { $\gA$ };
\node[color=colA] at (7, 0) { $\gO$ };
\node[color=colA] at (7, 1) { $\gA$ };
\node[color=colA] at (8, 0) { $\gA$ };
\node[color=colA] at (8, 1) { $\gA$ };
\node[color=colA] at (9, 0) { $\gAB$ };
\node[color=colA] at (9, 1) { $\gA$ };
\node[color=colA] at (10, 0) { $\gO$ };
\node[color=colA] at (10, 1) { $\gA$ };
\draw[color=colB,thick] (0.2,0.2)--(0.8,0.8);
\draw[color=colB,thick] (1.2,0.2)--(1.8,0.8);
\draw[color=colB,thick] (2.2,0.2)--(2.8,0.8);
\draw[color=colA,thick] (4.2,0.8)--(4.8,0.2);
\draw[color=colA,thick] (5.2,0.2)--(5.8,0.8);
\draw[color=colA,thick] (6.2,0.8)--(6.8,0.2);
\draw[color=colA,thick] (7.2,0.8)--(7.8,0.2);
\draw[color=colA,thick] (8.2,0.8)--(8.8,0.2);
\draw[color=colA,thick] (8.2,0.2)--(8.8,0.8);
\draw[color=colA,thick] (9.2,0.8)--(9.8,0.2);
\end{tikzpicture}
\end{center}

\noindent
Dans la représentation ci-dessus, les colonnes dans lesquelles $\gA$ 
(resp. $\gB$) est dominant ont été dessinées en bleu (resp. en vert). 
Conformément à la définition, les liens s'obtiennent en faisant partir 
des segments en diagonale vers la droite à partir des $\gA$ bleus et des 
$\gB$ verts, pourvu que la colonne située à droite soit de la même 
couleur. Nous insistons en particulier sur le fait qu'en aucun cas, un 
lien ne peut relier deux colonnes de couleurs différentes.
\end{ex}

\subsection{Détermination de la variété de Kisin} \label{secDetermination}

Nous énonçons les résultats de détermination de la variété de Kisin 
$\vK^{\psi}(\vv, \ttt, \rhobar)$ en termes du gène décoré $X$. Les 
preuves font l'objet de la partie suivante (\S 
\ref{sec:Demonstrations}).

\begin{thm}
\label{thequationsKisin}
La variété de Kisin $\vK^{\psi}(\vv, \ttt, \rhobar)$ est isomorphe
à une sous-variété fermée $\mathcal K(\ttt, \rhobar)$ de 
$\prod_{i=0}^{f-1} \P^1_{k_E}$.

Soit $X$ un gène associé au couple $(\rhobar,\ttt)$. En notant $[x_i : 
x_{i+f}]$ les coordonnées homogènes sur le $i$-ième facteur {\rm (}$0 
\leq i < f${\rm )} et en convenant que $x_i = x_{i \text{\rm\ mod } 2f}$ 
pour $i$ entier, la variété $\mathcal K(\ttt, \rhobar)$ est définie par 
les équations suivantes :

\medskip

\begin{itemize}
\item[(A)] $x_i = 0$ pour tout $i \in \llbracket 0, 2f-1 \rrbracket$ tel
que $X_i = \gO$

\medskip

\item[(B1)] $x_i x_{i+f+1} = x_{i+1} x_{i+f}$ {\rm (}\emph{i.e.} $[x_i : x_{i+f}]
= [x_{i+1} : x_{i+f+1}]${\rm )} chaque fois que nous observons,
dans $X$, le motif

\smallskip

\begin{center}
\begin{tikzpicture}[xscale=1.8,yscale=0.8]
\fill[color=fond] (-0.5,-0.5) rectangle (1.5,1.5);
\shade[left color=white, right color=fond] (-1,-0.5) rectangle (-0.5,1.5);
\shade[left color=fond, right color=white] (1.5,-0.5) rectangle (2,1.5);
\fill (-0.5,-0.48) rectangle (1.5,-0.52);
\shade[left color=black, right color=white] (1.5,-0.48) rectangle (2,-0.52);
\shade[left color=white, right color=black] (-1,-0.48) rectangle (-0.5,-0.52);
\fill (-0.5,1.48) rectangle (1.5,1.52);
\shade[left color=black, right color=white] (1.5,1.48) rectangle (2,1.52);
\shade[left color=white, right color=black] (-1,1.48) rectangle (-0.5,1.52);
\node at (0,1) { \ph$X_{i}$ };
\node at (1,1) { \ph$X_{i+1}$ };
\node at (0,0) { \ph$X_{i+f}$ };
\node at (1,0) { \ph$X_{i+f+1}$ };
\draw[thick] (0.25,0.7)--(0.65,0.2);
\draw[thick] (0.25,0.2)--(0.65,0.7);
\end{tikzpicture}
\end{center}

\medskip

\item[(B2)] $x_i x_{i+f+1} = 0$ chaque fois que nous observons, dans $X$, le motif

\smallskip

\begin{center}
\begin{tikzpicture}[xscale=1.8,yscale=0.8]
\fill[color=fond] (-0.5,-0.5) rectangle (1.5,1.5);
\shade[left color=white, right color=fond] (-1,-0.5) rectangle (-0.5,1.5);
\shade[left color=fond, right color=white] (1.5,-0.5) rectangle (2,1.5);
\fill (-0.5,-0.48) rectangle (1.5,-0.52);
\shade[left color=black, right color=white] (1.5,-0.48) rectangle (2,-0.52);
\shade[left color=white, right color=black] (-1,-0.48) rectangle (-0.5,-0.52);
\fill (-0.5,1.48) rectangle (1.5,1.52);
\shade[left color=black, right color=white] (1.5,1.48) rectangle (2,1.52);
\shade[left color=white, right color=black] (-1,1.48) rectangle (-0.5,1.52);
\node at (0,1) { \ph$X_{i}$ };
\node at (1,1) { \ph$X_{i+1}$ };
\node at (0,0) { \ph$X_{i+f}$ };
\node at (1,0) { \ph$X_{i+f+1}$ };
\draw[thick] (0.25,0.7)--(0.65,0.2);
\end{tikzpicture}
\end{center}

\medskip
\item[(B3)] $x_{i+1} x_{i+f} = 0$ chaque fois que nous observons, dans $X$, le motif

\smallskip

\begin{center}
\begin{tikzpicture}[xscale=1.8,yscale=0.8]
\fill[color=fond] (-0.5,-0.5) rectangle (1.5,1.5);
\shade[left color=white, right color=fond] (-1,-0.5) rectangle (-0.5,1.5);
\shade[left color=fond, right color=white] (1.5,-0.5) rectangle (2,1.5);
\fill (-0.5,-0.48) rectangle (1.5,-0.52);
\shade[left color=black, right color=white] (1.5,-0.48) rectangle (2,-0.52);
\shade[left color=white, right color=black] (-1,-0.48) rectangle (-0.5,-0.52);
\fill (-0.5,1.48) rectangle (1.5,1.52);
\shade[left color=black, right color=white] (1.5,1.48) rectangle (2,1.52);
\shade[left color=white, right color=black] (-1,1.48) rectangle (-0.5,1.52);
\node at (0,1) { \ph$X_{i}$ };
\node at (1,1) { \ph$X_{i+1}$ };
\node at (0,0) { \ph$X_{i+f}$ };
\node at (1,0) { \ph$X_{i+f+1}$ };
\draw[thick] (0.25,0.2)--(0.65,0.7);
\end{tikzpicture}
\end{center}

\end{itemize}
\end{thm}

\begin{rem}
L'écriture dans le jeu d'équations \emph{(B2)} (resp. le jeu d'équations 
\emph{(B3)}) sous-entend qu'il n'y a pas de lien entre $X_{i+1}$ et 
$X_{i+f}$ (resp. entre $X_i$ et $X_{i+f+1})$. Ainsi, pour un $i$ donné,
les cas \emph{(B1)}, \emph{(B2)} et \emph{(B3)} s'excluent mutuellement.
Nous remarquons de plus que les jeux d'équations \emph{(B2)} et \emph{(B3)} 
sont équivalents et se déduisent l'un de l'autre par la transformation 
$i \mapsto i+f$. Nous avons toutefois préféré mentionner explicitement 
cette répétition pour plus de clarté.
\end{rem}

En particulier, nous remarquons que la variété de Kisin 
$\vK^{\psi}(\vv, \ttt, \rhobar)$ est vide dès lors qu'un gène
associé à $(\ttt, \rhobar)$ contient le couple d'allèles $\binom
\gO \gO$. Nous verrons dans la suite (\emph{cf} proposition
\ref{corvide}) que la réciproque est également vraie.

Notons que l'isomorphisme $\mathcal K(\ttt, \rhobar) \to \vK^\psi(\vv, \ttt, 
\rhobar)$ est entièrement explicite\footnote{Nous renvoyons au \S 
\ref{subsec:Ktrhobar} pour sa construction.} ; autrement dit, nous 
pouvons également lire sur le gène, la structure du réseau de 
Breuil--Kisin correspondant à chacun des points de la variété de Kisin.
Le théorème \ref{thequationsKisin} permet également d'établir des 
propriétés fines sur la géométrie de la variété de Kisin ; ces aspects
sont discutés au \S \ref{sec:Geometrie}.

\subsection{Premiers exemples}
\label{quelquesexemples}

Dans ce paragraphe, nous presentons quelques exemples d'utilisation du 
théorème \ref{thequationsKisin} de détermination de la variété de Kisin et en profitons pour illustrer 
différentes propriétés géométriques des variétés de Kisin (\S 
\ref{exchaineP1}, \S \ref{subsubnonequi}, \S \ref{sssec:filrouge}). 
Ensuite (\S \ref{subsubgenerique}), sous l'hypothèse de généricité de la 
représentation $\rhobar$ (au sens de \cite{BM}), nous montrons que la 
variété de Kisin est toujours réduite à un point.

\subsubsection{Petites valeurs de $f$}

Pour les petites valeurs de $f$, il est possible de lister tous les 
gènes possibles (à symétrie près) ne contenant pas le couple d'allèles 
$\binom \gO \gO$ et satisfaisant aux conditions des lemmes 
\ref{proprelgene} et \ref{pasABBA} puis, pour chacun d'eux, de 
déterminer la variété de Kisin correspondante à l'aide du théorème 
\ref{thequationsKisin}.

Nous avons fait l'exercice pour $f=2$ et $f=3$ et avons reporté les 
résultats obtenus respectivement dans les tableaux de la figure 
\ref{fig:f2} (page \pageref{fig:f2}) et de la figure \ref{fig:f3} (page 
\pageref{fig:f3}). Pour $f=2$, nous avons indiqué en outre un ou 
plusieurs couples $(\ttt, \rhobar)$ qui conduisent à chaque gène 
présenté. Les résultats auxquels nous aboutissons sont en accord avec 
ceux de \cite{CDM} qui avaient été obtenus par des méthodes différentes 
qui ne s'étendent \emph{a priori} pas à $f > 2$.
Signalons enfin que, dans le cas $f=3$, nous avons fait apparaître la stratification par le genre sur les variétés de Kisin ; 
celle-ci est définie et étudiée au \S \ref{sec:Stratification}.

\renewcommand{\arraystretch}{1.5}
\begin{figure}
\begin{center}
\small
\begin{tabular}{|>{\centering}p{4.2cm}|>{\centering}p{2.9cm}|>{\centering}p{2.3cm}|p{2.9cm}|}
\hline
   $\rhobar = \Ind_{G_{F'}}^{G_F} (\omega_4^{1+r_0} \nr' (\theta) )$&Type $\ttt = \eta \oplus \eta'$  & Gène & \hfill{}Variété de Kisin\hfill\null \\
\hline
&&&\\
\vspace{-0.7cm} $r_0=0$&
\vspace{-0.7cm} $\omega_2^{r_0} \oplus \omega_2^{-p}$ & 
\vspace{-0.9cm}
\begin{tikzpicture}[xscale=0.8,yscale=0.5]
\draw [fill=fond, thick] (-0.5,-0.5) rectangle (1.5,1.5);
\node[color=colB] at (0, 0) { $\gB$ };
\node[color=colB] at (0, 1) { $\gB$ };
\node[color=colB] at (1, 0) { $\gB$ };
\node[color=colB] at (1, 1) { $\gA$ };
\draw[color=colB,thick] (-0.8,0.8)--(-0.2,0.2);
\draw[color=colB,thick] (0.2,0.8)--(0.8,0.2);
\draw[color=colB,thick] (0.2,0.2)--(0.8,0.8);
\draw[color=colB,thick] (1.2,0.2)--(1.8,0.8);
\end{tikzpicture}
& \vspace{-0.7cm}\hfill{}$[1:0]\times[1:0]$\hfill\null \\
\hline
&&&\\
\vspace{-0.7cm} $1 \leq r_0 \leq p-2$ & 
\vspace{-0.7cm} $ \omega_2^{r_0} \oplus \omega_2^{-p}$ & 
\vspace{-0.9cm}
\begin{tikzpicture}[xscale=0.8,yscale=0.5]
\draw [fill=fond, thick] (-0.5,-0.5) rectangle (1.5,1.5);
\node[color=colB] at (0, 0) { $\gO$ };
\node[color=colB] at (0, 1) { $\gB$ };
\node[color=colB] at (1, 0) { $\gO$ };
\node[color=colB] at (1, 1) { $\gAB$ };
\draw[color=colB,thick] (0.2,0.8)--(0.8,0.2);
\end{tikzpicture}
& \vspace{-0.7cm}\hfill{}$[1:0]\times[1:0]$\hfill\null\\
\hline
&&&\\
\vspace{-0.7cm} $1 \leq r_0 \leq p-2$ &
\vspace{-0.7cm} $\omega_4^{r_0 - p} \oplus \1$  & 
\vspace{-0.9cm}
\begin{tikzpicture}[xscale=0.8,yscale=0.5]
\draw [fill=fond, thick] (-0.5,-0.5) rectangle (1.5,1.5);
\node[color=colA] at (0, 0) { $\gAB$ };
\node[color=colA] at (0, 1) { $\gA$ };
\node[color=colA] at (1, 0) { $\gO$ };
\node[color=colA] at (1, 1) { $\gA$ };
\draw[color=colA,thick] (-0.8,0.2)--(-0.2,0.8);
\draw[color=colA,thick] (0.2,0.8)--(0.8,0.2);
\draw[color=colA,thick] (1.2,0.8)--(1.8,0.2);
\end{tikzpicture}
& \vspace{-0.7cm}\hfill{}$[0:1]\times[1:0]$\hfill\null \\
\hline
&&&\\
\vspace{-0.7cm} $0 \leq r_0 \leq p-3$ & 
\vspace{-0.7cm} $ \omega_2^{1 + r_0 - p} \oplus \omega_2^{-1}$ & 
\vspace{-0.9cm}
\begin{tikzpicture}[xscale=0.8,yscale=0.5]
\draw [fill=fond, thick] (-0.5,-0.5) rectangle (1.5,1.5);
\node[color=colA] at (0, 0) { $\gAB$ };
\node[color=colA] at (0, 1) { $\gA$ };
\node[color=colB] at (1, 0) { $\gO$ };
\node[color=colB] at (1, 1) { $\gB$ };
\end{tikzpicture}
& \vspace{-0.7cm}\hfill{}$\P^1_{k_E}\times[1:0]$\hfill\null \\
\hline
\end{tabular}\end{center}
\caption{Variété de Kisin de la représentation  $\rhobar = \Ind_{G_{F'}}^{G_F} (\omega_4^{1+r_0}\cdot \nr'(\theta)). $}
\label{fig:f2}
\end{figure}
\renewcommand{\arraystretch}{1}

\subsubsection{Chaînes de $\P^1$}
\label{exchaineP1}

Considérons le système d'équations donné par le fragment de diagramme de 
longueur $\ell+1$ (avec $\ell \leq f-1$) représenté ci-dessous.

\begin{center}
\begin{tikzpicture}[xscale=1.4,yscale=0.8]
\fill[color=fond] (-0.5,-0.5) rectangle (6.3,1.5);
\shade[left color=white, right color=fond] (-1,-0.5) rectangle (-0.5,1.5);
\shade[left color=fond, right color=white] (6.3,-0.5) rectangle (6.8,1.5);
\fill (-0.5,-0.48) rectangle (6.3,-0.52);
\shade[left color=black, right color=white] (6.3,-0.48) rectangle (6.8,-0.52);
\shade[left color=white, right color=black] (-1,-0.48) rectangle (-0.5,-0.52);
\fill (-0.5,1.48) rectangle (6.3,1.52);
\shade[left color=black, right color=white] (6.3,1.48) rectangle (6.8,1.52);
\shade[left color=white, right color=black] (-1,1.48) rectangle (-0.5,1.52);
\node at (0,1) { \ph$x_{i}$ };
\node at (1,1) { \ph$x_{i+1}$ };
\node at (2,1) { \ph$x_{i+2}$ };
\node[scale=1.5] at (3.25,1) { \ph$\cdots$ };
\node at (4.5,1) { \ph$x_{i+\ell-1}$ };
\node at (5.7,1) { \ph$x_{i+\ell}$ };
\node at (0,0) { \ph$x_{i+f}$ };
\node at (1,0) { \ph$x_{i+f+1}$ };
\node at (2,0) { \ph$x_{i+f+2}$ };
\node[scale=1.5] at (3.25,0) { \ph$\cdots$ };
\node at (4.5,0) { \ph$x_{i+\ell+f-1}$ };
\node at (5.7,0) { \ph$x_{i+\ell+f}$ };
\draw[thick] (0.25,0.2)--(0.75,0.7);
\draw[thick] (1.25,0.2)--(1.75,0.7);
\draw[thick] (2.25,0.2)--(2.75,0.7);
\draw[thick] (3.75,0.2)--(4.25,0.7);
\draw[thick] (4.85,0.2)--(5.35,0.7);
\end{tikzpicture}
\end{center}

Les équations correspondantes sont :
\begin{equation}
\label{eq:chaine}
x_{i+j+1} x_{i+j+f} = 0 \quad 
\text{pour } j \in \llbracket 0, \ell \rrbracket
\end{equation}
Ainsi si, pour un certain $j_0$, la coordonnée $x_{i+j_0}$ ne s'annule 
pas, nous déduisons que $x_{i+f+j_0-1} = 0$ et par suite que le point 
projectif $[x_{i+j_0-1} : x_{i+f+j_0-1}]$ est égal à $[1:0]$. De proche 
en proche, nous obtenons alors $[x_{i+j} : x_{i+j+f}] = [1:0]$ pour tout $j 
\in \llbracket 0, j_0 - 1 \rrbracket$. De manière similaire, si la 
coordonnée $x_{i+j_0+f}$ est non nulle, alors nous avons $[x_{i+j} : 
x_{i+j+f}] = [0:1]$ pour tout $j \in \llbracket j_0+1, \ell \rrbracket$. 
Il résulte de ceci que les solutions du système \eqref{eq:chaine} sont 
les $\ell$-uplets $([x_{i+j} : x_{i+j+f}])_{1 \leq j \leq \ell}$ de 
points projectifs de la forme :
$$[1:0], \, \ldots, \,[1:0], \, [x_{i+j_0} : x_{i+j_0+f}], \,
[0:1], \, \ldots, \, [0:1].$$
Géométriquement cela correspond à une chaîne de $\P^1_{k_E}$ de longueur
$\ell+1$.

\subsubsection{Un cas non équidimensionnel}
\label{subsubnonequi}

L'exemple précédent montre que les variétés de Kisin ne sont pas 
irréductibles en général. Elles peuvent même ne pas être 
équidimensionnelles. L'exemple le plus simple de ce phénomène est donné
par le diagramme suivant ($f=4$):

\begin{center}
\begin{tikzpicture}[xscale=1.3,yscale=0.8]
\draw [thick,fill=fond] (-0.6,-0.5) rectangle (3.6,1.5);
\draw[thick,->>] (-0.6,-0.5)--(-0.6,0.5);
\draw[thick,->>] (3.6,1.5)--(3.6,0.5);
\node at (0,1) { \ph$x_0$ };
\node at (1,1) { \ph$x_1$ };
\node at (2,1) { \ph$x_2$ };
\node at (3,1){\ph$x_3$};
\node at (0,0) { \ph$x_4$ };
\node at (1,0) { \ph$x_5$ };
\node at (2,0) { \ph$x_7$ };
\node at (3,0) {\ph$0$};
\draw[thick] (0.25,0.7)--(0.75,0.2);
\draw[thick] (1.25,0.2)--(1.75,0.7);
\end{tikzpicture}
\end{center}

\noindent
La variété de Kisin correspondante est l'union de $\{[0:1]\} \times \P^1_{k_E}
\times\{[0:1]\}\times \{[1:0]\}$ et $\P^1_{k_E}\times\{[1:0]\}\times\P^1_{k_E}\times \{[1:0]\}$, c'est-à-dire d'une 
copie de $\P^1_{k_E}$ et d'une copie de $\P^1_{k_E} \times \P^1_{k_E}$ s'intersectant 
transversalement en un point.

\subsubsection{Un produit}
\label{sssec:filrouge}

Reprenons à présent l'exemple \ref{exMoebius} dont le gène décoré a déjà 
été déterminé. En remplaçant les allèles $\gA$, $\gB$, $\gAB$ par des 
variables $x_i$ et l'allèle $\gO$ par $0$ comme le veut le théorème 
\ref{thequationsKisin}, nous obtenons le diagramme ci-après qui code les 
équations de la variété de Kisin correspondante.

\begin{center}
\begin{tikzpicture}[yscale=0.8]
\draw [fill=fond, thick] (-0.6,-0.5) rectangle (10.6,1.5);
\draw[thick,->>] (-0.6,-0.5)--(-0.6,0.5);
\draw[thick,->>] (10.6,1.5)--(10.6,0.5);
\node[color=colB] at (0, 0) { \ph $x_{11}$ };
\node[color=colB] at (0, 1) { \ph $x_0$ };
\node[color=colB] at (1, 0) { \ph $x_{12}$ };
\node[color=colB] at (1, 1) { \ph $x_1$ };
\node[color=colB] at (2, 0) { \ph $x_{13}$ };
\node[color=colB] at (2, 1) { \ph $x_2$ };
\node[color=colB] at (3, 0) { \ph $x_{14}$ };
\node[color=colB] at (3, 1) { \ph $x_3$ };
\node[color=colA] at (4, 0) { \ph $x_{15}$ };
\node[color=colA] at (4, 1) { \ph $x_4$ };
\node[color=colA] at (5, 0) { \ph $x_{16}$ };
\node[color=colA] at (5, 1) { \ph $x_5$ };
\node[color=colA] at (6, 0) { \ph $x_{17}$ };
\node[color=colA] at (6, 1) { \ph $x_6$ };
\node[color=colA] at (7, 0) { \ph $0$ };
\node[color=colA] at (7, 1) { \ph $x_7$ };
\node[color=colA] at (8, 0) { \ph $x_{19}$ };
\node[color=colA] at (8, 1) { \ph $x_8$ };
\node[color=colA] at (9, 0) { \ph $x_{20}$ };
\node[color=colA] at (9, 1) { \ph $x_9$ };
\node[color=colA] at (10, 0) { \ph $0$ };
\node[color=colA] at (10, 1) { \ph $x_{10}$ };
\draw[color=colB,thick] (0.2,0.2)--(0.8,0.7);
\draw[color=colB,thick] (1.2,0.2)--(1.8,0.7);
\draw[color=colB,thick] (2.2,0.2)--(2.8,0.7);
\draw[color=colA,thick] (4.2,0.7)--(4.8,0.2);
\draw[color=colA,thick] (5.2,0.2)--(5.8,0.7);
\draw[color=colA,thick] (6.2,0.7)--(6.8,0.2);
\draw[color=colA,thick] (7.2,0.7)--(7.8,0.2);
\draw[color=colA,thick] (8.2,0.7)--(8.8,0.2);
\draw[color=colA,thick] (8.2,0.2)--(8.8,0.7);
\draw[color=colA,thick] (9.2,0.7)--(9.8,0.2);
\end{tikzpicture}
\end{center}

À ce niveau, plusieurs simplifications élémentaires peuvent être faites. Par 
exemple, le lien qui relie $x_6$ à $0$ est inutile car le produit de ces 
deux valeurs est toujours nul. De même, comme $[x_7:0]$ désigne un point de 
l'espace projectif, nous pouvons, sans perte de généralité, supposer que 
$x_7$ vaut $1$. Le lien entre $x_7$ et $x_{19}$ implique alors que 
$x_{19}$ s'annule nécessairement et, par suite, comme précédemment, que 
$x_8$ peut être supposé égal à $1$. Les décorations nous apprennent 
également que les points projectifs $[x_8:x_{19}] = [1:0]$ et $[x_9: 
x_{20}]$ doivent coïncider ; ainsi, nous avons $x_{20} = 0$ et nous 
pouvons supposer que $x_9 = 1$. Enfin, la coordonnée $x_{10}$ peut être, 
elle aussi, supposée être égale à $1$ car elle est
face à un $0$. Après ces simplifications, nous obtenons :

\begin{center}
\begin{tikzpicture}[yscale=0.8]
\draw [fill=fond, thick] (-0.6,-0.5) rectangle (10.6,1.5);
\draw[thick,->>] (-0.6,-0.5)--(-0.6,0.5);
\draw[thick,->>] (10.6,1.5)--(10.6,0.5);
\node[color=colB] at (0, 0) { \ph $x_{11}$ };
\node[color=colB] at (0, 1) { \ph $x_0$ };
\node[color=colB] at (1, 0) { \ph $x_{12}$ };
\node[color=colB] at (1, 1) { \ph $x_1$ };
\node[color=colB] at (2, 0) { \ph $x_{13}$ };
\node[color=colB] at (2, 1) { \ph $x_2$ };
\node[color=colB] at (3, 0) { \ph $x_{14}$ };
\node[color=colB] at (3, 1) { \ph $x_3$ };
\node[color=colA] at (4, 0) { \ph $x_{15}$ };
\node[color=colA] at (4, 1) { \ph $x_4$ };
\node[color=colA] at (5, 0) { \ph $x_{16}$ };
\node[color=colA] at (5, 1) { \ph $x_5$ };
\node[color=colA] at (6, 0) { \ph $x_{17}$ };
\node[color=colA] at (6, 1) { \ph $x_6$ };
\node[color=colA] at (7, 0) { \ph $0$ };
\node[color=colA] at (7, 1) { \ph $1$ };
\node[color=colA] at (8, 0) { \ph $0$ };
\node[color=colA] at (8, 1) { \ph $1$ };
\node[color=colA] at (9, 0) { \ph $0$ };
\node[color=colA] at (9, 1) { \ph $1$ };
\node[color=colA] at (10, 0) { \ph $0$ };
\node[color=colA] at (10, 1) { \ph $1$ };
\draw[color=colB,thick] (0.2,0.2)--(0.8,0.7);
\draw[color=colB,thick] (1.2,0.2)--(1.8,0.7);
\draw[color=colB,thick] (2.2,0.2)--(2.8,0.7);
\draw[color=colA,thick] (4.2,0.7)--(4.8,0.2);
\draw[color=colA,thick] (5.2,0.2)--(5.8,0.7);
\end{tikzpicture}
\end{center}

En nous appuyant sur les deux exemples traités précédemment, nous 
pouvons affirmer que la variété de Kisin qui nous intéresse est 
isomorphe à un produit $\mathcal V_1 \times \mathcal V_2$ où $\mathcal 
V_1$ est une chaîne de $\P^1_{k_E}$ de longueur $4$ et $\mathcal V_2$ est
l'union d'une copie de $\P^1_{k_E}$ et d'une copie de $\P^1_{k_E} \times \P^1_{k_E}$
qui s'intersectent transversalement en un point.

\subsubsection{Le cas d'une représentation générique}
\label{subsubgenerique}
Rappelons $\rhobar \simeq \Ind_{G_{F'}}^{G_F} \Big( \omega_{2f}^h \cdot \nr'(\theta)\Big)$ avec
$h\equiv 1+\sum_{i=0}^{f-1}h_ip^{f-1-i}\mod (q+1)$ (voir (\ref{defhachi})).
\begin{definit}
La représentation $\rhobar$ est dite générique si pour tout $i\in\llbracket 0,f-1\rrbracket$, 
$1\leq h_i\leq p-2$.
\end{definit}

\newcommand{\ancienalpha}{\beta}

\begin{cor}\label{corgenerique}
Si $\rhobar$ est générique, la variété de Kisin $\vK^{\psi}(\vv, \ttt, 
\rhobar)$ est soit vide, soit réduite à un point.
\end{cor}

\begin{proof} 
Soit $i\in\llbracket 0,f-1\rrbracket$.
Un calcul direct à partir des formules du lemme \ref{lem:relalphai} 
aboutit à :
$$\alpha_{i+f} - \alpha_i \equiv \sum_{j=0}^{f-1} h_{i+j} p^{f-1-j}
\pmod e.$$
Or l'hypothèse de généricité implique que $1 \leq h_j \leq p-2$ pour 
tout $j \in \Z$ et, par suite, que le membre de droite dans la
congruence précédente est dans l'intervalle $[\nu, e-\nu]$. De plus,
s'il existe un entier $i$ tel que $\alpha_{i+f} - \alpha_i \equiv
\pm \nu \pmod e$, nous déduisons que $\eta = \eta'$ (\emph{cf}
démonstration du lemme \ref{pasABBA}), ce qui est exclu. Ainsi, pour tout
$i$, la différence $\alpha_{i+f} - \alpha_i$ est congrue modulo $e$ à un 
entier de l'intervalle $]\nu, e-\nu[$. Il résulte que ceci que $\alpha_i 
\geq \nu$ ou $\alpha_{i+f} \geq \nu$ et, par conséquent, qu'au moins
l'un des deux allèles $X_i$ ou $X_{i+f}$ est égal à $\gO$. Par le
théorème \ref{thequationsKisin}, ceci implique que $\vK^{\psi}(\vv, 
\ttt, \rhobar)$ est soit vide, soit réduite à un point.
\end{proof}

\section{Le calcul de la variété de Kisin}
\label{sec:Demonstrations}

Le but de cette partie est de démontrer le théorème 
\ref{thequationsKisin}.

\subsection{Module de Breuil--Kisin maximal.}
\label{paramax} 

Nous reprenons les notations des paragraphes précédents $(\alpha_i)_{i\in\Z}$, 
$(X_i)_{i\in\Z}$ (voir \S \ref{secrep}) et 
$\bbM(\rhobar)=(\bbM(\rhobar)^{(i)})_{0\leq i\leq f-1}$ de base 
$(\varepsilon_0^{(i)},\varepsilon_1^{(i)})_{0\leq i\leq f-1}$ (voir \S 
\ref{secdescente}) associées à $ \rhobar \simeq \Ind_{G_{F'}}^{G_F}\Big( 
\omega_{2f}^h \cdot \nr'(\theta)\Big)$.
Effectuons les changements de base 
$$
e^{(i)}_0=u^{-\lfloor \frac{p^ih}{q+1}\rfloor}\varepsilon _0^{(i)},\;\;\; 
e^{(i)}_1=u^{-\lfloor \frac{p^{i+f}h}{q+1}\rfloor - e \lfloor \frac{p^ih}{q+1}\rfloor}
\varepsilon _1^{(i)},\;\;\;  0\leq i\leq f-1.
$$
Dans les bases $(e_0^{(i)},e_1^{(i)})_{0 \leq i \leq f-1}$, les matrices 
de $\varphi: \bbM(\rhobar)^{(i)} \rightarrow \bbM(\rhobar)^{(i+1)}$ 
sont
\begin{equation}
\label{equaGi}
\begin{array}{rl} 
\left(\begin{matrix}u^{h_i}&0\cr 0&u^{h_{i+f}} \end{matrix}\right)
& \text{pour } 0\leq i\leq f-2, \medskip \\
\left(\begin{matrix} 0&\theta^{-1} u^{h_{2f-1}}\cr u^{h_{f-1}}& 0 \end{matrix}\right)& \text{pour } i=f-1.
\end{array}
\end{equation}
De plus, $\omega_f$ étant d'ordre $e$, la donn\'ee de descente agit par 
$g\in\Gal(L/F)$,
\begin{align}
[g]\cdot(e_0^{(i)}) &=(\tau_i\circ\omega_f)(g)^{-\lfloor \frac{p^ih}{q+1}\rfloor}e_0^{(i)} \label{eq:g0} \\ 
[g]\cdot(e_1^{(i)}) &=(\tau_i\circ\omega_f)(g)^{-\lfloor \frac{p^{i+f}h}{q+1}\rfloor}e_1^{(i)}. \label{eq:g1}
\end{align}

Notons $\MK(\bar{\rho})$ le sous $k_F[[u]]$-module de 
$\bbM(\bar{\rho})$ engendr\'e par les $e_0^{(i)},e_1^{(i)}$, $0\leq 
i\leq f-1$. Comme $\MK(\bar{\rho})$ est stable sous l'action de $k_E$, 
c'est un $\big(k_E\otimes_{\F_p}k_F[[u]]\big)$-module. La proposition
suivante est un analogue de la Proposition 3.6.7 de \cite{CL}.

\begin{prop}\label{prop:modulemax}
Soit $R$ une $k_E$-algèbre.
Le $\big(R \otimes_{\F_p}k_F[[u]]\big)$-module $R \otimes_{k_E} 
\MK(\rhobar)$ est maximal pour l'inclusion parmi les 
$\big(R \otimes_{\F_p}k_F[[u]]\big)$-réseaux de $R \otimes_{\F_p} 
\bbM(\rhobar)$ qui sont stables par $\varphi$.
\end{prop}

\begin{proof}
Il suffit de montrer que si $x$ est un élément de $R \otimes_{\F_p}
\bbM(\rhobar)$ qui n'appartient pas à $R \otimes_{k_E}
\MK(\rhobar)$, alors le $\big(R \otimes_{\F_p}k_F[[u]]\big)$-module
engendré par les $\varphi^n(x)$ pour $n \geq 0$ n'est pas de type
fini. Or, un calcul direct montre que si $v > 0$, alors :
$$\varphi^{2f}(u^{-v} e_0^{(i)}) = \alpha_0 u^{-w_0} e_0^{(i)}
\quad \text{et} \quad
\varphi^{2f}(u^{-v} e_1^{(i)}) = \alpha_1 u^{-w_1} e_1^{(i)}$$
où $\alpha_0$ et $\alpha_1$ sont des éléments non nuls de $k_E$ et $w_0$ 
et $w_1$ sont des entiers strictement inférieurs à $v$. La propriété 
annoncée et, par suite, la proposition en résultent.
\end{proof}

\subsection{Détermination des réseaux}
\label{subsecDet}

Dans ce paragraphe \S \ref{subsecDet}, nous fixons une $k_E$-algèbre $S$ 
ainsi qu'un réseau de Breuil--Kisin $\MK$ de type $(\vv,\ttt)$ (voir 
définition \ref{defmoduleBK}) sur $S \otimes_{\Zp} \SK$ inclus dans $S 
\otimes_{\F_p} \bbM(\rhobar)$. Pour alléger les notations, nous posons
à partir de maintenant $\SK_S = S \otimes_{\Zp} \SK$.
De la proposition \ref{prop:modulemax}, nous déduisons que $\MK$ est 
inclus dans $S \otimes_{k_E} \MK(\rhobar)$. Nous rappelons
qu'étant donné que $S$ est une $k_E$-algèbre, les modules $\MK$ et 
$\MK(\rhobar)$ se décomposent de façon canonique comme suit :
$$\MK = \MK^{(0)} \oplus \cdots \oplus \MK^{(f-1)}
\quad \text{et} \quad
\MK(\rhobar) = \MK(\rhobar)^{(0)} \oplus \cdots \oplus \MK(\rhobar)^{(f-1)}$$
où les $\MK^{(i)}$ (resp. $\MK(\rhobar)^{(i)}$) sont des modules sur 
$\SK_S^{(i)} = S \otimes_{k_F, \tau_i} k_F[[u]]$ (resp. sur $S 
\otimes_{k_F, \tau_i} k_F[[u]]$). L'inclusion $\MK 
\subset S\otimes_{k_E} \MK(\rhobar)$ se lit alors facteur par facteur : 
elle implique que, pour tout $i$, nous avons $\MK^{(i)} \subset S 
\otimes_{k_E} \MK(\rhobar)^{(i)}$.

Dans ce paragraphe \S \ref{subsecDet}, nous supposons de surcroît que 
$\MK$ est libre (de rang $2$) comme module sur $\SK_S$. Ceci est 
équivalent à demander la liberté des $\MK^{(i)}$ sur $S \otimes_{k_E} 
k_E[[u]]$.

\subsubsection{Action de la donnée de descente}\label{subsubPi}

De l'inclusion $\MK^{(i)} \subset S \otimes_{k_E} 
\MK(\rhobar)^{(i)}$, nous déduisons que, pour tout $i$ entre $0$ et 
$f-1$, toute $\SK_S^{(i)}$-base de $\MK^{(i)}$ a des coordonnées dans la 
base $(e_0^{(i)},e_1^{(i)})$ qui appartiennent à $\SK_S^{(i)}$.
Soit $i$ un indice entre $0$ et $f-1$. Comme $\MK$ est de type 
$(\vv_0,\ttt)$, il existe une $\SK_S^{(i)}$-base $(e_\eta^{(i)}, 
e_{\eta'}^{(i)})$ de $\MK^{(i)}$ sur laquelle l'action de $g \in 
\Gal(L/K)$ est donnée par
$g \cdot e_{\eta}^{(i)} = \bar{\eta}(g) e_{\eta}^{(i)}$ et
$g \cdot e_{\eta'}^{(i)} = \bar{\eta}'(g) e_{\eta'}^{(i)}$.
Fixons une telle base $(e_\eta^{(i)},e_{\eta'}^{(i)})$ et notons 
$P^{(i)}$ la matrice de passage de $(e_\eta^{(i)},e_{\eta'}^{(i)})$ à 
$(e_0^{(i)},e_1^{(i)})$.

\begin{lem}\label{lemPi}
La matrice $P^{(i)}$ est de la forme :
$$P^{(i)}=
\left(\begin{array}{cc} 
u^{\alpha_{i}}a_i&u^{\alpha'_{i}}b'_i \cr
u^{\alpha_{i+f}}b_i&u^{\alpha'_{i+f}}a'_i \cr
\end{array}\right),$$
avec $a_i$, $a'_i$, $b_i$ et $b'_i$ dans $S \otimes_{k_E} k_E[[u^e]]$ 
et $(\alpha'_{i})_{i\geq 0}$ la suite d'éléments de $\llbracket 
0,e-1\rrbracket$ définis par les congruences
\begin{equation}\label{alpha'}
\alpha'_{i}+\alpha_{i+f}\equiv \nu\pmod e.
\end{equation}
\end{lem}

\begin{proof}
Les équations \eqref{eq:g0} et \eqref{eq:g1} assurent que $g \in 
\Gal(L/K)$ agit sur $u^{\alpha_i} e_0^{(i)}$ et sur $u^{\alpha_{i+f}} 
e_1^{(i)}$ par multiplication par $\bar\eta(g)$. Ainsi, étant donnés 
$s_0, s_1 \in \SK_S^{(i)}$, le groupe de Galois $\Gal(L/K)$ agit sur 
$s_0 e_0^{(i)} + s_1 e_1^{(i)}$ \emph{via} le caractère $\bar\eta$ si, 
et seulement s'il fixe $u^{\alpha_i} s_0$ et $u^{\alpha_{i+f}} s_1$,
c'est-à-dire si, et seulement si ces deux éléments appartiennent à 
$S \otimes_{k_E} k_E[[u^e]]$. Nous en déduisons la forme annoncée
pour la première colonne de la matrice $P^{(i)}$.
Pour la deuxième colonne, il suffit de remarquer que la condition 
$\det\rhobar_{|I_F}=\omega_f^{h} = (\bar{\eta}\bar{\eta}'\omega)_{|I_F}$ 
implique :
$$\alpha'_i \equiv
\left\lfloor \frac{p^i h}{q+1} \right\rfloor - p^i \gamma \pmod e$$
c'est-à-dire que $\alpha'_i$ est construit de la même manière que
$\alpha_i$ après avoir échangé les roles de $\eta$ et $\eta'$.
\end{proof}

\begin{rem}\label{remgamma'}
D'après la démonstration ci-dessus, les entiers $\alpha'_i$ sont égaux 
aux $\alpha_i(h,\gamma', \gamma)$ (voir lemme \ref{lem:depchoix}). Ceci 
justifie la notation que nous avons utilisée et éclaire sur la 
complémentarité des rôles joués par $\alpha_{i}$ (associé à $\eta'$) et 
$\alpha'_{i}$ (associé à $\eta$).
\end{rem}

\subsubsection{La condition sur le déterminant}
\label{subsec:conddet}

Pour tout $i$ dans $\Z / f \Z$ , la matrice de l'application
$\varphi : \MK^{(i)} \to \MK^{(i+1)}$ dans les bases respectives
$(e_\eta^{(i)}, e_{\eta'}^{(i)})$ et 
$(e_\eta^{(i+1)}, e_{\eta'}^{(i+1)})$ s'écrit :
\begin{equation}\label{eq:Hi}
H^{(i)}=(P^{(i+1)})^{-1} G^{(i)}\varphi(P^{(i)})
\end{equation}
où $G^{(i)}$ est la matrice donnée par la formule \eqref{equaGi}.

\begin{lem}
\label{lem:valdet}
Pour tout $i$ dans $\llbracket 0 , f-1 \rrbracket $, nous avons
$$\det P^{(i)} \equiv a_i u^\nu \pmod{u^{\nu+1} \SK_S}$$
où $a_i$ est un élément inversible de $S$ et où nous rappelons que 
$\nu = \frac{p^f-1}{p-1} - 1$.
\end{lem}

\begin{proof}
Soit $i \in \Z/f\Z$.
Le fait que $\MK^{(i)}$ et $\MK(\rhobar)^{(i)}$ engendrent le même 
$\SK_S^{(i)}[1/u]$-module implique que la matrice $P^{(i)}$ est inversible dans 
$M_2(\SK_S^{(i)}[1/u])$. Ainsi son déterminant vérifie
$\det P^{(i)} \equiv a_i u^{\nu_i} \pmod{u^{\nu_i+1} \SK_S}$
pour un certain entier positif ou nul $\nu_i$ et un certain élément
$a_i$ qui est inversible dans $S$. En outre, vue la forme de $G^{(i)}$,
l'égalité \eqref{eq:Hi} entraîne l'égalité numérique
$e = p \nu_i - \nu_{i+1} + p-1$. Comme $\nu_f$ et $\nu_0$ coïncident, 
tous les $\nu_i$ sont égaux au point fixe $\nu$.
\end{proof}

Du lemme \ref{lem:valdet}, il résulte que, pour tout entier $d$ supérieur ou égal à $\nu$, les vecteurs $u^d e_0^{(i)} $ et $u^d e_1^{(i)} $ sont dans l'image de $P^{(i)}$, c'est-à-dire dans $\MK^{(i)}$.
Ceci est en particulier vrai pour $d = e$ et implique que nous pouvons, sans perte de généralité, choisir $P^{(i)}$ de la forme
$$
P^{(i)}=
\left(
\begin{array}{ccc} 
u^{\alpha_{i}}a_i&u^{\alpha'_{i}}b'_i \cr
u^{\alpha_{i+f}}b_i&u^{\alpha'_{i+f}}a'_i \cr
\end{array}\right),$$
avec $a_i$, $b_i$, $a'_i$ et $b'_i$ dans $S$. En vertu du lemme
\ref{lem:valdet}, nous obtenons :
\begin{equation}\label{equadet}
 \alpha_{i}+\alpha'_{i+f}=\nu
 \quad \text{ou} \quad 
 \alpha_{i+f}+\alpha'_{i}=\nu.
\end{equation}

\subsubsection{Une forme normale pour les $P^{(i)}$}
\label{subsec:formenormale}

Soit $i \in\llbracket 0,f-1\rrbracket$ un entier fixé.
Supposons, dans un premier temps, que $\alpha_{i}+\alpha'_{i+f} = \nu 
<\alpha_{i+f}+\alpha'_{i}$. Alors $\alpha_{i+f}+\alpha'_{i}$ est
nécessairement supérieur ou égal à $\nu + e$, étant donné qu'il doit
être congru à $\nu$ modulo $e$. Comme $\alpha_{i+f}<e$ et $\alpha'_{i}<e$, 
nous en déduisons que $\alpha_{i}\leq\nu<\alpha'_{i}$ et $\alpha'_{i+f}
\leq\nu<\alpha_{i+f}$. Ainsi, quitte à modifier la base de $\MK^{(i)}$
--- ce qui revient à faire des opérations sur les colonnes de $P^{(i)}$
--- nous pouvons supposer que $P^{(i)}$ prend la forme :
\begin{equation}\label{eq:Pi1}
P^{(i)}=\left(\begin{array}{cc} u^{\alpha_{i}}& 0\cr 
0& u^{\nu-\alpha_{i}}\cr\end{array}\right).
\end{equation}

De la même manière, si $\alpha'_{i}+\alpha_{i+f} = \nu <\alpha'_{i+f}+
\alpha_{i}$, nous pouvons supposer que :
\begin{equation}\label{eq:Pi2}
P^{(i)}=\left(\begin{array}{cc} 0 & u^{\nu-\alpha_{i+f}}\cr 
u^{\alpha_{i+f}}&0\cr\end{array}\right).
\end{equation}

Il reste à examiner le cas où $\alpha_{i}+\alpha'_{i+f} = \alpha_{i+f} + 
\alpha'_{i} = \nu$. Posons $\delta = \nu - (\alpha_i + \alpha_{i+f})$. 
De même que dans la démonstration du lemme \ref{pasABBA}, nous voyons 
que la condition $\eta \neq \eta'$ implique $\delta \neq 0$. Par
ailleurs, la matrice $P^{(i)}$ s'écrit :
$$P^{(i)}= \left( \begin{array}{cc} 
u^{\alpha_{i}}a_i & u^{\alpha_{i}+\delta}b'_i \cr
u^{\alpha_{i+f}}b_i & u^{\alpha_{i+f} + \delta}a'_i
\end{array}\right).$$
Ainsi, dans le cas où $\delta > 0$, l'image de $P^{(i)}$ ne dépend
pas des valeurs de $a'_i$ et $b'_i$ dès lors que celles-ci vérifie
la condition :
\begin{equation}
\label{eq:ptproj}
a_i a'_i - b_i b'_i \in S^\times
\end{equation}
où $S^\times$ désigne le groupe des unités de $S$. De plus, remplacer 
le couple $(a_i, b_i)$ par $(s a_i, s b_i)$ avec $s \in S^\times$ ne 
modifie pas non plus l'image de $P^{(i)}$. 
Par ailleurs, l'existence des éléments $a'_i$ et $b'_i$ vérifiant la 
condition \eqref{eq:ptproj} entraîne que le sous-module 
de $S^2$ engendré par $(a_i, b_i)$ définit un point de l'espace 
projectif $\P^1(S)$ (\emph{i.e.} le quotient $S^2 / (a_i,b_i)S$ est 
projectif). Dans la suite, nous utilisons la notation classique 
$[a_i:b_i]$ pour désigner ce point.
En conclusion, le réseau $\MK^{(i)}$ est entièrement caractérisé par
le $S$-point $[a_i:b_i]$ de la droite projective.
De la même manière, si $\delta < 0$, nous trouvons que $\MK^{(i)}$ est
déterminé par le point projectif $[a'_i:b'_i] \in \P^1(S)$.

\subsubsection{Stabilité par $\varphi$}
\label{subsec:stabphi}

Examinons à présent à quelle condition sur $a_i, b_i, a'_i, b'_i$ le
réseau $\MK$ est stable par $\varphi$, c'est-à-dire à quelle condition
les matrices $H^{(i)}$ définies par l'égalité \eqref{eq:Hi} sont à
coefficients dans $\SK_S$. Un calcul immédiat montre que la matrice 
$H^{(i)}$ est de la forme :
\begin{equation}
\label{hachi}
H^{(i)}  = 
(P^{(i+1)})^{-1} G^{(i)}\varphi(P^{(i)}) 
= (\det(P^{(i+1)}))^{-1} K^{(i)} 
= (\det(P^{(i+1)}))^{-1}
\begin{pmatrix}
A_i & B'_i\\
 B_i & A'_i
\end{pmatrix}
\end{equation}
où les coefficients de la matrice $K^{(i)}$ sont donnés par les formules 
suivantes :
\begin{itemize}
\item[$\bullet$] $A_i = a_{i} a'_{i+1} u^{\alpha'_{i+1+f} + p \alpha_{i} + h_{i}} - b_i b'_{i+1}u^{\alpha'_{i+1} + p \alpha_{i+f} + h_{i+f}}, $
\item[$\bullet$] $B'_i = b'_{i} a'_{i+1} u^{\alpha'_{i+1+f} + p \alpha'_{i} + h_{i}} - a'_i b'_{i+1}u^{\alpha'_{i+1} + p \alpha'_{i+f} + h_{i+f}}, $
\item[$\bullet$] $B_i = - a_{i} b_{i+1} u^{\alpha_{i+1+f} + p \alpha_{i} + h_{i}} + b_i a_{i+1}u^{\alpha_{i+1} + p \alpha_{i+f} + h_{i+f}}, $
\item[$\bullet$] $A'_i = - b'_{i} b_{i+1} u^{\alpha_{i+1+f} + p \alpha'_{i} + h_{i}} + a'_i a_{i+1}u^{\alpha_{i+1} + p \alpha'_{i+f} + h_{i+f}},$
\end{itemize}
lorsque $i$ est entre $0$ et $f-2$ et par une formule analogue qui 
tient compte de la forme antidiagonale de $G^{(f-1)}$ (\emph{cf} 
Eq.~\eqref{equaGi}) lorsque $i = f-1$. Comme le déterminant de 
$P^{(i+1)}$ est de valuation $\nu$, la matrice $H^{(i)}$ est à 
coefficients dans $\SK_S^{(i+1)}$ si, et seulement si tous les coefficients de 
$K^{(i)}$ de degré strictement inférieur à $\nu$ s'annulent. Ceci
fournit des conditions algébriques qui déterminent les équations de
la variété de Kisin, comme nous allons le préciser dans la partie \S 
\ref{subsec:Ktrhobar} suivante.

\subsubsection{La variété $\mathcal K(\ttt,\rhobar)$}
\label{subsec:Ktrhobar}

Considérons l'espace $\mathcal K = \prod_{i=0}^{f-1}
\P^1_{k_E}$ et notons $[x_i:x_{i+f}]$ les coordonnées homogènes sur la $i$-ième 
copie de $\P^1_{k_E}$ en convenant, de même que dans l'énoncé du théorème
\ref{thequationsKisin}, que $x_i = x_{i \text{ mod } 2f}$ pour $i \in
\N$. Soit $\mathcal K(\ttt,\rhobar)$ la sous-variété de $\mathcal K$ 
définie par les familles d'équations suivantes :

\begin{itemize}
\item[\textit{(A)}] $x_i = 0$ lorsque $\alpha_i + \alpha'_{i+f} > \nu$, \medskip
\item[\textit{(A')}] $x_{i+f} = 0$ lorsque $\alpha_{i+f} + \alpha'_i > \nu$, \medskip
\item[\textit{(B)}] l'annulation des coefficients de degré strictement inférieur
à $\nu$ dans $A_i$, $A'_i$, $B_i$ et $B'_i$ après avoir posé
lorsque $\alpha_i + \alpha'_{i+f} = \alpha_{i+f} + \alpha'_i = \nu$ :
$$\begin{array}{ll}
x_i = a_i \text{ et } x_{i+f} = b_i &
\text{si } \alpha_i + \alpha_{i+f} < \nu, \smallskip \\ 
x_i = a'_i \text{ et } x_{i+f} = b'_i &
\text{si } \alpha_i + \alpha_{i+f} > \nu. 
\end{array}$$
\end{itemize}

Nous allons construire un morphisme $\mathcal K(\ttt, \rhobar) \to 
\vK^\psi(\vv, \ttt, \rhobar)$ à l'aide du foncteur des points. Fixons une 
$k_E$-algèbre $T$ et considérons un point $(L_0, \ldots, L_{f-1}) \in 
\mathcal K(T)$. Chaque $L_i$ est donc un sous-$T$-module projectif de 
rang $1$ de $T^2$ tel que le quotient $T^2/L_i$ soit aussi projectif.
À cette donnée, nous associons un nouveau $f$-uplet $(\Lambda^{(0)}, 
\ldots, \Lambda^{(f-1)})$ où $\Lambda^{(i)}$ est le $\SK_T^{(i)}$-réseau 
de $\MK(\rhobar)^{(i)}$ défini ainsi :

\begin{itemize}
\item[\textit{(A)}] si $\alpha_i + \alpha'_{i+f} > \nu$,
$\Lambda^{(i)} = u^{\alpha_i} e_0^{(i)} + u^{\nu - \alpha_i} e_1^{(i)}$ ;
\smallskip

\item[\textit{(A')}] si $\alpha'_i + \alpha_{i+f} > \nu$,
$\Lambda^{(i)} = u^{\nu-\alpha_{i+f}} e_0^{(i)} + u^{\alpha_{i+f}} e_1^{(i)}$ ;
\medskip

\item[\textit{(B)}] si $\alpha_i + \alpha'_{i+f} = \alpha'_i + 
\alpha_{i+f} = \nu$ et $\alpha_i + \alpha_{i+f} < \nu$ :
$$\Lambda^{(i)} = \iota\big(L_i \otimes_T \SK_T^{(i)} + 
u^{\nu - \alpha_i - \alpha_{i+f}} \: (\SK_T^{(i)})^2\big)$$
où $\iota : (\SK_T^{(i)})^2 \to \MK(\rhobar)^{(i)}, (s_0, s_1) \to
u^{\alpha_i} s_0 e_0^{(i)} + u^{\alpha_{i+f}} s_1 e_1^{(i)}$ ;
\medskip

\item[\textit{(B')}] si $\alpha_i + \alpha'_{i+f} = \alpha'_i + 
\alpha_{i+f} = \nu$ et $\alpha_i + \alpha_{i+f} > \nu$ :
$$\Lambda^{(i)} = \iota'\big(L_i \otimes_T \SK_T^{(i)} + 
u^{\nu - \alpha'_i - \alpha'_{i+f}} \: (\SK_T^{(i)})^2\big)$$
où $\iota' : (\SK_T^{(i)})^2 \to \MK(\rhobar)^{(i)}, (s_0, s_1) \to
u^{\alpha'_i} s_0 e_0^{(i)} + u^{\alpha'_{i+f}} s_1 e_1^{(i)}$.
\end{itemize}

Cette association définit un morphisme $\mathcal K$ dans la 
grassmanienne affine adéquate qui induit, par restriction et 
corestriction, un morphisme entre variétés algébriques $\mathcal 
K(\ttt,\rhobar) \to \vK^\psi(\vv, \ttt, \rhobar)$. Nous vérifions 
immédiatement que ce dernier est injectif sur les $T$-points pour 
toute $k_E$-algèbre $T$. Considérons à présent une $k_E$-algèbre
$T$, ainsi que $(\Lambda^{(0)}, \ldots, \Lambda^{(f-1)})$ un 
$T$-point de la variété de Kisin $\vK^\psi(\vv, \ttt, \rhobar)$. Il
existe un recouvrement ouvert $T \to S$ tel que les $S \otimes_T
\Lambda^{(i)}$ soient tous libres. Des résultats des \S
\ref{subsubPi}--\ref{subsec:stabphi}, nous déduisons alors que le 
$S$-point $(S \otimes_T \Lambda^{(i)})_{0\leq i < f}$ de
$\vK^\psi(\vv, \ttt, \rhobar)$ provient d'un 
$S$-point de $\mathcal K(\ttt,\rhobar)$. Il en résulte que le morphisme 
$\mathcal K(\ttt,\rhobar) \to \vK^\psi(\vv, \ttt, \rhobar)$ est un 
isomorphisme et nous avons ainsi démontré la première assertion du 
théorème \ref{thequationsKisin}.

\subsection{Traduction génétique}

Afin d'obtenir une démonstration complète du théorème 
\ref{thequationsKisin}, il nous reste à réinterpréter les équations 
définissant $\mathcal K (\ttt,\rhobar)$ dans le langage génétique. C'est 
l'objet de cette partie. Nous choisissons et fixons un gène associé au 
couple $(\rhobar,\ttt)$, que nous noterons $X = (X_i)_{i\in\Z}$.

\begin{lem}\label{lem:alphaidiag}
Pour $i\in\llbracket 0,f-1\rrbracket$, nous avons les équivalences
suivantes :
\begin{enumerate}[i)]
\item $\alpha_{i}+\alpha'_{i+f}>\nu$ si et seulement si 
$X_i = \gO$
\item $\alpha'_{i}+\alpha_{i+f}>\nu$ si et seulement si 
$X_{i+f} = \gO$
\end{enumerate}
\end{lem}

\begin{proof}
Par définition, si $X_i = \gO$ alors $\alpha_i > \nu$ et 
donc \emph{a fortiori} nous avons $\alpha_{i}+\alpha'_{i+f}>\nu$.
Réciproquement, supposons $\alpha_{i} + \alpha'_{i+f} > \nu$. Alors,
de la congruence \eqref{alpha'} appliquée à $i+f$, nous déduisons
que $\alpha_{i}+\alpha'_{i+f} \geq \nu+e$ et, par suite, $\alpha_i
\geq \nu + e - \alpha'_{i+f} > \nu$ car $\alpha'_{i+f} < e$. Ainsi
$X_i = \gO$.

La démonstration du \emph{ii)} est similaire.
\end{proof}

Il résulte du lemme \ref{lem:alphaidiag} que les jeux 
d'équations \emph{(A)} et \emph{(A')} apparaissant dans la définition de 
$\mathcal K(\ttt, \rhobar)$ (\S \ref{subsec:Ktrhobar}) 
correspondent aux équations \emph{(A)} du 
théorème \ref{thequationsKisin}. Il ne reste donc plus qu'à démontrer
que le jeu \emph{(B)} fournit les mêmes équations que les jeux \emph{(B1)}, 
\emph{(B2)} et \emph{(B3)} du théorème \ref{thequationsKisin}.
D'après la discussion qui suit la définition \ref{defdominance}, si $X$ n'a pas de caractère dominant, le résultat est clair. Nous 
supposons donc, à partir de maintenant, que $X$ a un caractère dominant.

\begin{lem}\label{lemgenereseau}
Soit $i\in\llbracket 0,f-1\rrbracket$ tel que $X_i \neq \gO$ et $X_{i+f} 
\neq \gO$. Alors :

\begin{enumerate}[i)]
\item
$\alpha_i + \alpha'_{i+f} = \alpha'_i + \alpha_{i+f} = \nu$ \smallskip
\item
$\gA$ (resp. $\gB$) est dominant en $i$ ssi
$\alpha_i + \alpha_{i+f} < \nu$ (resp. $\alpha_i + \alpha_{i+f} > \nu$).
\end{enumerate}
\end{lem}

\begin{proof}
L'énoncé \emph{i)} est une conséquence directe du lemme 
\ref{lem:alphaidiag}.
Supposons à présent que $\gA$ soit dominant dans $\binom{X_i}{X_{i+f}}$ 
et démontrons que $\alpha_i + \alpha_{i+f} < \nu$. S'il y a parmi $X_i$ 
et $X_{i+f}$ exactement un $\gA$ et un $\gB$, le lemme 
\ref{lem:relalphai} implique que
$$\nu-\alpha_{i+1}-\alpha_{i+f+1}=p \cdot (\nu-\alpha_i-\alpha_{i+f})$$
et nous sommes ramenés à l'énoncé analogue où $i$ est remplacé
par $i+1$. Ainsi, nous pouvons supposer que le couple d'allèles
$\binom{X_i}{X_{i+f}}$ est $\binom \gA \gA$, $\binom \gA \gAB$ ou
$\binom \gAB \gA$. Mais alors :
$$
\begin{array}{rcll}
\nu-\alpha_{i} - \alpha_{i+f} & > &
\frac{p-2}{p} \nu-\epsilon_i-\epsilon_{i+f} & 
\mbox{si } \binom{X_i}{X_{i+f}}=\binom{\gA}{\gA} \medskip \cr
\nu-\alpha_{i} - \alpha_{i+f} & > & 0 & \text{sinon.}
\end{array}$$
Dans tous les cas, nous avons bien $\nu-\alpha_{i} - \alpha_{i+f} > 0$ 
puisque $\frac{p-2}{p} \nu-\epsilon_i - \epsilon_{i+f} = 
(p-2)(1+\cdots+p^{f-2})-2 > 0$ pour $p\geq 5$.
De la même manière, nous démontrons que si $\gB$ est dominant dans
$\binom{X_i}{X_{i+f}}$, alors $\alpha_i + \alpha_{i+f} > \nu$. Les
équivalences de \emph{ii)} en découlent.
\end{proof}

Rappelons que le jeu d'équations \emph{(B)} s'obtient en écrivant
que les coefficients de degré strictement inférieur à $\nu$ dans les
quatre expressions suivantes s'annulent :
\begin{itemize}
\item[$\bullet$] $A_i = a_{i} a'_{i+1} u^{\alpha'_{i+1+f} + p \alpha_{i} + h_{i}} - b_i b'_{i+1}u^{\alpha'_{i+1} + p \alpha_{i+f} + h_{i+f}}, $
\item[$\bullet$] $B'_i = b'_{i} a'_{i+1} u^{\alpha'_{i+1+f} + p \alpha'_{i} + h_{i}} - a'_i b'_{i+1}u^{\alpha'_{i+1} + p \alpha'_{i+f} + h_{i+f}}, $ 
\item[$\bullet$] $B_i = - a_{i} b_{i+1} u^{\alpha_{i+1+f} + p \alpha_{i} + h_{i}} + b_i a_{i+1}u^{\alpha_{i+1} + p \alpha_{i+f} + h_{i+f}}, $ 
\item[$\bullet$] $A'_i = - b'_{i} b_{i+1} u^{\alpha_{i+1+f} + p \alpha'_{i} + h_{i}} + a'_i a_{i+1}u^{\alpha_{i+1} + p \alpha'_{i+f} + h_{i+f}}. $
\end{itemize}
Nous remarquons que toutes les puissances de $u$ qui apparaissent dans 
les coefficients diagonaux $A_i$ et $A'_i$ sont (positives et) congrues 
à $\nu$ modulo $e$. Ces puissances sont donc supérieures ou égales à 
$\nu$ et les coefficients $A_i$ et $A'_i$ n'apportent donc aucune 
contrainte. Pour $B_i$ et $B'_i$, nous aurons besoin du lemme suivant.

\begin{lem}\label{lemcasequa}
Pour $i\in\llbracket0, f-1\rrbracket$, nous avons :

\smallskip

\begin{itemize}
\item[$\bullet$]  
$\alpha_{i+1+f} + p \alpha_{i} + h_{i} < \nu$ ssi
$\gA$ est dominant en $(i+1)$ et $X_i\in\{\gA,\gAB\}$.\\
Lorsque c'est le cas, $X_{i+f} = \gO$ ou $\gA$ est dominant en $i$. \smallskip

\item[$\bullet$]
$\alpha_{i+1} + p \alpha_{i+f} + h_{i+f} < \nu$ ssi
$\gA$ est dominant en $(i+1)$ et $X_{i+f}\in\{\gA,\gAB\}$.\\
Lorsque c'est le cas, $X_i = \gO$ ou $\gA$ est dominant en $i$. \smallskip

\item[$\bullet$] 
$\alpha'_{i+1+f} + p \alpha'_{i} + h_{i} < \nu$ ssi
$\gB$ est dominant en $(i+1)$ et $X_{i+f}\in\{\gB,\gO\}$.\\
Lorsque c'est le cas, $X_i = \gO$ ou $\gB$ est dominant en $i$. \smallskip

\item[$\bullet$]
$\alpha'_{i+1} + p \alpha'_{i+f} + h_{i+f} < \nu$ ssi
$\gB$ est dominant en $(i+1)$ et $X_i \in\{\gB,\gO\}$.\\
Lorsque c'est le cas, $X_{i+f} = \gO$ ou $\gB$ est dominant en $i$. \smallskip
\end{itemize}

\noindent
De plus, dans chacun des quatre cas précédents, nous avons
$X_{i+1} \neq \gO$ et $X_{i+1+f} \neq \gO$.
\end{lem}

\begin{proof}
Supposons $\alpha_{i+1+f} + p \alpha_{i} + h_{i} < \nu$. Alors $p 
\alpha_{i} + h_{i} < \nu$ et, par suite, $\alpha_{i} < \nu/p < 
\nu_{\rhobar,i}$. Ainsi $X_i\in\{\gA,\gAB\}$. 
De plus $\alpha_{i+1}\equiv p \alpha_{i} + h_{i} \pmod e$ et, donc, 
d'après le lemme \ref{lem:relalphai}, $ \alpha_{i+1}=p \alpha_{i} + 
h_{i}<\nu$. Ainsi $\alpha_{i+1} + \alpha'_{i+1+f} = \alpha'_{i+1} + 
\alpha_{i+1+f}=\nu$ et $\alpha'_{i+1} - \alpha_{i+1} = \nu - 
\alpha_{i+1+f} - \alpha_{i+1}>0$. Donc $\gA$ est dominant en $(i+1)$
d'après le lemme \ref{lemgenereseau}.
Les autres résultats sont analogues, en constatant qu'un calcul 
similaire à celui du lemme \ref{lem:relalphai} montre que les 
coefficients $(\alpha'_{i})_{i\geq 0}$, définis par la congruence 
(\ref{alpha'}) satisfont 
$p\alpha'_{i+f}+h_{i}=\alpha'_{i+1+f}$ si et seulement si
$X_{i+f}\in\{\gO,\gB\}$.
\end{proof}

Supposons à présent qu'il apparaisse dans $B_i$ un coefficient de degré 
strictement inférieur à $\nu$. Alors, $\alpha_{i+1+f} + p \alpha_{i} + 
h_{i} < \nu$ ou $\alpha_{i+1} + p \alpha_{i+f} + h_{i+f} < \nu$ et, par 
suite, d'après les deux premières assertions du lemme \ref{lemcasequa}, 
$\gA$ est dominant en $(i+1)$. Nous remarquons, de plus, que le 
coefficient devant $u^{\alpha_{i+1+f} + p \alpha_{i} + h_{i}}$ (resp. 
devant $u^{\alpha_{i+1} + p \alpha_{i+f} + h_{i+f}}$) s'annule si 
$X_{i+f} = \gO$ (resp. $X_i = \gO$). En utilisant à nouveau le lemme 
\ref{lemcasequa}, nous déduisons que $\gA$ est aussi dominant en $i$.
Ces conclusions, combinées au lemme \ref{proprelgene} et au fait que
le couple d'allèles $\binom \gO \gO$ ne peut apparaître, restreignent les
possibilités pour $\binom {X_i}{X_{i+f}}$ de la manière suivante :
$$\binom{X_{i}}{X_{i+f}} \in \left\{
\binom{\gA}{\gA}, \binom{\gA}{\gB}, \binom{\gB}{\gA},
\binom{\gO}{\gA}, \binom{\gA}{\gO}
\right\}.$$
Le tableau ci-dessous présente, pour chacune des valeurs possibles de 
$\binom {X_i}{X_{i+f}}$, l'équation donnée par $B_i$ ainsi que la 
décoration entre les positions $i$ et $i+1$. (Remarquons que, si les 
deux puissances de $u$ qui apparaissent dans $B_i$ sont strictement 
inférieures à $\nu$, alors elles sont nécessairement égales car congrues 
modulo $e$.)

\newcommand{\tikzcross}[1]{
\begin{tikzpicture}[scale=#1]
\draw[thick] (0,0)--(1,1);
\draw[thick] (0,1)--(1,0);
\end{tikzpicture}}
\newcommand{\tikzne}[1]{
\begin{tikzpicture}[scale=#1]
\draw[thick] (0,0)--(1,1);
\end{tikzpicture}}
\newcommand{\tikzse}[1]{
\begin{tikzpicture}[scale=#1]
\draw[thick] (0,1)--(1,0);
\end{tikzpicture}}

\begin{center}
\renewcommand{\arraystretch}{1.5}
\noindent
\begin{tabular}{| c | c | c |c| c|c|}
\hline
$\binom{X_i}{X_{i+f}}$& Équation& Décoration \\ \hline
$\binom{\gA}{\gA}$ & $a_i b_{i+1} =  a_{i+1} b_i$ & \tikzcross{0.25} \\ \hline
$\binom{\gA}{\gO}$ & $b_{i+1} = 0 $  & \tikzse{0.25} \\ \hline
$\binom{\gA}{\gB}$ & $a_i b_{i+1} = 0$ & \tikzse{0.25} \\ \hline
$\binom{\gO}{\gA}$ & $a_{i+1} = 0$ & \tikzne{0.25} \\ \hline
$\binom{\gB}{\gA}$ & $a_{i+1} b_i = 0$ & \tikzne{0.25} \\ \hline
\end{tabular}
\renewcommand{\arraystretch}{1}
\end{center}

\noindent
En gardant à l'esprit qu'étant donné que $\gA$ est dominant en $i$ et 
$i+1$, les variables $a_i$, $b_i$, $a_{i+1}$ et $b_{i+1}$ correspondent 
respectivement à $x_i$, $x_{i+f}$, $x_{i+1}$ et $x_{i+f+1}$, nous 
constatons que les équations provenant de $B_i$ sont un sous-ensemble 
de celles des jeux \emph{(B1)}, \emph{(B2)} et \emph{(B3)}.

Réciproquement, considérons une équation $E$ donnée par l'un des trois 
jeux précédents correspondant à une décoration entre les positions $i$ 
et $(i+1)$ où $\gA$ est dominant. Si $X_i \in \{\gA, \gAB\}$, la 
première équivalence du lemme \ref{lemcasequa} nous apprend que 
$\alpha_{i+1+f} + p \alpha_{i} + h_{i} < \nu$, d'où nous déduisons que 
l'équation $E$ est bien une équation de la variété $\mathcal K(\ttt, 
\rhobar)$. Nous concluons de la même manière si $X_{i+f} \in \{\gA, 
\gAB\}$ en utilisant la deuxième équivalence du lemme \ref{lemcasequa}. 
Enfin, nous remarquons que ces deux cas regroupent toutes les 
possibilités puisque nous avons supposé que $\gA$ est dominant en $i$.
Nous avons ainsi démontré que les équations provenant de $B_i$ 
coïncident exactement avec les équations des jeux \emph{(B1)}, 
\emph{(B2)} et \emph{(B3)} qui correspondent à une décoration entre
deux positions où $\gA$ est dominant.

Exactement de la même manière, nous démontrons que les équations
correspondant à une décoration entre deux positions où $\gB$ est
dominant coïncident avec celles provenant $B'_i$. Étant donné que,
par définition, les décorations sur $X$ n'existent qu'entre deux
indices $i$ et $(i+1)$ ayant même allèle dominant, le théorème
\ref{thequationsKisin} est démontré.

\section{Géométrie de la variété de Kisin}
\label{sec:Connexite} \label{sec:Geometrie}

Le théorème \ref{thequationsKisin} présente les variétés de Kisin
$\vK^\psi(\vv, \ttt, \rhobar)$ comme des sous-variétés de $\prod_{i=0}^{f-1} \P^1_{k_E}$ définies
par des équations explicites. Le but de ce paragraphe (\S \ref{sec:Geometrie}) est d'expliquer
comment appréhender la géométrie des
variétés de Kisin qui n'est pas si simple en général. En effet, les variétés $\vK^\psi(\vv, \ttt,
\rhobar)$ peuvent avoir de multiples composantes irréductibles de
dimension variable.

Dans toute cette partie, la représentation $\rhobar$ et le type 
galoisien $\ttt$ sont fixés, ainsi que des entiers $h$, $\gamma$ et 
$\gamma'$ qui les représentent. Comme précédemment, nous notons $(X_i)_{i 
\in \Z}$ le gène décoré associé à ces données. 

\subsection{Réduction des équations de la variété de Kisin}
\label{ssec:reduction}\label{paraKisinvide}

Comme nous l'avons déjà vu sur l'exemple du \S \ref{sssec:filrouge}, le 
système d'équations donné par le théorème \ref{thequationsKisin} se 
simplifie par des considérations successives immédiates dont nous 
faisons maintenant la liste :
\begin{enumerate}[i)]
\item si deux {\it colonnes} consécutives du gène   sont reliées par deux liens (en forme 
de croix), les points projectifs correspondant sont égaux et nous 
pouvons fusionner les deux colonnes en question ;
\item si la valeur de la variable $x_i$ est imposée égale à $0$, nous
pouvons supposer, sans perte de généralité, que $x_{i+f}$ vaut $1$ ;
\item nous pouvons supprimer tout lien partant ou arrivant sur une
variable dont nous savons déjà qu'elle est contrainte égale à $0$ ;
\item si un lien relie les variables $x_i$ et $x_j$ et que nous
savons déjà que $x_i$ est contrainte égale à $1$, nous pouvons en
déduire que $x_j$ s'annule et donc ajouter l'équation
$x_j = 0$, ce qui permet par suite de supprimer le lien entre $x_i$
et $x_j$.
\end{enumerate}
La règle \emph{i)} permet de supposer qu'il n'y a plus de lien en forme 
de croix. Cette étape ne conduit à aucune contradiction car deux symboles 
$X_i=X_{i+f+1} =\gO$ qui imposent les équations $x_i=x_{i+f+1}=0$ n'ont 
pas de lien entre eux.
La règle \emph{ii)} induit une contradiction si $X_i=X_{i+f}=\gO$, cas 
pour lequel la variété de Kisin est vide. 
Après la réduction \emph{iv)}, il peut être nécessaire d'itérer 
l'ensemble des réductions \emph{ii)} à \emph{iv)} jusqu'à ce que le 
procédé se stabilise. Par exemple, si nous plaçons un $0$ en $x_i$, 
celui-ci implique un $1$ en $x_{i+f}$ (par l'étape \emph{ii)}) qui, à son tour, 
implique $x_{i+1}=0$ s'il y a un lien entre $X_{i+f}$ et 
$X_{i+1}$.

\begin{ex}
Nous avons déjà expliqué au \S \ref{sssec:filrouge} comment appliquer 
les règles ci-dessus sur un exemple. Nous présentons ci-après un autre 
exemple qui montre que les simplifications peuvent se propager de façon 
importante. Voici son gène\footnote{Il correspond aux paramètres $h = 
\frac{p-1} 2 + (p-1) p^3 + (p-1) p^5 + p^8 + p^9 + p^{10}$, $\gamma = - 
1 - p^4 - p^5 - p^7$ et $\gamma' = - \frac{p+1} 2 - p^2 - 2 p^3$.} :

\begin{center}
\begin{tikzpicture}[scale=0.8]
\draw [fill=fond, thick] (-0.5,-0.5) rectangle (10.5,1.5);
\draw[thick,->>] (-0.5,-0.5)--(-0.5,0.3);
\draw[thick,->>] (10.5,1.5)--(10.5,0.7);
\node[scale=0.5] at (0, 1.8) { $0$ };
\node[scale=0.5] at (1, 1.8) { $1$ };
\node[scale=0.5] at (2, 1.8) { $2$ };
\node[scale=0.5] at (3, 1.8) { $3$ };
\node[scale=0.5] at (4, 1.8) { $4$ };
\node[scale=0.5] at (5, 1.8) { $5$ };
\node[scale=0.5] at (6, 1.8) { $6$ };
\node[scale=0.5] at (7, 1.8) { $7$ };
\node[scale=0.5] at (8, 1.8) { $8$ };
\node[scale=0.5] at (9, 1.8) { $9$ };
\node[scale=0.5] at (10, 1.8) { $10$ };
\node[color=colA] at (0, 0) { $\gB$ };
\node[color=colA] at (0, 1) { $\gA$ };
\node[color=colA] at (1, 0) { $\gB$ };
\node[color=colA] at (1, 1) { $\gA$ };
\node[color=colA] at (2, 0) { $\gB$ };
\node[color=colA] at (2, 1) { $\gA$ };
\node[color=colA] at (3, 0) { $\gA$ };
\node[color=colA] at (3, 1) { $\gA$ };
\node[color=colA] at (4, 0) { $\gB$ };
\node[color=colA] at (4, 1) { $\gA$ };
\node[color=colA] at (5, 0) { $\gA$ };
\node[color=colA] at (5, 1) { $\gAB$ };
\node[color=colA] at (6, 0) { $\gA$ };
\node[color=colA] at (6, 1) { $\gO$ };
\node[color=colB] at (7, 0) { $\gB$ };
\node[color=colB] at (7, 1) { $\gO$ };
\node[color=colB] at (8, 0) { $\gB$ };
\node[color=colB] at (8, 1) { $\gB$ };
\node[color=colB] at (9, 0) { $\gB$ };
\node[color=colB] at (9, 1) { $\gAB$ };
\node[color=colA] at (10, 0) { $\gA$ };
\node[color=colA] at (10, 1) { $\gO$ };
\draw[color=colA,thick] (-0.8,0.8)--(-0.2,0.2);
\draw[color=colA,thick] (0.2,0.8)--(0.8,0.2);
\draw[color=colA,thick] (1.2,0.8)--(1.8,0.2);
\draw[color=colA,thick] (2.2,0.8)--(2.8,0.2);
\draw[color=colA,thick] (3.2,0.8)--(3.8,0.2);
\draw[color=colA,thick] (3.2,0.2)--(3.8,0.8);
\draw[color=colA,thick] (4.2,0.8)--(4.8,0.2);
\draw[color=colA,thick] (5.2,0.2)--(5.8,0.8);
\draw[color=colB,thick] (7.2,0.2)--(7.8,0.8);
\draw[color=colB,thick] (8.2,0.8)--(8.8,0.2);
\draw[color=colB,thick] (8.2,0.2)--(8.8,0.8);
\draw[color=colA,thick] (10.2,0.2)--(10.8,0.8);
\end{tikzpicture}
\end{center}

\noindent
Le lien qui apparaît entre la colonne $10$ et la colonne $0$ impose la
valeur de $[x_0 : x_{11}]$ à
$[1:0]$. Mais, de proche en proche, cela entraîne que $[x_i, x_{i+11}]
= [1:0]$ pour $i \in \llbracket 0, 5 \rrbracket$. Pour $i \in \{6,7,
10\}$, la valeur de $[x_i, x_{i+11}]$ est imposée, égale à $[0:1]$ par
la présence de l'allèle $\gO$. Enfin, les liens entre les colonnes $7$
et $8$ et les colonnes $8$ et $9$ conduisent à $[x_8 : x_{19}] =
[x_9 : x_{20}] = [0,1]$. Ainsi la valeur de tous les points projectifs
$[x_i, x_{i+11}]$ est uniquement déterminée (sans qu'il y ait de
contradiction) : la variété de Kisin correspondante est donc réduite à un 
unique point $\{[0:1]\}^6\times\{[1:0]\}^5$.
\end{ex}

Lorsque le gène ne contient pas de couple $\binom{\gO}{\gO}$, ce procédé de simplification ne conduit à aucune 
contradiction (une variable à la fois contrainte égale à $0$ et égale à 
$1$). En effet si, d'une part, le gène $X$ ne contient pas de $\gO$, la réduction se limite à l'étape i) (de fusion des colonnes reliées par deux liens).
D'autre part, si le gène contient un $\gO$, quitte à changer le 
plongement $\tau_0$, nous pouvons supposer que $X_0=\gO$ ou $X_{f}=\gO$. 
Si $X_0=\gO$, d'après le lemme \ref{proprelgene}, $X_{2f-1} \in 
\{\gO,\gAB\}$ et il n'y a pas de lien entre $X_{2f-1}$ et $X_f$. Donc les 
modifications ne se propagent pas de $\binom{X_{f-1}}{X_{2f-1}}$ à 
$\binom{X_0}{X_f}$. Le cas $X_{f}=\gO$ est analogue.
Enfin le lemme \ref{lem:tjsdroite} ci-dessous garantit que les 
modifications ii) à iv) se propagent toujours vers la droite.  

\begin{lem}
\label{lem:tjsdroite}
Si $X_i = \gO$ pour un certain entier $i$, il n'y a pas de lien
entre $X_{i+f}$ et $X_{i-1}$.
\end{lem}

\begin{proof}
Sous l'hypothèse $X_i = \gO$, le lemme \ref{proprelgene} assure que 
$X_{i-1} \in \{\gAB, \gO\}$ d'où nous déduisons qu'aucun lien ne part de 
$X_{i-1}$.
\end{proof}

Une fois que la série des réductions \emph{i)} à \emph{iv)} est 
achevée, il est facile d'établir la condition suffisante et nécessaire de 
non vacuité de la variété de Kisin.

\begin{prop}\label{corvide}
La variété de Kisin $\vK^\psi(\vv_0,\ttt,\rhobar)$ est non vide si et 
seulement si pour tout $i$, $\binom{X_i}{X_{i+f}}\not=\binom{\gO}{\gO}$.
\end{prop}

\begin{proof}

 Notons $X'$ le gène obtenu à partir de $X$ après stabilisation de 
l'ensemble des réductions \emph{i)} à \emph{iv)}. Notons $x_i'$ les variables correspondantes au gène $X'$, elles forment un sous-ensemble des variables $x_i$ obtenu après l'étape \emph{i)} en fusionnant les colonnes reliées par des croix.
Pour obtenir un point 
de la variété de Kisin, il suffit de spécifier dans $X'$ les variables 
$x'_i$ encore \og libres \fg\ (en accord avec les équations liées aux liens 
restants et en évitant d'avoir simultanément $x'_i=0$ et $x'_{i+f}=0$) et, ensuite, de
se souvenir de l'ensemble des réductions effectuées à l'étape \emph{i)} (pour 
reconstruire un point compatible avec le gène $X$ à partir d'un point 
compatible avec le gène réduit $X'$).

Si $X$ ne contient pas de $\gO$, alors pour tout $i\in\llbracket 0,f-1 
\rrbracket$, nous spécifions $x'_i$ à la valeur $0$ (resp. $1$) et 
$x'_{i+f}$ à la valeur $1$ (resp. $0$) si $x'_{2f-1}$ et $x'_f$ (resp. 
$x'_{f-1}$ et $x'_0$) ne sont pas liés. Ceci définit un point de la 
variété de Kisin.

Si $X$ contient un $\gO$, quitte à changer le plongement 
initial, nous pouvons supposer que $X_0=\gO$ ou $X_{f-1}=\gO$. Traitons 
le cas $X_0=\gO$. Le cas $X_{f-1}=\gO$ est analogue.
Pour la spécification des variables, nous partons de $x'_0=0$ et tant que 
tant que la variable $x'_{i+f}$ n'a pas une valeur contrainte à $0$, nous 
spécifions les variables libres $x'_i$ à la valeur $0$ et les variables 
libres $x'_{i+f}$ à la valeur $1$.
Si pour tout $i\in\llbracket 1,f-1\rrbracket$, $x'_{i+f}=0$, nous 
spécifions ainsi toutes les variables libres du gène décoré $X'$. Sinon 
pour l'indice $i_0\in\llbracket 1,f-1\rrbracket$ minimum tel que 
$x'_{i_0+f}=0$, nous posons $x'_{i_0}=1$. Puis nous procédons de même à 
partir de $x'_{i_0+f}=0$ en spécifiant la valeur des variables libres 
$x'_{i_0+i+f}$ à 0 (resp. $x'_{i_0+i}$ à 1) tant que $x'_{i_0+i}\neq0$.
Nous recommençons ensuite le même processus à partir de $x'_{i_0+i_1}=0$
où $i_1$ est l'indice minimum pour lequel $x'_{i_0+i_1}=0$, et ainsi de
suite. Nous parcourons ainsi tout le gène, en spécifiant successivement 
les valeurs des variables libres à $1$ ou à $0$.

L'élément ainsi construit est compatible avec les liens. En effet, d'une 
part, après réduction, les liens n'apparaissent plus qu'entre des 
couples de variables libres (car les règles \emph{ii)} et \emph{iv)} 
permettent de supprimer tous les liens dont la valeur d'une extrémité 
est spécifiée à $1$ ou $0$). D'autre part, si nous avons deux variables 
libres $x'_i$ et $x'_{i+f+1}$, alors les deux autres variables 
$x'_{i+f}$ et $x'_{i+1}$ sont libres et l'affectation des valeurs 
$x'_i=x'_{i+1}=0$ (ou $x'_{i+f}=x'_{i+f+1}=0$) est compatible avec tout 
lien éventuel.
\end{proof}

\begin{rem} \label{allelesaboucle}
Dans le cas où le gène initial $X$ ne contient pas d'occurence $\gO$, il 
est possible que le gène réduit $X'$ (introduit dans la démonstration de 
la proposition \ref{corvide}) soit constituté d'une unique suite 
de couples de variables libres, chacun étant relié à son successeur
par un ou plusieurs liens, de façon à ce que la séquence 
d'allèles correspondantes présente une boucle. Dans la suite (\S 
\ref{paradimension}), nous détaillons le cas général d'une séquence 
d'allèles sans boucle en indiquant comment adapter les constructions et 
les démonstrations au cas à boucle.
\end{rem}

\subsection{Composantes irréductibles et dimension}
\label{ssec:compirred}\label{paradimension}

Après avoir effectué la réduction du \S \ref{ssec:reduction}, nous 
pouvons découper le gène réduit $X'$ en morceaux de taille minimale (sans 
couper de lien !) ne comprenant qu'un couple de variables à valeurs 
spécifiées ($\binom{0}{1}$ ou $\binom{1}{0}$) ou qu'une suite (dite 
séquence d'allèles) de couples de variables libres tels que deux couples 
successifs sont reliés par un lien.

La variété de Kisin $\vK^\psi(\vv_0, \ttt, \rhobar)$ s'écrit alors comme 
un produit de variétés, dites \emph{variétés facteur}, correspondant
aux séquences d'allèles obtenues à partir de la réduction du gène $X$ 
décoré associé à $\rhobar$ et $\tau$.

Une séquence d'allèles de longueur $\ell$ est de la forme :
\begin{center}
\begin{tikzpicture}[xscale=1.2,yscale=0.8]
\draw[thick,fill=fond] (-0.6,-0.5) rectangle (6.4,1.5);
\node at (0, 0) { \ph $z_{1}$ };
\node at (0, 1) { \ph $y_1$ };
\node at (1, 0) { \ph $z_{2}$ };
\node at (1, 1) { \ph $y_{2}$ };
\node at (2, 0) { \ph $z_{3}$ };
\node at (2, 1) { \ph $y_{3}$ };
\node at (3, 0) { \ph $z_{4}$ };
\node at (3, 1) { \ph $y_{4}$ };
\node at (4, 0) { \ph $\cdots$ };
\node at (4, 1) { \ph $\cdots$ };
\node at (5, 0) { \ph $z_{\ell-1}$ };
\node at (5, 1) { \ph $y_{\ell-1}$ };
\node at (6, 0) { \ph $z_{\ell}$ };
\node at (6, 1) { \ph $y_{\ell}$ };
\draw[thick] (0.2,0.7)--(0.8,0.2);

\draw[thick] (1.2,0.2)--(1.8,0.7);
\draw[thick] (2.2,0.2)--(2.8,0.7);
\draw[thick] (3.2,0.7)--(3.8,0.2);
\draw[thick] (4.2,0.7)--(4.8,0.2);
\draw[thick] (5.2,0.2)--(5.8,0.7);
\end{tikzpicture}
\end{center}

\noindent
où les variables $y_i$ et les $z_i$ correspondent à certains $x_i$.  Dans le cas où le gène initial $X$ ne contient pas d'occurence $\gO$,  il est possible qu'il y ait un lien supplémentaire (dit boucle) entre $y_1$ et $y_\ell$ ou entre $z_1$ et $z_\ell$ (voir remarque \ref{allelesaboucle}). \\

Soit $\mathcal V$ une variété facteur correspondant à une séquence 
d'allèles fixée (sans boucle) de longueur $\ell$. Notons $j_1, \ldots, j_n$ les indices 
(s'ils existent)  de $\llbracket 2,\ell -1\rrbracket$ présentant une 
alternance de pente. S'il n'y a pas d'alternance de pente, nous posons $n=0$ et $j_1=\ell$.
Posons $j_0=1$ et $j_{n+1}=\ell$ et définissons les intervalles :
$$
I_m = [j_m, j_{m+1}] \text{ pour } m \in \llbracket 0, n \rrbracket.$$
Nous convenons de plus que $I_{n+1}=\emptyset$.
\begin{definit}
Une partie $S \subset \llbracket 1, \ell \rrbracket$ est dite
\emph{exprimable}  si 
pour tout $m \in \llbracket 0, n+1 \rrbracket$, l'intersection $S \cap
I_m$ est de cardinal au plus $1$.
\end{definit}
Une partie exprimable $S=\{r_1,\ldots,r_s\}\subset\llbracket 
1,\ell\rrbracket$ définit une sous-variété ${\mathcal V}_S$ de la 
variété facteur $\mathcal V$ isomorphe à $(\P^1_{k_E})^s$ de la façon suivante.
Supposons que $y_1$ et $z_2$ sont liés. Alors, en convenant que
$[y_0:z_0]=[0:1]$, la sous-variété ${\mathcal V}_S$ est l'ensemble des 
$([y_i:z_i])_{1\leq i\leq \ell}$ avec, tout d'abord :
\begin{itemize}
\item[$\bullet$] $[y_{i}:z_{i}]\in\P^1_{k_E}$ si $i\in S$,
\item[$\bullet$] $[y_{i}:z_{i}]=[0:1]$ si $1\leq i<r_1,$
\end{itemize}
puis, pour $u$ prenant successivement les valeurs $1, \ldots, s-1$ :
\begin{itemize}
\item[$\bullet$] si $r_u\in\{j_1,\ldots,j_{n}\}$, alors
$[y_{i}:z_{i}]=[y_{r_u-1}:z_{r_u-1}]$ pour $i\in\llbracket r_u+1,r_{u+1}-1\rrbracket$
\item[$\bullet$] si $r_u\not\in\{j_1,\ldots,j_n\}$, alors
pour $i\in\llbracket r_u+1,r_{u+1}-1\rrbracket$ :
     $$[y_{i}:z_{i}]=\left\{\begin{array}{rcl} [0:1]&\mbox{ si }&[y_{r_u-1}:z_{r_u-1}]=[1:0]\cr
 [1:0]&\mbox{ si }&[y_{r_u-1}:z_{r_u-1}]=[0:1]\cr\end{array}\right.$$
\end{itemize}
et, enfin :
\begin{itemize}
\item[$\bullet$] si $r_s \in\{j_1,\ldots,j_n\}$, 
$[y_{i}:z_{i}]=[y_{r_u-1}:z_{r_u-1}]$ pour $i\in\llbracket r_s+1,\ell \rrbracket$
\item[$\bullet$] si $r_s\not\in\{j_1,\ldots,j_n\}$, alors
pour $i\in\llbracket r_s+1,\ell\rrbracket$ :
     $$[y_{i}:z_{i}]=\left\{\begin{array}{rcl} [0:1]&\mbox{ si }&[y_{r_s-1}:z_{r_s-1}]=[1:0]\cr
 [1:0]&\mbox{ si }&[y_{r_s-1}:z_{r_s-1}]=[0:1].\cr\end{array}\right.$$
\end{itemize}
La construction de ${\mathcal V}_S$ est analogue si $z_{1}$ et $y_{2}$ sont liés, en initialisant 
les valeurs par  $[y_{r}:z_{r}]=[1:0]$ pour $-1\leq r<r_1$.\\

Autrement dit, ${\mathcal V}_S$ s'obtient comme la sous-variété de $(\P^1_{k_E})^{\ell}$ produit de $\P^1_{k_E}$ aux indices $i\in S$ et des valeurs fixées à $[0:1]$ ou $[1:0]$ aux autres indices. L'alternance des valeurs fixées et le caractère exprimable de $S$ garantissent l'inclusion ${\mathcal V}_S\subset{\mathcal V}$.

\begin{rem} Dans le cas d'une séquence d'allèles à boucle (voir remarque \ref{allelesaboucle}), il convient d'adapter la définition de partie exprimable $S\subset\llbracket 1,\ell\rrbracket$ à la présence d'un lien supplémentaire, \emph{i.e.} $S\cap I_{n+1}$ est de cardinal au plus 1 pour 
$$I_{n+1}=\left\{\begin{array}{lll} 
\{1,\ell\}&\mbox{ s'il y a une alternance de pente en $1$ et $\ell$,}\cr
\{1\}\cup[j_n,\ell]&\mbox{ s'il y a une alternance de pente en $1$, mais pas en $\ell$,}\cr
\{\ell\}\cup[1,j_1]&\mbox{ s'il y a une alternance de pente en $\ell$, mais pas en 1,}\cr
[1,j_1]\cup[j_n,\ell]&\mbox{ s'il n'y a pas d'alternance de pente en $1$, ni en $\ell$. }\cr
\end{array}\right.$$
La construction de la sous-variété ${\mathcal V}_S$ associée à une partie exprimable $S$ est analogue, avec une initialisation adaptée des valeurs des variables pour $0\leq r<r_1$. 
\end{rem}

\begin{prop}
\label{prop:dimV}
Les composantes irréductibles de $\mathcal V$ sont en bijection avec les 
parties exprimables $S$ maximales pour l'inclusion.

La dimension de $\mathcal V$ est égale au cardinal maximal d'une partie 
exprimable et est toujours inférieure ou égale à $\frac {\ell+1} 2$.
\end{prop}

\begin{proof}
Un examen des équations définissant $\mathcal V$ montre que $\mathcal V$
est égale à la réunion des $\mathcal V_S$ pour $S$ parcourant l'ensemble
des parties exprimables maximales de $\llbracket 0, \ell \rrbracket$. Par
ailleurs, les variétés $\mathcal V_S$ sont toutes irréductibles car elles
s'identifient à des produits de copies de $\P^1_{k_E}$. Il résulte de ceci que
les composantes irréductibles de $\mathcal V$ sont à compter parmi les
$\mathcal V_S$. Par ailleurs, aucune des variétés $\mathcal V_S$ (pour
$S$ parcourant l'ensemble des parties exprimables maximales) n'est 
incluse dans une autre à cause de la condition de maximalité. Nous en
déduisons la première assertion de la proposition. 

Le fait que la dimension de $\mathcal V$ soit égale au cardinal maximal 
d'une partie exprimable de $\llbracket 1, \ell \rrbracket$ est alors 
évident après avoir remarqué qu'une partie exprimable de cardinal maximal 
est \emph{a fortiori} maximale pour l'inclusion. Enfin la majoration
$\dim \mathcal V \leq \frac {\ell + 1} 2$ résulte du fait qu'une partie
exprimable de $\llbracket 1, \ell \rrbracket$ ne peut pas contenir deux
entiers consécutifs.
\end{proof}

\begin{ex}
Examinons le cas où l'ensemble de valeurs prises par les $j_m$ est 
$\llbracket 1, \ell \rrbracket$, c'est-à-dire, graphiquement, où les 
liens montent et descendent alternativement :

\begin{center}
\begin{tikzpicture}[xscale=1.2,yscale=0.8]
\draw [thick, fill=fond] (-0.5,-0.5) rectangle (6.3,1.5);
\node at (0,1) { \ph$y_1$ };
\node at (1,1) { \ph$y_2$ };
\node at (2,1) { \ph$y_3$ };
\node[scale=1.5] at (3.25,1) { \ph$\cdots$ };
\node at (4.5,1) { \ph$y_{\ell-1}$ };
\node at (5.7,1) { \ph$y_{\ell}$ };
\node at (0,0) { \ph$z_1$ };
\node at (1,0) { \ph$z_2$ };
\node at (2,0) { \ph$z_3$ };
\node[scale=1.5] at (3.25,0) { \ph$\cdots$ };
\node at (4.5,0) { \ph$z_{\ell-1}$ };
\node at (5.7,0) { \ph$z_{\ell}$ };
\draw[thick] (0.25,0.2)--(0.75,0.7);
\draw[thick] (1.25,0.7)--(1.75,0.2);
\draw[thick] (2.25,0.2)--(2.75,0.7);
\draw[thick] (3.75,0.7)--(4.25,0.2);
\draw[thick] (4.85,0.2)--(5.35,0.7);
\end{tikzpicture}
\end{center}

\noindent
Dans cette situation, une partie de $\llbracket 1, \ell \rrbracket$ est 
exprimable si et seulement si elle ne contient pas deux entiers 
consécutifs. Une partie exprimable de cardinal maximal est donc obtenue 
en ne gardant que les entiers impairs de $\llbracket 1, \ell \rrbracket$ ; 
elle compte $\lfloor \frac{\ell + 1} 2 \rfloor$ éléments. La dimension 
de la variété facteur $\mathcal V_\ell$ correspondante est donc, dans ce 
cas, égale à $\lfloor \frac{\ell + 1} 2 \rfloor$.

Nous constatons, par ailleurs, qu'au moins pour $\ell \geq 3$, il existe 
d'autres parties exprimables maximales ayant des cardinaux moindres. Un 
exemple est donné par la partie $S \subset \llbracket 1, \ell \rrbracket$ 
qui regroupe les entiers congrus à $2$ modulo $3$. Cette partie a pour 
cardinal $\lfloor \frac{\ell + 1} 3 \rfloor$ ; la variété $\mathcal 
V_\ell$ possède donc au moins une composante irréductible de cette 
dimension.

Nous pouvons aisément être plus précis et dénombrer les
composantes irréductibles de $\mathcal V_\ell$. Pour ce faire,
posons :
$$P_\ell(X) = \sum_{d=0}^{\dim \mathcal V_\ell} a_{\ell, d} X^d$$
où $a_{\ell,d}$ est la nombre de composantes irréductibles de $\mathcal
V_\ell$ de dimension $d$. En remarquant qu'une partie exprimable 
maximale de $\llbracket 1, \ell \rrbracket$ contient nécessairement
soit $1$, soit $2$, nous obtenons la relation de récurrence suivante :
\begin{equation}
\label{eq:recPl}
P_\ell(X) = X \cdot \big( P_{\ell-2}(X) + P_{\ell-3}(X) \big).
\end{equation}
Par ailleurs, les premiers termes de la suite des $P_\ell(X)$ sont
aisés à calculer : nous trouvons $P_0(X) = 1$, $P_1(X) = X$ et $P_2(X) 
= 2X$. Les $P_\ell(X)$ suivants s'obtiennent alors en itérant 
\eqref{eq:recPl}. Par exemple, nous trouvons :
$$P_{20}(X) = 11 \cdot X^{10} + 120 \cdot X^9 + 126 \cdot X^8 + 
8 \cdot X^7$$
et constatons que la géométrie des variétés
$\mathcal V_\ell$ --- et donc, par corollaire, celle des variétés de 
Kisin --- peut être assez complexe. Remarquons, pour conclure, que la 
valeur $P_\ell(1)$ donne le nombre de composantes irréductibles de 
$\mathcal V_\ell$ ; une étude simple permet de montrer qu'elle se 
comporte asymptotiquement comme $c \cdot \alpha^\ell$ où $c$ est une 
constante strictement positive et $\alpha > 1$ est solution de $\alpha^3 
= \alpha + 1$.
\end{ex}

Pourvu que nous ayons gardé trace des réductions successives permettant 
de passer du gène décoré $X$ associé à $\rhobar$ et à $\ttt$ aux 
différentes séquences d'allèles qui composent le gène décoré $X'$ obtenu 
par réduction, nous déduisons immédiatement de la proposition 
\ref{prop:dimV} les composantes irréductibles (produit des composantes 
irréductibles des variétés facteur) et la dimension de la variété de 
Kisin $\vK^\psi(\vv_0, \ttt, \rhobar)$.

\subsection{Connexité des variétés de Kisin}
\label{paraKisinconnexe}

L'objectif de cette partie est de démontrer la connexité des variétés de 
Kisin $\vK^\psi(\vv_0, \ttt, \rhobar)$. D'après la construction effectuée au paragraphe précédent,  il suffit d'établir la connexité des variétés {\it facteur}
$\mathcal V$ du \S \ref{ssec:compirred} dont nous reprenons 
intégralement les notations. La proposition suivante est la clé de la 
démonstration. 

\begin{prop}
\label{prop:interirred}
Soient $S$ (resp. $T$) une partie exprimable maximale de $\llbracket
1, \ell \rrbracket$ et $\mathcal V_S$ (resp. $\mathcal V_T$) la
composante irréductible de $\mathcal V$ qui lui correspond.
Alors $\mathcal V_S \cap \mathcal V_T$ est non vide si et seulement
si pour tout $m \in \llbracket 0, n+1 \rrbracket$, $a \in S \cap I_m$ 
et $b \in T \cap I_m$, nous avons {\rm (}$|a-b| \leq 1${\rm )} ou 
{\rm (}$|a-b|=\ell-1$ et $m=n+1${\rm )}.
\end{prop}

\begin{proof}
La proposition est claire si la séquence d'allèles définissant $\mathcal V$ ne contient pas d'alternance de pente (voir \S \ref{exchaineP1}). Le résultat général s'obtient en travaillant successivement sur chaque fragment de diagrammes d'indice $I_m$ pour $m\in\llbracket 0,n+1\rrbracket$.
\end{proof}

\begin{cor}
La variété $\mathcal V$ est connexe.
\end{cor}

\begin{proof}
Étant donné deux parties exprimables maximales $S$ et $T$ de $\llbracket 
1, \ell \rrbracket$, il est facile de construire une suite $S_0, \ldots, 
S_d$ de parties exprimables maximales de $\llbracket 1, \ell \rrbracket$ telles que
$S_0 = S$, $S_d = T$ et le couple $(S_i, S_{i+1})$ vérifie la condition
de la proposition \ref{prop:interirred} pour tout $i$. Le corollaire
s'en déduit.
\end{proof}

\begin{cor}
\label{corcon}
La variété de Kisin $\vK^\psi(\vv_0, \ttt, \rhobar)$ est connexe.
\end{cor}

\begin{proof}
C'est immédiat étant donné que 
$\vK^\psi(\vv_0, \ttt, \rhobar)$ s'écrit comme un produit
de variétés facteurs $\mathcal V$ de la forme précédente.
\end{proof}
Notons que ce corollaire répond, dans notre cadre, à la conjecture~2.4.16 de \cite{Ki1}.

\section{Stratification et espace de déformations}
\label{sec:Stratification}

Le but de cette dernière partie est de donner des indications indiquant 
que l'anneau de déformations $R^\psi(\vv, \ttt, \rhobar)$ et la variété 
de Kisin $\vK^\psi(\vv, \ttt, \rhobar)$ sont deux objets intimement 
liés.

\subsection{Spécialisation et genre}

Soit $D^\psi(\vv,\ttt,\rhobar)$ l'espace rigide qui a pour anneau des 
fonctions $R^\psi(\vv, \ttt, \rhobar)[1/p]$. Le théorème 
\ref{th:isomgeneric} nous autorise à penser à $D^\psi(\vv,\ttt, 
\rhobar)$ également comme à la fibre générique de 
$\widehat{\GR}^\psi(\vv, \ttt, \rhobar)$. Nous prendrons garde toutefois 
au fait que cette dernière assertion n'a \emph{a priori} pas un sens 
précis étant donné que le schéma formel $\widehat{\GR}^\psi( 
\vv,\ttt,\rhobar)$ n'entre pas dans le cadre de la géométrie rigide 
classique car il n'est pas muni de la topologie $p$-adique (mais de la 
topologie $\m$-adique où $\m$ est l'idéal maximal de $R^\psi(\vv, \ttt, 
\rhobar)$).

Malgré tout, nous pouvons construire une application de spécialisation 
$$\sp : D^\psi(\vv,\ttt,\rhobar) \to \vK^\psi(\vv, \ttt, \rhobar)$$
qui est définie, au moins, de façon ensembliste sur les points fermés. 
Dans la suite, lorsque 
$E'$ est une extension finie de $E$,  $\oEp$ fait référence à son 
anneau des entiers, $\m_{E'}$ à son idéal maximal, $k_{E'} = \oEp / 
\m_{E'}$ à son corps résiduel et $\pi_{E'}$ à l'une de ses 
uniformisantes. Afin de construire le morphisme de spécialisation, nous 
remarquons en premier lieu que la donnée d'un point $x \in D^\psi(\vv,
\ttt,\rhobar)$ de corps résiduel $E'$ est équivalente à celle d'un 
couple $(\rho,\MK)$ où :
\begin{itemize}
\item[$\bullet$] $\rho$ est une classe d'isomorphisme de représentations 
$\rho : G_F \to \GL_2(\oEp)$ potentiellement Barsotti--Tate de type 
$(\vv,\ttt)$, de déterminant $\psi$ et dont la réduction modulo 
$\pi_{E'}$ est isomorphe à $\rhobar$, et
\item[$\bullet$] $\MK$ est un réseau de type $(\vv,\ttt)$ à 
l'intérieur de $\bbM(E' \otimes_{\oEp} \rho)$ (qui est uniquement 
déterminé). \smallskip
\end{itemize}
Le quotient $\MK / \pi_{E'} \MK$ est alors un réseau de type $(\vv, 
\ttt)$ dans $\bbM( \rhobar)$. Il apparaît donc comme un $k_{E'}$-point 
de $\vK^\psi(\vv, \ttt, \rhobar)$ et en détermine ainsi un point fermé. 
Cette construction définit une application ensembliste $\sp : 
D^\psi(\vv,\ttt,\rhobar) \to \vK^\psi(\vv, \ttt, \rhobar)$.

La théorie des genres introduite et développée dans \cite{Br} puis reprise 
dans \cite{CDM} permet de décrire l'image réciproque d'un point fermé de 
$\vK^\psi(\vv, \ttt,\rhobar)$ par l'application de spécialisation. Rappelons 
tout d'abord la définition du genre d'un module de Breuil--Kisin.
Soit $R = \oEp$ ou $R = k_{E'}$. Soit 
$$\MK = \MK^{(0)} \oplus \cdots \oplus \MK^{(f-1)}$$
 un module de Breuil--Kisin
qui est libre de rang $2$ sur $\SK_R$, de type de Hodge $\vv$ et de 
type galoisien $\ttt$. Pour tout $i \in \llbracket 0, f-1\rrbracket$, il existe 
une base $(e^{(i)}_\eta, e^{(i)}_{\eta'})$ dans laquelle $\Gal(L/K)$ agit par 
la matrice $\Big(\begin{matrix} \eta &0 \\ 0 & \eta' \end{matrix}\Big)$.
Notons $G^{(i)}$ la matrice de l'application $\varphi : \MK^{(i)} \to 
\MK^{(i+1)}$ dans les bases précédentes.

\begin{definit} Avec les notations ci-dessus,
le genre $(g_i(\MK))_{0\leq i\leq f-1}$ de $\MK$ est défini par
\begin{itemize}
\item[$\bullet$] $g_i(\MK) = \Ieta$ si le coefficient en
haut à gauche de $G^{(i)}$ est inversible,
\item[$\bullet$] $g_i(\MK) = \Ietap$ si le coefficient en
bas à droite de $G^{(i)}$ est inversible,
\item[$\bullet$] $g_i(\MK) = \II$ sinon.
\end{itemize}
\end{definit}

Suivant la stratégie de la démonstration du lemme 3.1.7 de \cite{CDM}, 
nous constatons que le genre $(g_i(\MK))_{0\leq i\leq f-1}$ ne dépend 
que de $\MK$ --- et notamment pas du choix des bases $(e^{(i)}_\eta, 
e^{(i)}_{\eta'})_{0\leq i\leq f-1}$.

En caractéristique $p$, le genre définit une stratification de la 
variété de Kisin. Plus précisément, considérons l'ensemble $\mathcal G = 
\{\I, \II\}$ et munissons $\mathcal G^f$ de l'ordre partiel $\leq$ 
défini sur chaque composante par $\I \leq \II$. Soit encore $\oubgenre : 
\{\Ieta,\Ietap,\II\} \to \mathcal G$ l'application qui envoie $\Ieta$ et 
$\Ietap$ sur $\I$ et $\II$ sur $\II$.
Avec ces notations, nous avons le théorème suivant dont la démonstration 
est reportée au \S \ref{sssec:demstrat}.

\begin{thm}
\label{th:stratification}
La variété $\vK^\psi(\vv, \ttt,\rhobar)$ admet une stratification par
des sous-variétés réduites et localement fermées 
$\vK^\psi_g(\vv, \ttt,\rhobar)$ {\rm (}$g=(g_i)_{0\leq i< f} \in \mathcal G^f${\rm )} 
caractérisées par :
$$\vK^\psi_g(\vv, \ttt,\rhobar)(k) = 
\left\{ \begin{array}{c}
\text{réseaux } \MK \text{ tels que} \\
\oubgenre(g_i(\MK)) = g_i, \, \forall i
\end{array} \right\}
\subset \vK^\psi(\vv, \ttt,\rhobar)(k)$$
pour toute extension finie $k$ de $k_E$. 

De plus, l'adhérence d'une strate $\vK^\psi_g(\vv, \ttt,\rhobar)$ est
égale à l'union des strates $\vK_{g'}(\vv, \ttt,\rhobar)$ prise sur 
tous les $g' \geq g$.
\end{thm}

\begin{rem}
Dans \cite{Ca3}, dans un contexte différent, une stratification a
également été définie sur certaines variétés de Kisin. Il semble
toutefois difficile \emph{a priori} de comparer ces deux constructions. 
Néanmoins, la question se pose de savoir s'il existe une méthode 
générale pour stratifier les variétés de Kisin qui conduirait, d'une 
part, à celle de cet article et, d'autre part à celle de \cite{Ca3}.
\end{rem}

Dans la suite, étant donné un point $x \in \vK^\psi(\vv,\ttt,\rhobar)$, nous 
notons $g(x) = (g_0(x), \ldots, g_{f-1}(x))$ l'unique élément $g \in 
\mathcal G^f$ pour lequel $x \in \vK^\psi_g(\vv,\ttt,\rhobar)$. Les fonctions 
$g_i$ ainsi définies sont semi-continues inférieurement sur la variété 
de Kisin, l'ensemble d'arrivée $\mathcal G = \{\I, \II\}$ étant muni de 
la topologie discrète.

En caractéristique nulle, le genre joue également un rôle important. 
De la proposition 5.2 de \cite{Br} et de la proposition 3.1.12 de 
\cite{CDM}, nous déduisons le résultat suivant.

\begin{prop}\label{propclassification}
Supposons $\ttt$ non dégénéré (voir définition \ref{hypnondegenere}).
Soit $\MK=\MK^{(0)}\oplus\cdots\oplus\MK^{(f-1)}$ un module de 
Breuil--Kisin de rang $2$ sur $\SK_{\oEp}$ de type de Hodge $\vv$ et 
de type galoisien $\ttt = \eta\oplus \eta'$.

Alors il existe des éléments $\alpha, \alpha'$ dans $\oEp^\times$ ainsi que,
pour tout $i\in \llbracket 0,f-1\rrbracket$, une base $e_\eta^{(i)},
e_{\eta'}^{(i)}$ de $\MK^{(i)}$ et des paramètres $a_i, a'_i$ (vivant
dans un espace précisé ci-après) tels que :
\begin{enumerate}[(1)]
\item pour tout $i$, la donnée de descente agit sur $e_\eta^{(i)}$
(resp. $e_{\eta'}^{(i)}$) par le caractère $\eta$
(resp. $ \eta'$) ;
\item en notant $G^{(i)}$ la matrice de $\varphi_\MK : 
\MK^{(i)}\rightarrow\MK^{(i+1)}$, la matrice $\tilde G^{(i)} = G^{(i)}$ 
pour $0\leq i\leq f-2$ (resp. $\tilde G^{(i)} = 
\begin{pmatrix}\alpha^{-1}&0\\ 0&\alpha'^{-1}\\ \end{pmatrix} G^{(f-1)}$ 
pour $i = f-1$) est de la forme :
$$\begin{array}{lll}
\text{si } g_i(\MK) = \Ieta \text{~:} &
\tilde G^{(i)}=
\begin{pmatrix} u^{e}+p& 0\\a_iu^{\gamma_{i+1}}& 1\\ \end{pmatrix} &
\text{avec } a_i \in \oEp, a'_i = 0, \medskip \\
\text{si } g_i(\MK) = \Ietap\text{~:} &
\tilde G^{(i)} =
\begin{pmatrix} 1&a'_iu^{e-\gamma_{i+1}}\\0& u^e+p\\ \end{pmatrix} &
\text{avec } a_i = 0, a'_i \in \oEp, \medskip \\
\text{si } g_i(\MK) = \II \text{~:} &
\tilde G^{(i)} =
\begin{pmatrix} a_i&u^{e-\gamma_{i+1}}\\u^{\gamma_{i+1}}& a'_i\\ \end{pmatrix} &
\text{avec } a_i, a'_i \in \m_{E'}, a_i a'_i = -p.
\end{array}$$
\end{enumerate}

De plus $(\beta, \beta', b_0, b'_0, \ldots, b_{f-1}, b'_{f-1})$ est une 
autre famille de paramètres vérifiant les conditions ci-dessus si et 
seulement s'il existe un élement $\lambda \in \oEp^\times$ pour lequel
les trois conditions suivantes soient vérifiées :
\begin{itemize}
\item pour tout $i \in \llbracket 0, f-1\rrbracket$,
$b_i = \lambda^{(-1)^{n_i}} a_i$ et $b'_i = \lambda^{-(-1)^{n_i}} a'_i$,
\item si $n_f$ est pair, alors $\alpha = \alpha'$ et $\beta = \beta'$,
\item si $n_f$ est impair, alors $\lambda = \frac \alpha \beta = \frac
{\beta'}{\alpha'}$
\end{itemize}
où, $n_i$ désigne le nombre de $\II$ parmi $g_0(\MK), \ldots, g_{i-1}(\MK)$.
\end{prop}

Considérons à présent un point fermé $x \in \vK^\psi(\vv, \ttt, \rhobar)$. 
Il résulte de la proposition \ref{propclassification} ci-dessus que 
l'image inverse de $x$ par l'application de spécialisation $\sp$ est 
en bijection avec $\Spm (W_{\oE}(k_E(x)) \otimes_{\oE} 
R_{g(x)}[1/p])$ où $W_{\oE}(k_E(x)) = \oE \otimes_{W(k_E)} W(k_E(x))$ et,
pour $g = (g_0, \ldots, g_{f-1}) \in \mathcal G^f$ :
\begin{equation}
\label{eq:defRg}
R_g = R_{g_0} \:\hat \otimes \:R_{g_1} \:\hat \otimes \: \cdots\:
\hat \otimes \:R_{g_{f-1}}
\end{equation}
avec $R_\I = \oE[[T_i]]$ (et donc correspondant à une boule ouverte) et 
$R_\II = \frac{\oE[[U_i, V_i]]}{U_iV_i + p}[1/p]$ (et donc correspondant à
une couronne ouverte). Ainsi, géométriquement, l'espace de déformations 
$D^\psi(\vv,\ttt,\rhobar)$ s'obtient comme une union de produits de 
boules et de couronnes, la manière dont cette union se réalise étant 
encodée par la variété de Kisin munie de sa stratification.

Convenons que deux variétés de Kisin stratifiées $\vK^\psi(\vv, \ttt_1, 
\rhobar_1)$ et $\vK^\psi(\vv, \ttt_2, \rhobar_2)$ sont isomorphes si elles 
partagent le même $f$ et s'il existe une permutation $\sigma$ de $\{0, 
\ldots, f-1\}$ ainsi qu'un isomorphisme $\vK^\psi(\vv, \ttt_1, \rhobar_1) \to 
\vK^\psi(\vv, \ttt_2, \rhobar_2)$ qui induit, pour tout $g = (g_0, \ldots, 
g_{f-1}) \in \mathcal G^f$, un isomorphisme entre $\vK^\psi_g(\vv, \ttt_1, 
\rhobar_1)$ et $\vK^\psi_{\sigma \cdot g} (\vv, \ttt_2, \rhobar_2)$ où,
par définition, $\sigma \cdot g = (g_{\sigma(0)}, \ldots,
g_{\sigma(f-1)})$.

\begin{conj}
\label{conj:D}
Supposons que $\ttt$ est non dégénéré.
L'espace de déformations $D^\psi(\vv,\ttt,\rhobar)$ est entièrement
déterminé par la classe d'isomorphisme de la variété de Kisin stratifiée
$\vK^\psi(\vv, \ttt,\rhobar)$.
\end{conj}

Au \S \ref{subsec:candidats}, nous donnons une recette conjecturale pour 
construire un candidat pour l'espace rigide $D^\psi(\vv,\ttt,\rhobar)$ à 
partir de la variété de Kisin stratifiée.
Nous insistons, par ailleurs, sur le fait que la donnée de la 
stratification est absolument nécessaire : la variété de Kisin $\vK^\psi(\vv, 
\ttt,\rhobar)$ vue comme variété algébrique abstraire ne permet pas, à 
elle 
seule, de reconstruire $D^\psi(\vv,\ttt,\rhobar)$, comme nous le verrons 
au \S \ref{subsec:stratexemples}.
Il n'est également \emph{pas} vrai que la variété de Kisin 
stratifiée détermine l'anneau $R^\psi(\vv,\ttt,\rhobar)$. Nous 
conjecturons toutefois la propriété plus faible suivante.

\begin{conj}
\label{conj:R}
Supposons que $\ttt$ est non dégénéré.
L'anneau de déformations $R^\psi(\vv,\ttt,\rhobar)$ est entièrement
déterminé par un gène (n'importe lequel) associé au couple $(\ttt,\rhobar)$.
\end{conj}

\subsection{Stratification par le genre}
\label{subsec:stratgenre}

Le but de ce paragraphe est de démontrer le théorème 
\ref{th:stratification} puis de donner des équations explicites pour les 
strates $\vK^\psi_g(\vv,\ttt,\rhobar)$ et d'étudier la géométrie des variétés 
de Kisin stratifiées.

\subsubsection{Démonstration du théorème \ref{th:stratification}}
\label{sssec:demstrat}

Nous rappelons que pour tout $i$ dans $\Z / f \Z$ , la matrice de 
$\varphi$ de $\MK^{(i)}$ dans $\MK^{(i+1)}$ est, dans les bases 
$(e_\eta^{(i)}, e_{\eta'}^{(i)})$ donnée par les formules (\ref{hachi})
$$H^{(i)} 
= (\det(P^{(i+1)}))^{-1}
\begin{pmatrix}
A_i & B'_i\\
 B_i & A'_i
\end{pmatrix}.$$
Par conséquent, le genre $g_i(\MK)$ se lit sur la valuation $u$-adique de $A_i$ où pour $0\leq i\leq f-2$ (et une formule analogue pour $i=f-1$) :
\begin{itemize}
\item[$\bullet$] $A_i = a_{i} a'_{i+1} u^{\alpha'_{i+1+f} + p \alpha_{i} + h_i} - b_i b'_{i+1}u^{\alpha'_{i+1} + p \alpha_{i+f} + h_{i+f}}, $
\end{itemize}
et où
$P^{(i)}=
\left(\begin{array}{cc} 
u^{\alpha_{i}}a_i&u^{\alpha'_{i}}b'_i \cr
u^{\alpha_{i+f}}b_i&u^{\alpha'_{i+f}}a'_i \cr
\end{array}\right)$ est la matrice de passage de $(e_\eta^{(i)},e_{\eta'}^{(i)})$ à $(e_0^{(i)},e_1^ {(i)})$ (voir \S \ref{subsubPi}). Rappelons que
d'après le lemme \ref{lem:valdet}, le déterminant de $P^{(i)}$ est de valuation $u$-adique $\nu$. 

\begin{lem}\label{lemstrat}
Soient $k$ une extension finie de $k_E$ et $([x_i:x_{i+f}])_{0\leq i
\leq f-1}$ un $k$-point de la variété de Kisin $\vK^{\psi}(\vv,\ttt,
\rhobar)$. Notons $\MK=(\MK^{(i)})_{0\leq i\leq f-1}$ le $k \otimes_{k_E}\SK$-réseau 
correspondant et $X = (X_i)_{i\in\Z}$ le gène associé.

\begin{enumerate}[i)]
\item Si $\gA$ est dominant en $i$ et $(i+1)$, alors :
\begin{itemize}
\item[$\bullet$] $g_i(\MK) = \Ieta$
s'il existe $j\in\{i,i+f\}$, tel que $X_j=\gB$ et $x_{j+1}\neq 0$,
\item[$\bullet$] $g_i(\MK) = \Ietap$
s'il existe $j\in\{i,i+f\}$, tel que $X_j=\gA$ et $x_j\neq0$,
\item[$\bullet$] $g_i(\MK) = \II$ sinon.
\end{itemize}
\item Si $\gA$ est dominant en $i$ et $\gB$ est dominant en $(i+1)$, 
alors :
\begin{itemize}
\item[$\bullet$] $g_i(\MK) = \II$ dans l'un des trois cas suivants :
\begin{itemize}
\item[a)] $X_i=X_{i+f}=\gA$ et $[x_i: x_{i+f}]=[x_{i+f+1}:x_{i+1}]$
\item[b)] $X_i=\gA$, $X_{i+f}\neq\gA$ et $x_ix_{i+1}=0$
\item[c)] $X_i\neq\gA$, $X_{i+f}=\gA$ et $x_{i+f}x_{i+f+1}=0$
\end{itemize}
\item[$\bullet$] $g_i(\MK) = \Ietap$ sinon.
\end{itemize}
\end{enumerate}
Les résultats précédents sont encore vrais si nous échangeons simultanément $\gA$ et $\gB$ d'une part et $\Ieta$ et $\Ietap$ d'autre part. 
\end{lem}
\begin{proof} 
Supposons $\gA$ dominant dans les couples d'allèles $\binom{X_i}{X_{i+f}}$ et
$\binom{X_{i+1}}{X_{i+1+f}}$.

Si $\binom{X_i}{X_{i+f}}=\binom{\gA}{\gB}$. D'après le lemme \ref{lem:relalphai}, nous avons
$$\alpha_{i+1}=p \alpha_{i} + h_i \mbox{ et } \alpha_{i+1+f}=p\alpha_{i+f}+h_{i+f}-e.$$
De plus comme $X_j\not=\gO$ (d'après le lemme \ref{proprelgene}) pour $j\in\{i,i+f,i+1,i+1+f\}$ et que le couple d'allèles $\binom\gO\gO$ ne peut pas apparaître, nous déduisons :
 $$\alpha_{i+1}+\alpha'_{i+f+1}=\alpha_i+\alpha_{i+f}'=\alpha'_i+\alpha_{i+f}=\alpha'_{i+1}+\alpha_{i+1+f}=\nu.$$
Ainsi $A_i=u^{\nu}(a_ia'_{i+1}-b_ib_{i+1}'u^e)$ et $A_i'=u^{\nu}(-b'_ib_{i+1}+a'_ia_{i+1}u^e).$
Le lien entre $X_i$ et $X_{i+1+f}$ implique $a_ib_{i+1}=0$. 
Donc $\MK^{(i)}$ est de genre $\II$ si et seulement si $a_i=b_{i+1}=0$
et il est de genre $\Ieta$ (resp. $\Ietap$) si $b_{i+1}\not=0$ (resp.
$a_i\not=0$).

Si $\binom{X_i}{X_{i+f}}=\binom{\gA}{\gA}$, nous obtenons de même 
$A_i=u^{\nu}(a_ia_{i+1}'-b_ib_{i+1}')$. Or $[a_i:b_i]=[a_{i+1}:b_{i+1}]$ 
et $a_{i+1}a'_{i+1}-b_{i+1}b_{i+1}'\in k_E^*$. Donc $\MK^{(i)}$ est de 
genre $\Ietap$.

Les autres cas sont analogues.
\end{proof}

\begin{rem} 
En particulier, l'énoncé du lemme \ref{lemstrat} implique que si le 
couple d'allèles en $i\in\llbracket 0, f-1\rrbracket$ du gène $X$ 
satisfait 
$\binom{X_i}{X_{i+f}}\in\left\{\binom{\gO}{\gAB},\binom{\gAB}{\gO}\right\}$ 
alors la composante $\MK^{(i)}$ est toujours de genre $\II$.
\end{rem}

Le théorème \ref{th:stratification} découle à présent simplement du 
lemme \ref{lemstrat}.

\subsubsection{Traduction génétique}\label{subsubtraduction}

Le but du paragraphe \S \ref{subsubtraduction} est d'expliquer comment les équations des strates 
de $\vK^\psi(\vv,\ttt,\rhobar)$ peuvent se lire sur le gène. À cette fin, 
nous introduisons une décoration supplémentaire, appelée 
\emph{décoration horizontale} :

\begin{definit} \label{lemdecorationsupplementaire}
Pour $i\in\llbracket 0,f-1\rrbracket$,
\begin{itemize}
\item[$\bullet$] si $\gA$ est dominant en $i$ et $\gB$ est dominant en $i+1$, la 
\emph{décoration horizontale} est composée d'un lien de $X_i$ à $X_{i+1}$ 
si $X_i=\gA$ et d'un lien de $X_{i+f}$ à $X_{i+f+1}$ si $X_{i+f}=\gA$ ;
\item[$\bullet$] si $\gB$ est dominant en $i$ et $\gA$ est dominant en $i+1$, la 
\emph{décoration horizontale} est composée d'un lien de $X_i$ à $X_{i+1}$ 
si $X_i=\gB$ et d'un lien de $X_{i+f}$ à $X_{i+f+1}$ si $X_{i+f}=\gB$.
\end{itemize}
\end{definit}

\begin{ex}
Le gène de l'exemple \ref{exMoebius} (déjà repris dans le \S
\ref{sssec:filrouge}), muni de toutes ses décorations,
se représente comme suit :

\medskip

\begin{center}
\begin{tikzpicture}[scale=0.8]
\draw [fill=fond, thick] (-0.5,-0.5) rectangle (10.5,1.5);
\draw[thick,->>] (-0.5,-0.5)--(-0.5,0.5);
\draw[thick,->>] (10.5,1.5)--(10.5,0.5);
\node[color=colB] at (0, 0) { $\gB$ };
\node[color=colB] at (0, 1) { $\gA$ };
\node[color=colB] at (1, 0) { $\gB$ };
\node[color=colB] at (1, 1) { $\gA$ };
\node[color=colB] at (2, 0) { $\gB$ };
\node[color=colB] at (2, 1) { $\gA$ };
\node[color=colB] at (3, 0) { $\gB$ };
\node[color=colB] at (3, 1) { $\gB$ };
\node[color=colA] at (4, 0) { $\gB$ };
\node[color=colA] at (4, 1) { $\gA$ };
\node[color=colA] at (5, 0) { $\gA$ };
\node[color=colA] at (5, 1) { $\gB$ };
\node[color=colA] at (6, 0) { $\gAB$ };
\node[color=colA] at (6, 1) { $\gA$ };
\node[color=colA] at (7, 0) { $\gO$ };
\node[color=colA] at (7, 1) { $\gA$ };
\node[color=colA] at (8, 0) { $\gA$ };
\node[color=colA] at (8, 1) { $\gA$ };
\node[color=colA] at (9, 0) { $\gAB$ };
\node[color=colA] at (9, 1) { $\gA$ };
\node[color=colA] at (10, 0) { $\gO$ };
\node[color=colA] at (10, 1) { $\gA$ };
\draw[color=colB,thick] (0.2,0.2)--(0.8,0.8);
\draw[color=colB,thick] (1.2,0.2)--(1.8,0.8);
\draw[color=colB,thick] (2.2,0.2)--(2.8,0.8);
\draw[color=colA,thick] (4.2,0.8)--(4.8,0.2);
\draw[color=colA,thick] (5.2,0.2)--(5.8,0.8);
\draw[color=colA,thick] (6.2,0.8)--(6.8,0.2);
\draw[color=colA,thick] (7.2,0.8)--(7.8,0.2);
\draw[color=colA,thick] (8.2,0.8)--(8.8,0.2);
\draw[color=colA,thick] (8.2,0.2)--(8.8,0.8);
\draw[color=colA,thick] (9.2,0.8)--(9.8,0.2);
\draw[color=colstr,thick] (10.2,1)--(10.7,1);
\draw[color=colstr,thick] (-0.7,0)--(-0.2,0);
\draw[color=colstr,thick] (3.2,0)--(3.8,0);
\draw[color=colstr,thick] (3.2,1)--(3.8,1);
\end{tikzpicture}
\end{center}

\noindent
les décorations horizontales étant indiquées en rouge.
\end{ex}
Le lemme \ref{lemstrat} se traduit alors simplement en termes de décorations du gène $X$ :
\begin{prop}
\label{prop:stratgene}
Soient $k$ une extension finie de $k_E$ et $([x_i:x_{i+f}])_{0\leq i
\leq f-1}$ un $k$-point de la variété de Kisin $\vK^{\psi}(\vv,\ttt,
\rhobar)$ associée au gène $X = (X_i)$. Notons $\MK=(\MK^{(i)})_{0\leq i\leq f-1}$ le $\SK_k$-réseau 
correspondant.

\medskip

Soit $i\in\llbracket 0,f-1\rrbracket$.
Si $\{X_i,X_{i+f}\}=\{\gA\gB,\gO\}$, alors $g_i(\MK) = \II$.

\medskip

Dans le cas contraire, nous avons :
\begin{enumerate}[i)]
\item si $\gA$ est dominant en $i$ et $(i+1)$, alors :
\begin{itemize}
\item s'il n'y a pas de lien entre les deux couples d'allèles ou s'il 
n'y a qu'un lien entre $X_i$ et $X_{i+f+1}$ (resp. entre $X_{i+f}$ et 
$X_{i+1}$) et $x_i=x_{i+f+1}=0$ (resp. $x_{i+f}=x_{i+1}=0$), alors 
$g_i(\MK) = \II$,
\item s'il n'y a qu'un lien entre $X_i$ et $X_{i+f+1}$ (resp. entre 
$X_{i+f}$ et $X_{i+1}$) et $x_{i+1+f}\not=0$ (resp. $x_{i+1}\not=0$), 
alors $g_i(\MK) = \Ieta$,
\item sinon $g_i(\MK) = \Ietap$,
\end{itemize}
\item si $\gA$ est dominant en $i$ et $\gB$ est dominant en $(i+1)$,
alors : 
\begin{itemize}
\item s'il y a deux liens entre les deux couples d'allèles en $i$ et en 
$i+1$ et si $[x_i:x_{i+f}]=[x_{i+f+1}:x_{i+1}]$, alors $g_i(\MK) = \II$,
\item s'il n'y a qu'un lien entre $X_i$ et $X_{i+1}$ (resp. entre 
$X_{i+f}$ et $X_{i+f+1}$) et $x_ix_{i+1}=0$ (resp. 
$x_{i+f}x_{i+f+1}=0$), alors $g_i(\MK) = \II$,
\item sinon $g_i(\MK) = \Ietap$.
\end{itemize}
\end{enumerate}
Les résultats précédents sont encore vrais si nous échangeons 
simultanément $\gA$ et $\gB$ d'une part et $\Ieta$ et $\Ietap$ 
d'autre part.
\end{prop}

\subsubsection{Géométrie des variétés de Kisin stratifiées}
\label{sssec:geomstrat}

Considérons, comme jusqu'à présent, une variété de Kisin 
$\vK^\psi(\vv,\ttt,\rhobar) \subset \prod_{i=0}^{f-1} \P^1_{k_E}$ munie de 
stratification et fixons $X = (X_i)_i$ un gène correspondant. Le but de 
du paragraphe \S \ref{sssec:geomstrat} est de démontrer que $\vK^\psi(\vv,\ttt,\rhobar)$ peut s'écrire 
comme un produit de variétés facteur sur lesquelles la stratification 
\og descend \fg\ et a une forme particulièrement simple.

Pour $i \in \llbracket 0, f-1 \rrbracket$, notons $\pr_i : 
\prod_{i=0}^{f-1} \P^1_{k_E} \to \P^1_{k_E}$ la projection sur la
$i$-ième composante. Soit $S$ l'ensemble des indices $i \in \llbracket 
0, f-1 \rrbracket$ pour lequels l'application composée :
$$\textstyle
\vK^\psi(\vv,\ttt,\rhobar) \hookrightarrow \prod_{i=0}^{f-1} \P^1_{k_E}
\stackrel{\pr_i}{\longrightarrow} \P^1_{k_E}$$
est constante (nécessairement égale à $[0:1]$ ou à $[1:0]$).
Nous supposons que $S$ est non vide. D'après le
lemme \ref{lem:ilyaO}, ceci se produit dès que $\rhobar$ n'est pas
dégénérée. Quitte à modifier le plongement
$\tau_0$, nous pouvons supposer en outre que $0 \in S$. Notons
$0 = i_1 < i_2 < \cdots < i_r$ les éléments de $S$ triés par ordre
croissant et convenons que $i_{r+1} = f$. Posons également $S_j =
\llbracket i_j, i_{j+1}-1 \rrbracket$ pour tout $j \in \llbracket 0, 
r-1 \rrbracket$. Au fragment de gène allant des positions $i_j$ à 
$i_{j+1}-1$, nous pouvons associer une variété facteur $\calV_j$ 
incluse dans $\prod_{i \in S_j} \P^1_{k_E}$, de manière à avoir 
$\vK^\psi(\vv,\ttt,\rhobar) = \calV_1 \times \cdots \times\calV_r$, cette 
identification étant compatible avec le plongement dans 
$\prod_{i=0}^{f-1} \P^1_{k_E}$ (voir \S \ref{ssec:compirred}).

\begin{lem}
\label{lem:stratgi}
Soit $j \in \llbracket 0, r-1 \rrbracket$.
Pour tout indice $i \in S_j$
le genre $g(x)$ d'un point fermé $x \in \vK^\psi(\vv,\ttt,\rhobar)$ ne
dépend que la composante de $x$ sur $\calV_j$.
\end{lem}

\begin{proof}
Le lemme est clair si $i \neq i_{j+1}-1$ car $g_i(x)$ ne dépend que 
des coordonnées de $x$ aux indices $i$ et $i+1$. Pour $i = i_{j+1}-1$,
il résulte du fait que la fonction $\pr_{i_{j+1}}$ est constante sur 
$\vK^\psi(\vv,\ttt,\rhobar)$.
\end{proof}

Le lemme \ref{lem:stratgi} nous dit exactement que, pour $i \in S_j$, 
les fonctions $g_i$ passent au quotient et définissent des fonctions, 
que nous notons encore $g_i$, sur $\calV_j$. Nous obtenons comme ceci 
une stratification sur les variétés $\calV_j$ et la
décomposition $\vK^\psi(\vv,\ttt,\rhobar) = \calV_1 \times \calV_2 \times 
\cdots \times \calV_r$ est compatible aux stratifications. Forts de ce 
résultat, nous proposons une version raffinée de la conjecture 
\ref{conj:D}.

\begin{conj}
\label{conj:D2}
Avec les notations précédentes, l'espace de déformations
$D^\psi(\vv,\ttt,\rhobar)$ s'écrit :
$$D^\psi(\vv,\ttt,\rhobar) = D_1 \times D_2 \times \cdots \times
D_r$$
où $D_i$ dépend uniquement de la variété stratifiée $\calV_i$.
\end{conj}

Fixons à présent un indice $j \in \llbracket 1,r \rrbracket$ et étudions la variété stratifiée 
$\calV_j$. Remarquons, pour commencer, que le lemme \ref{lemstrat} nous 
apprend que, s'il y a une décoration en forme 
de croix entre les indices $i$ et $(i+1)$, alors la fonction $g_i$ est 
constante égale à $\I$ sur la variété de Kisin. Notons $\calV'_j$ 
la variété stratifiée associée à la portion de gène comprise 
entre les indices $i_j$ et $i_{j+1}$ dans laquelle nous avons contracté 
les croix (comme au \S \ref{ssec:reduction}). Nous avons une 
identification canonique $\calV'_j = \calV_j$ et la stratification sur 
$\calV'_j$ coïncide avec celle sur $\calV_j$ ; seuls les \og noms \fg\ 
des strates sont modifiés par le fait que les $\I$ correspondant aux
croix disparaissent.
Ainsi, sans perte de généralité, nous pouvons 
supposer qu'il n'y a pas de croix dans les décorations du gène $X$, ce 
que nous faisons à partir de maintenant. De même, quitte à échanger 
$\eta$ et $\eta'$, nous pouvons supposer que la fonction $\pr_{i_j}$ 
est constante égale à $[0:1]$ sur $\vK^\psi(\vv,\ttt,\rhobar)$.

Introduisons deux nouvelles notations. Premièrement, si $k$ est une 
extension finie de $k_E$ et $x = [u:v] \in \P^1(k)$, posons $x^{-1} = 
[v:u] \in \P^1(k)$. Deuxièment, si $a$ et $b$ sont deux éléments d'un 
ensemble, convenons que $\delta(a,b)$ vaut $\II$ si $a = b$ et $\I$ 
sinon.

\begin{prop}
\label{prop:stratP1}
Supposons $i_{j+1} > i_j + 1$. Alors, pour $i \in S_j$, nous avons :
$$g_i(x) = \delta(x_i, x_{i+1}^{-1})
\quad \text{pour tout $k$-point $x$ de $\vK^\psi(\vv,\ttt,\rhobar)$}$$
où $(x_i)_{0 \leq i < f}$ désigne l'image de $x$ dans $\prod_{i=0}^{f-1}
\P^1_{k_E}(k)$ et, par convention, $x_f = x_0$.
\end{prop}

\begin{proof}
C'est une conséquence simple de la proposition \ref{prop:stratgene}
et du lemme \ref{lem:stratgene} ci-dessous.
\end{proof}

\begin{lem}
\label{lem:stratgene}
Sous les hypothèses précédentes, pour $i \in \llbracket i_j, i_{j+1}-1
\rrbracket$, le gène et sa décoration prennent, aux positions $i$ et $i+1$, 
l'une des formes suivantes :

\medskip

\begin{tabular}{p{4cm}p{2.1cm}p{2.1cm}p{2.1cm}p{2.1cm}}
\raisebox{0.4cm}{$\bullet$ si $i_j = i < i_{j+1}-1$ :} &
\begin{tikzpicture}[yscale=0.5]
\draw[transparent] (-0.3,0)--(-0.3,1);
\node at (0,1) { $\bullet$ };
\node at (1,1) { $\bullet$ };
\node at (1,0) { $\bullet$ };
\node at (0,0) { $\bullet$ };
\draw[thick] (0.2,0)--(0.8,0);
\end{tikzpicture} &
\begin{tikzpicture}[yscale=0.5]
\draw[transparent] (-0.3,0)--(-0.3,1);
\node at (0,1) { $\bullet$ };
\node at (1,1) { $\bullet$ };
\node at (1,0) { $\bullet$ };
\node at (0,0) { $\bullet$ };
\draw[thick] (0.2,0.8)--(0.8,0.2);
\end{tikzpicture} &
\begin{tikzpicture}[yscale=0.5]
\draw[transparent] (-0.3,0)--(-0.3,1);
\node at (0,1) { $\bullet$ };
\node at (1,1) { $\bullet$ };
\node at (1,0) { $\bullet$ };
\node at (0,0) { $\bullet$ };
\draw[thick] (0.2,0)--(0.8,0);
\draw[thick] (0.2,1)--(0.8,1);
\end{tikzpicture} \medskip \\
\raisebox{0.4cm}{$\bullet$ si $i_j < i < i_{j+1}-1$ :} &
\begin{tikzpicture}[yscale=0.5]
\draw[transparent] (-0.3,0)--(-0.3,1);
\node at (0,1) { $\bullet$ };
\node at (1,1) { $\bullet$ };
\node at (1,0) { $\bullet$ };
\node at (0,0) { $\bullet$ };
\draw[thick] (0.2,0.2)--(0.8,0.8);
\end{tikzpicture} &
\begin{tikzpicture}[yscale=0.5]
\draw[transparent] (-0.3,0)--(-0.3,1);
\node at (0,1) { $\bullet$ };
\node at (1,1) { $\bullet$ };
\node at (1,0) { $\bullet$ };
\node at (0,0) { $\bullet$ };
\draw[thick] (0.2,0.8)--(0.8,0.2);
\end{tikzpicture} &
\begin{tikzpicture}[yscale=0.5]
\draw[transparent] (-0.3,0)--(-0.3,1);
\node at (0,1) { $\bullet$ };
\node at (1,1) { $\bullet$ };
\node at (1,0) { $\bullet$ };
\node at (0,0) { $\bullet$ };
\draw[thick] (0.2,0)--(0.8,0);
\draw[thick] (0.2,1)--(0.8,1);
\end{tikzpicture} \medskip \\
\raisebox{0.4cm}{$\bullet$ si $i_j < i = i_{j+1}-1$ :} &
\begin{tikzpicture}[yscale=0.5]
\draw[transparent] (-0.3,0)--(-0.3,1);
\node at (0,1) { $\gAB$ };
\node at (1,1) { $\gO$ };
\node at (1,0) { $\bullet$ };
\node at (0,0) { $\bullet$ };
\draw[thick] (0.2,0)--(0.8,0);
\end{tikzpicture} &
\begin{tikzpicture}[yscale=0.5]
\draw[transparent] (-0.3,0)--(-0.3,1);
\node at (0,1) { $\gAB$ };
\node at (1,1) { $\gO$ };
\node at (1,0) { $\bullet$ };
\node at (0,0) { $\bullet$ };
\draw[thick] (0.2,0.2)--(0.8,0.8);
\end{tikzpicture} &
\begin{tikzpicture}[yscale=0.5]
\draw[transparent] (-0.3,0)--(-0.3,1);
\node at (0,1) { $\bullet$ };
\node at (1,1) { $\bullet$ };
\node at (1,0) { $\gO$ };
\node at (0,0) { $\gAB$ };
\draw[thick] (0.2,1)--(0.8,1);
\end{tikzpicture} &
\begin{tikzpicture}[yscale=0.5]
\draw[transparent] (-0.3,0)--(-0.3,1);
\node at (0,1) { $\bullet$ };
\node at (1,1) { $\bullet$ };
\node at (1,0) { $\gO$ };
\node at (0,0) { $\gAB$ };
\draw[thick] (0.2,0.8)--(0.8,0.2);
\end{tikzpicture}
\end{tabular}
\end{lem}

\begin{rem}
Le lemme ne dit rien dans le cas où $i_{j+1} = i_j + 1$.
\end{rem}

\begin{proof}
D'après la discussion qui suit la définition \ref{defdominance}, les
hypothèses que nous avons faites assurent que $X$ a un caractère
dominant.

Supposons tout d'abord $i_j < i < i_{j+1}-1$. Ainsi ni $i$, ni $i+1$ 
n'est dans $S$. Par le théorème \ref{thequationsKisin} et le lemme 
\ref{proprelgene}, ceci implique que $X_i$ et $X_{i+f}$ sont tous les 
deux différents de $\gAB$ et de $\gO$. Autrement dit, $X_i, X_{i+f} \in 
\{\gA, \gB\}$. Ainsi, si $Y \in \{\gA, \gB\}$ désigne l'allèle dominant 
en $i$, nous avons nécessairement $X_i = Y$ ou $X_{i+f} = Y$. Autrement
dit, au moins un lien part de $X_i$ ou de $X_{i+f}$. En outre, nous
avons supposé précédemment qu'il ne part pas deux liens en diagonale.
Il ne reste donc plus qu'à exclure le cas où un seul lien horizontal
partirait. Mais, si cela arrivait, nous aurions $\{X_i, X_{i+f}\} =
\{\gA, \gB\}$ et $Y$ resterait dominant en $(i+1)$, ce qui est 
incompatible avec des liens horizontaux.

Considérons maintenant le cas où $i = i_j < i_{j+1}-1$ et notons
$Y\in\{\gA,\gB\}$ l'allèle dominant en $i$. Comme précédemment, nous
obtenons $X_i \neq \gAB$ et $X_{i+f} \neq \gAB$, d'où nous déduisons
que $X_i = Y$ ou $X_{i+f} = Y$. Ainsi, au moins un lien part de $X_i$
ou de $X_{i+f}$. Dans le cas où $Y$ reste dominant en $(i+1)$, il ne
peut y avoir deux liens par hypothèse. Il ne peut pas non plus
y avoir un unique lien entre $X_{i+f}$ et $X_{i+1}$ car, par le 
théorème \ref{thequationsKisin} et l'hypothèse $\pr_i(\calV_j) = 
\{[0:1]\}$, cela impliquerait que $\pr_{i+1}(\calV_j) = \{[0:1]\}$
et donc $i+1 \in S$, ce qui est exclu. Il ne reste donc plus, dans
ce cas, que la possibilité d'un unique lien entre $X_i$ et $X_{i+f+1}$.
Dans le cas contraire où $Y$ n'est plus dominant en $(i+1)$, 
nous devons exclure la possibilité $X_i = Y$, $X_{i+f} \neq Y$. Mais,
de l'hypothèse $\pr_i(\calV_j) = \{[0:1]\}$, nous déduisons que
$X_{i+f} \neq \gO$. Ainsi, si $X_{i+f} \neq Y$, nous devons avoir
$X_{i+f} = \bar Y$ où $\bar Y$ est l'allèle complémentaire de $Y$,
c'est-à-dire l'unique élément de $\{\gA, \gB\}$ qui n'est pas $Y$. 
Par conséquent, si nous supposons également $X_i = Y$, nous obtenons 
$\{X_i, X_{i+f}\} = \{\gA, \gB\}$, ce qui contredit le fait que $Y$
ne soit pas dominant en $(i+1)$.

Reste enfin à traiter le cas où $i = i_{j+1} - 1$. D'après le lemme 
\ref{lem:tjsdroite}, nous avons nécessairement $X_{i+1} = \gO$ ou 
$X_{i+1+f} = \gO$. Si $X_{i+1} = \gO$, nous avons $X_i \in \{\gO, 
\gAB\}$ par le lemme \ref{proprelgene}. Mais la possibilité $X_i = \gO$ 
est exclue car $i \not\in S$. Ainsi $X_i = \gAB$. Comme la variété de 
Kisin est supposée non vide, ceci implique $X_{i+f} \neq \gAB$. Nous 
avons aussi $X_{i+f} \neq \gO$ puisque $i \not\in S$. Ainsi $X_{i+f} \in 
\{\gA, \gB\}$ et $Y = X_{i+f}$ est dominant en $i$. Nous en déduisons 
qu'un lien part de $X_{i+f}$ et, par suite, que nous sommes dans l'un 
des deux premiers cas du lemme. Le cas où $X_i = \gO$ se traite de la 
même manière, en échangeant les rôles de $i$ et $i+f$.
\end{proof}

\subsection{Quelques exemples}
\label{subsec:stratexemples}

\begin{figure}
\hspace{-1.9cm}
\begin{minipage}{18.5cm}
\small \renewcommand{\arraystretch}{1.5}
\begin{tabular}{|c|c|} 
\hline
Gène & Variété de Kisin 
\\ \hline

\begin{tikzpicture}[xscale=0.8,yscale=0.5]
\draw[transparent] (-0.5,1.5)--(-0.5,2);
\draw [fill=fond, thick] (-0.5,-0.5) rectangle (2.5,1.5);
\node[color=colA] at (0, 0) { $\gA$ };
\node[color=colA] at (0, 1) { $\gO$ };
\node[color=colB] at (1, 0) { $\gAB$ };
\node[color=colB] at (1, 1) { $\gB$ };
\node[color=colA] at (2, 0) { $\gO$ };
\node[color=colA] at (2, 1) { $\gA$ };
\draw[color=colA,thick] (-0.8,0.2)--(-0.2,0.8);
\draw[color=colA,thick] (2.2,0.8)--(2.8,0.2);
\draw[color=colstr,thick] (1.2,1)--(1.8,1);
\draw[color=colstr,thick] (0.2,0)--(0.7,0);
\end{tikzpicture}

& 

\begin{tikzpicture}
\draw[very thick] (-0.7,0)--(2.7,0);
\node[color=colE,right] at (2.7,0) { $\P^1$ };
\node[scale=0.8] at (1,-0.25) { $\I{\times}\I{\times}\I$ };
\fill[color=ptmarque] (0,0) circle (1mm);
\node[color=ptmarque,scale=0.8] at (0,0.25) { $\I{\times}\II{\times}\I$ };
\fill[color=ptmarque] (2,0) circle (1mm);
\node[color=ptmarque,scale=0.8] at (2,0.25) { $\II{\times}\I{\times}\I$ };
\end{tikzpicture}

\\ \cline{1-2}

\begin{tikzpicture}[xscale=0.8,yscale=0.5]
\draw[transparent] (-0.5,1.5)--(-0.5,2);
\draw [fill=fond, thick] (-0.5,-0.5) rectangle (2.5,1.5);
\node[color=colA] at (0, 0) { $\gA$ };
\node[color=colA] at (0, 1) { $\gO$ };
\node[color=colB] at (1, 0) { $\gAB$ };
\node[color=colB] at (1, 1) { $\gB$ };
\node[color=colB] at (2, 0) { $\gO$ };
\node[color=colB] at (2, 1) { $\gB$ };
\draw[color=colB,thick] (1.2,0.8)--(1.8,0.2);
\draw[color=colstr,thick] (-0.8,0)--(-0.2,0);
\draw[color=colstr,thick] (2.2,1)--(2.8,1);
\draw[color=colstr,thick] (0.2,0)--(0.7,0);
\end{tikzpicture}

&

\begin{tikzpicture}
\draw[very thick] (-0.7,0)--(2.7,0);
\node[color=colE,right] at (2.7,0) { $\P^1$ };
\node[scale=0.8] at (1,-0.25) { $\I{\times}\I{\times}\I$ };
\fill[color=ptmarque] (0,0) circle (1mm);
\node[color=ptmarque,scale=0.8] at (0,0.25) { $\II{\times}\I{\times}\I$ };
\fill[color=ptmarque] (2,0) circle (1mm);
\node[color=ptmarque,scale=0.8] at (2,0.25) { $\I{\times}\II{\times}\I$ };
\end{tikzpicture}

\\ \cline{1-2}

\begin{tikzpicture}[xscale=0.8,yscale=0.5]
\draw[transparent] (-0.5,1.5)--(-0.5,2);
\draw [fill=fond, thick] (-0.5,-0.5) rectangle (2.5,1.5);
\node[color=colA] at (0, 0) { $\gA$ };
\node[color=colA] at (0, 1) { $\gO$ };
\node[color=colA] at (1, 0) { $\gB$ };
\node[color=colA] at (1, 1) { $\gA$ };
\node[color=colA] at (2, 0) { $\gAB$ };
\node[color=colA] at (2, 1) { $\gA$ };
\draw[color=colA,thick] (0.2,0.2)--(0.8,0.8);
\draw[color=colA,thick] (1.2,0.8)--(1.8,0.2);
\end{tikzpicture}

& 

\begin{tikzpicture}
\draw[very thick] (-0.7,0)--(2.7,0);
\node[color=colE,right] at (2.7,0) { $\P^1$ };
\node[scale=0.8] at (1,-0.25) { $\I{\times}\I{\times}\I$ };
\fill[color=ptmarque] (0,0) circle (1mm);
\node[color=ptmarque,scale=0.8] at (0,0.25) { $\I{\times}\II{\times}\I$ };
\fill[color=ptmarque] (2,0) circle (1mm);
\node[color=ptmarque,scale=0.8] at (2,0.25) { $\I{\times}\I{\times}\II$ };
\end{tikzpicture}

\\ \cline{1-2}

\begin{tikzpicture}[xscale=0.8,yscale=0.5]
\draw[transparent] (-0.5,1.5)--(-0.5,2);
\draw [fill=fond, thick] (-0.5,-0.5) rectangle (2.5,1.5);
\node[color=colA] at (0, 0) { $\gA$ };
\node[color=colA] at (0, 1) { $\gA$ };
\node[color=colA] at (1, 0) { $\gA$ };
\node[color=colA] at (1, 1) { $\gB$ };
\node[color=colA] at (2, 0) { $\gA$ };
\node[color=colA] at (2, 1) { $\gB$ };
\draw[color=colA,thick] (-0.8,0.8)--(-0.2,0.2);
\draw[color=colA,thick] (2.2,0.2)--(2.8,0.8);
\draw[color=colA,thick] (0.2,0.2)--(0.8,0.8);
\draw[color=colA,thick] (0.2,0.8)--(0.8,0.2);
\draw[color=colA,thick] (1.2,0.2)--(1.8,0.8);
\end{tikzpicture}

& 

\begin{tikzpicture}
\draw[very thick] (-0.7,0)--(2.7,0);
\node[color=colE,right] at (2.7,0) { $\P^1$ };
\node[scale=0.8] at (1,-0.25) { $\I{\times}\I{\times}\I$ };
\fill[color=ptmarque] (0,0) circle (1mm);
\node[color=ptmarque,scale=0.8] at (0,0.25) { $\I{\times}\II{\times}\I$ };
\fill[color=ptmarque] (2,0) circle (1mm);
\node[color=ptmarque,scale=0.8] at (2,0.25) { $\I{\times}\I{\times}\II$ };
\end{tikzpicture}

\\ \cline{1-2}

\begin{tikzpicture}[xscale=0.8,yscale=0.5]
\draw[transparent] (-0.5,1.5)--(-0.5,2);
\draw [fill=fond, thick] (-0.5,-0.5) rectangle (2.5,1.5);
\node[color=colA] at (0, 0) { $\gA$ };
\node[color=colA] at (0, 1) { $\gA$ };
\node[color=colA] at (1, 0) { $\gA$ };
\node[color=colA] at (1, 1) { $\gB$ };
\node[color=colA] at (2, 0) { $\gB$ };
\node[color=colA] at (2, 1) { $\gA$ };
\draw[color=colA,thick] (-0.8,0.8)--(-0.8,0.8);
\draw[color=colA,thick] (-0.8,0.2)--(-0.2,0.8);
\draw[color=colA,thick] (2.2,0.8)--(2.8,0.2);
\draw[color=colA,thick] (0.2,0.2)--(0.8,0.8);
\draw[color=colA,thick] (0.2,0.8)--(0.8,0.2);
\draw[color=colA,thick] (1.2,0.2)--(1.8,0.8);
\end{tikzpicture}

&

\begin{tikzpicture}
\draw[very thick] (-0.7,0)--(2.7,0);
\node[color=colE,right] at (2.7,0) { $\P^1$ };
\node[scale=0.8] at (1,-0.25) { $\I{\times}\I{\times}\I$ };
\fill[color=ptmarque] (0,0) circle (1mm);
\node[color=ptmarque,scale=0.8] at (0,0.25) { $\I{\times}\II{\times}\I$ };
\fill[color=ptmarque] (2,0) circle (1mm);
\node[color=ptmarque,scale=0.8] at (2,0.25) { $\I{\times}\I{\times}\II$ };
\end{tikzpicture}

\\ \cline{1-2}

\begin{tikzpicture}[xscale=0.8,yscale=0.5]
\draw[transparent] (-0.5,1.5)--(-0.5,2);
\draw [fill=fond, thick] (-0.5,-0.5) rectangle (2.5,1.5);
\node[color=colA] at (0, 0) { $\gA$ };
\node[color=colA] at (0, 1) { $\gO$ };
\node[color=colB] at (1, 0) { $\gB$ };
\node[color=colB] at (1, 1) { $\gB$ };
\node[color=colB] at (2, 0) { $\gAB$ };
\node[color=colB] at (2, 1) { $\gB$ };
\draw[color=colB,thick] (1.2,0.8)--(1.8,0.2);
\draw[color=colB,thick] (1.2,0.2)--(1.8,0.8);
\draw[color=colstr,thick] (-0.8,0)--(-0.2,0);
\draw[color=colstr,thick] (2.2,1)--(2.8,1);
\draw[color=colstr,thick] (0.2,0)--(0.8,0);
\end{tikzpicture}

&

\begin{tikzpicture}
\draw[very thick] (-0.7,0)--(2.7,0);
\node[color=colE,right] at (2.7,0) { $\P^1$ };
\node[scale=0.8] at (1,-0.25) { $\I{\times}\I{\times}\I$ };
\fill[color=ptmarque] (0,0) circle (1mm);
\node[color=ptmarque,scale=0.8] at (0,0.25) { $\II{\times}\I{\times}\I$ };
\fill[color=ptmarque] (2,0) circle (1mm);
\node[color=ptmarque,scale=0.8] at (2,0.25) { $\I{\times}\I{\times}\II$ };
\end{tikzpicture}

\\ \hline

\begin{tikzpicture}[xscale=0.8,yscale=0.5]
\draw[transparent] (-0.5,1.5)--(-0.5,2);
\draw [fill=fond, thick] (-0.5,-0.5) rectangle (2.5,1.5);
\node[color=colA] at (0, 0) { $\gAB$ };
\node[color=colA] at (0, 1) { $\gO$ };
\node[color=colA] at (1, 0) { $\gO$ };
\node[color=colA] at (1, 1) { $\gA$ };
\node[color=colB] at (2, 0) { $\gAB$ };
\node[color=colB] at (2, 1) { $\gB$ };
\draw[color=colstr,thick] (-0.8,0)--(-0.2,0);
\draw[color=colstr,thick] (2.2,1)--(2.8,1);
\draw[color=colstr,thick] (1.2,1)--(1.8,1);
\end{tikzpicture}

& 

\begin{tikzpicture}
\draw[very thick] (-0.7,0)--(2.7,0);
\node[color=colE,right] at (2.7,0) { $\P^1$ };
\node[scale=0.8] at (1.5,-0.25) { $\II{\times}\I{\times}\I$ };
\fill[color=ptmarque] (0,0) circle (1mm);
\node[color=ptmarque,scale=0.8] at (0,0.25) { $\II{\times}\II{\times}\II$ };
\end{tikzpicture}

\\ \hline

\begin{tikzpicture}[xscale=0.8,yscale=0.5]
\draw[transparent] (-0.5,-0.5)--(-0.5,-0.8);
\draw [fill=fond, thick] (-0.5,-0.5) rectangle (2.5,1.5);
\node[color=colA] at (0, 0) { $\gA$ };
\node[color=colA] at (0, 1) { $\gO$ };
\node[color=colB] at (1, 0) { $\gB$ };
\node[color=colB] at (1, 1) { $\gA$ };
\node[color=colB] at (2, 0) { $\gAB$ };
\node[color=colB] at (2, 1) { $\gB$ };
\draw[color=colB,thick] (1.2,0.2)--(1.8,0.8);
\draw[color=colstr,thick] (-0.8,0)--(-0.2,0);
\draw[color=colstr,thick] (2.2,1)--(2.8,1);
\draw[color=colstr,thick] (0.2,0)--(0.8,0);
\end{tikzpicture}

&

\begin{tikzpicture}
\draw[transparent] (-2,0)--(2,1.1);
\draw [very thick] (-2.4,-0.1)--(0.4,0.6);
\draw [very thick] (-0.4,0.6)--(2.4,-0.1);
\fill[color=ptmarque] (0,0.5) circle (1mm);
\node[color=ptmarque,scale=0.8] at (0,0.8) { $\II{\times}\II{\times}\II$ };
\node[scale=0.8,right] at (1,0.4) { $\I{\times}\I{\times}\II$ };
\node[scale=0.8,left] at (-1,0.4) { $\II{\times}\I{\times}\I$ };
\node[color=colE,right] at (1.6,-0.25) { $\P^1$ };
\node[color=colE,left] at (-1.6,-0.25) { $\P^1$ };
\end{tikzpicture}

\\ \hline

\end{tabular}
\hfill
\begin{tabular}{|c|c|}
\hline
Gène & Variété de Kisin 
\\ \hline

\begin{tikzpicture}[xscale=0.8,yscale=0.5]
\draw[transparent] (-0.5,-0.5)--(-0.5,-1.4);
\draw [fill=fond, thick] (-0.5,-0.5) rectangle (2.5,1.5);
\node[color=colA] at (0, 0) { $\gA$ };
\node[color=colA] at (0, 1) { $\gO$ };
\node[color=colB] at (1, 0) { $\gA$ };
\node[color=colB] at (1, 1) { $\gB$ };
\node[color=colB] at (2, 0) { $\gAB$ };
\node[color=colB] at (2, 1) { $\gB$ };
\draw[color=colB,thick] (1.2,0.8)--(1.8,0.2);
\draw[color=colstr,thick] (-0.8,0)--(-0.2,0);
\draw[color=colstr,thick] (2.2,1)--(2.8,1);
\draw[color=colstr,thick] (0.2,0)--(0.8,0);
\end{tikzpicture}

&

\begin{tikzpicture}
\draw[transparent] (-2,-0.78)--(2,1.3);
\draw [very thick] (-2.4,-0.1)--(0.4,0.6);
\draw [very thick] (-0.4,0.6)--(2.4,-0.1);
\fill[color=ptmarque] (0,0.5) circle (1mm);
\node[color=ptmarque,scale=0.8] at (0,0.8) { $\I{\times}\II{\times}\I$ };
\fill[color=ptmarque] (-2,0) circle (1mm);
\node[color=ptmarque,scale=0.8] at (-1.5,-0.25) { $\I{\times}\I{\times}\II$ };
\fill[color=ptmarque] (2,0) circle (1mm);
\node[color=ptmarque,scale=0.8] at (1.5,-0.25) { $\II{\times}\I{\times}\I$ };
\node[scale=0.8,right] at (1,0.4) { $\I{\times}\I{\times}\I$ };
\node[scale=0.8,left] at (-1,0.4) { $\I{\times}\I{\times}\I$ };
\node[color=colE,right] at (2.4,-0.1) { $\P^1$ };
\node[color=colE,left] at (-2.4,-0.1) { $\P^1$ };
\end{tikzpicture}

\\ \hline

\begin{tikzpicture}[xscale=0.8,yscale=0.5]
\draw[transparent] (-0.5,-0.5)--(2.5,-3.8);
\draw [fill=fond, thick] (-0.5,-0.5) rectangle (2.5,1.5);
\node[color=colA] at (0, 0) { $\gA$ };
\node[color=colA] at (0, 1) { $\gO$ };
\node[color=colB] at (1, 0) { $\gB$ };
\node[color=colB] at (1, 1) { $\gB$ };
\node[color=colA] at (2, 0) { $\gAB$ };
\node[color=colA] at (2, 1) { $\gA$ };
\draw[color=colA,thick] (-0.8,0.2)--(-0.2,0.8);
\draw[color=colA,thick] (2.2,0.8)--(2.8,0.2);
\draw[color=colstr,thick] (0.2,0)--(0.8,0);
\draw[color=colstr,thick] (1.2,0)--(1.8,0);
\draw[color=colstr,thick] (1.2,1)--(1.8,1);
\end{tikzpicture}

&

\begin{tikzpicture}
\draw[transparent] (0,-0.1)--(6,4.3);
\draw[fill=colG!60] (0,0) rectangle (6,4);
\clip (0,0) rectangle (6,4);
\draw (6.1,-0.1) ellipse (1.1cm and 0.6cm);
\node[above left,xscale=0.6,yscale=0.7] at (6,0) { $\mathbb P^1 {\times} \mathbb P^1$ };
\node[color=colF,scale=0.8] at (5,3.5) { $\I{\times}\I{\times}\I$ };
\begin{scope}[thick,color=colB]
\draw (0,3.5)--(6,1)
  node[sloped,midway,xshift=-0.2cm,yshift=1.5mm,scale=0.8] { $\I{\times}\II{\times}\I$ };
\draw (0,0.5)--(6,2)
  node[sloped,midway,xshift=-0.3cm,yshift=-2mm,scale=0.8] { $\I{\times}\I{\times}\II$ };
\draw (0.5,0)--(1.5,4)
  node[midway,sloped,yshift=1.5mm,xshift=-1mm,scale=0.8] { $\II{\times}\I{\times}\I$ };
\end{scope}
\fill[color=ptmarque] (4.5,1.625) circle (1mm);
\node[color=ptmarque,scale=0.8,xshift=3mm,yshift=3.5mm] at (4.5,1.625) { $\I\!{\times}\II{\times}\II$ };
\fill[color=ptmarque] (0.667,0.667) circle (1mm);
\node[color=ptmarque,scale=0.8,right,xshift=0.5mm,yshift=-2.5mm] at (0.667,0.667) { $\II{\times}\I{\times}\II$ };
\fill[color=ptmarque] (1.245,2.981) circle (1mm);
\node[color=ptmarque,scale=0.8,right,xshift=0.5mm,yshift=1.5mm] at (1.245,2.981) { $\II{\times}\II{\times}\I$ };
\end{tikzpicture}

\\ \hline

\begin{tikzpicture}[xscale=0.8,yscale=0.5]
\draw[transparent] (-0.5,-0.5)--(2.5,-4.2);
\draw [fill=fond, thick] (-0.5,-0.5) rectangle (2.5,1.5);
\node[color=colA] at (0, 0) { $\gA$ };
\node[color=colA] at (0, 1) { $\gA$ };
\node[color=colB] at (1, 0) { $\gB$ };
\node[color=colB] at (1, 1) { $\gA$ };
\node[color=colB] at (2, 0) { $\gB$ };
\node[color=colB] at (2, 1) { $\gB$ };
\draw[color=colB,thick] (1.2,0.2)--(1.8,0.8);
\draw[color=colstr,thick] (0.2,0)--(0.8,0);
\draw[color=colstr,thick] (0.2,1)--(0.8,1);
\draw[color=colstr,thick] (2.2,0)--(2.8,0);
\draw[color=colstr,thick] (2.2,1)--(2.8,1);
\draw[color=colstr,thick] (-0.8,0)--(-0.2,0);
\draw[color=colstr,thick] (-0.8,1)--(-0.2,1);
\end{tikzpicture}

&

\begin{tikzpicture}
\draw[transparent] (0,-0.1)--(5,4.4);
\fill[color=colG] (0,0)--(3,0)--(5,1.5)--(2,1.5)--cycle;
\fill[color=colG!50] (0,0)--(2,1.5)--(2,4)--(0,2.5)--cycle;
\draw (0,0)--(3,0)--(5,1.5)--(2,1.5)--(2,4)--(0,2.5)--cycle;
\begin{scope}[color=colB,thick]
\draw (0,2.5)--(2,1.5)
  node[sloped,midway,yshift=1.7mm,scale=0.8] { $\II{\times}\I{\times}\I$ };
\draw (2,1.5)--(0,0)
  node[sloped,midway,yshift=1.7mm,scale=0.8] { $\I{\times}\II{\times}\I$ };
\draw (0,0)--(5,1.5)
  node[sloped,midway,yshift=1.7mm,scale=0.8] { $\I{\times}\I{\times}\II$ };
\end{scope}
\fill[color=ptmarque] (2,1.5) circle (1mm);
\node[color=ptmarque,scale=0.8,above right] at (2,1.5) { $\II{\times}\II{\times}\I$ };
\fill[color=ptmarque] (0,0) circle (1mm);
\node[color=ptmarque,scale=0.8,yshift=-1mm,below] at (0,0) { $\I{\times}\II{\times}\II$ };
\node[above left,xscale=0.6,yscale=0.5,xslant=1] at (3,0) { $\mathbb P^1 \times \mathbb P^1$ };
\node[below left,rotate=40,scale=0.6,xshift=0.5mm,xslant=0.9] at (2,4) { $\mathbb P^1 \times \mathbb P^1$ };
\end{tikzpicture}

\\ \hline

\end{tabular}
\renewcommand{\arraystretch}{1}
\end{minipage}
\caption{Liste des gènes et de leurs variétés de Kisin pour $f=3$}
\label{fig:f3}
\end{figure}

\subsubsection{Pour $f=2$}

D'après le tableau de la figure \ref{fig:f2} (page \pageref{fig:f2}), 
les deux seules variétés de Kisin qui peuvent apparaître pour $f=2$ sont 
$\Spec k_E$ et $\P^1_{k_E}$. En outre, dans le second cas, à isomorphisme près,
la stratification prend la forme suivante :
\begin{center}
\begin{tikzpicture}
\draw[very thick] (-0.7,0)--(2.7,0);
\node[color=colE,right] at (2.7,0) { $\P^1$ };
\node[scale=0.8] at (1,-0.25) { $\I{\times}\I$ };
\fill[color=ptmarque] (0,0) circle (1mm);
\node[color=ptmarque,scale=0.8] at (0,0.25) { $\I{\times}\II$ };
\fill[color=ptmarque] (2,0) circle (1mm);
\node[color=ptmarque,scale=0.8] at (2,0.25) { $\II{\times}\I$ };
\end{tikzpicture}
\end{center}
Comme nous l'avons démontré dans \cite{CDM}, l'anneau de déformations 
correspondant $R^\psi(\vv,\ttt,\rhobar)$ est, dans ce cas, isomorphe à 
$\frac{\oE[[T,U,V]]}{UV + p^2}$ et sa \og fibre générique \fg\ 
$D^\psi(\vv, \ttt, \rhobar)$ s'écrit comme un produit $B \times C$ où 
$B$ est une boule ouverte et $C$ est la couronne ouverte bordée par le 
cercle unité et celui de rayon $p^{-2}$. Cette couronne $C$ s'écrit 
comme l'union de deux couronnes ouvertes --- celle bordée par les 
cercles de rayon $1$ et $p^{-1}$ d'une part et celle bordée par les 
cercles de rayon $p^{-1}$ et $p^{-2}$ d'autre part --- et du cercle de 
rayon $p^{-1}$ qui s'écrit lui-même comme une union de boules ouvertes 
indexée par les points fermés de $\P^1_{k_E} \backslash \{0, \infty\}$. 
Nous constatons que ces résultats sont en accord avec la stratification 
sur $\vK^\psi(\vv,\ttt,\rhobar)$ que nous avons décrite précédemment.

\subsubsection{Pour $f=3$}

Il n'est pas difficile de faire la liste de tous les gènes possibles de 
longueur $3$ satisfaisant aux conditions des lemmes \ref{proprelgene} et 
\ref{pasABBA} et ne contenant pas le couple d'allèles $\binom\gO\gO$. Le 
tableau de la figure \ref{fig:f3} les présente à symétrie près et fait 
apparaître, pour chacun d'eux, la variété de Kisin stratifiée 
correspondante.

Nous constatons que les variétés de Kisin stratifiées sont isomorphes
pour les six premiers exemples de la colonne de gauche. Ainsi, d'après
la conjecture \ref{conj:D}, les espaces de déformations $D^\psi(\vv,
\ttt,\rhobar)$ devraient être isomorphes entre eux dans tous ces cas.
En nous appuyant sur les résultats obtenus pour $f=2$, nous pouvons
aller plus loin et proposer un candidat explicite pour $D^\psi(\vv,
\ttt,\rhobar)$ qui est 
$\Spm \big( \frac{\oE[[T_1,T_2,U,V]]}{UV + p^2}[1/p]\big)$.

Il est plus aventureux d'exhiber un candidat pour l'anneau 
$R^\psi(\vv, \ttt,\rhobar)$ lui-même. Par exemple, il existe des
couples $(\ttt, \rhobar)$ qui tombent sous le coup de l'exemple
de la cinquième ligne et pour lesquels il y a exactement $3$ poids
de Serre commun à $\ttt$ et $\rhobar$. Ainsi, d'après la conjecture
de Breuil--Mézard raffinée \cite{BM,EG} --- qui a été démontrée, dans ce 
cas, par Emerton et Gee dans \cite{EG} --- la multiplicité d'Hilbert--Samuel 
de $R^\psi(\vv, \ttt,\rhobar)$ est nécessairement $\geq 3$. Ceci exclut 
donc la possibilité $\frac{\oE[[T_1,T_2,U,V]]}{UV + p^2}$ pour 
$R^\psi(\vv, \ttt,\rhobar)$. La même discussion vaut pour l'exemple
de la sixième ligne. En revanche, pour les quatre premiers 
exemples, cette possibilité n'est pas exclue par un argument de ce type 
et, de fait, nous conjecturons que $R^\psi(\vv, \ttt,\rhobar) = 
\frac{\oE[[T_1,T_2,U,V]]}{UV + p^2}$ dans ces cas.

\subsubsection{Chaînes de $\P^1$}

Considérons une variété stratifiée facteur (voir \S
\ref{sssec:geomstrat}) $\calV$ portant sur $(\ell+1)$ indices, qui 
est isomorphe à une chaîne de $\ell$ copies de $\P^1_{k_E}$ représentée 
ci-dessous :

\begin{center}
\begin{tikzpicture}
\begin{scope}[thick]
\draw (-0.3,-0.3) to[out=70,in=200] (1.3,1.3);
\draw (0.7,1.3) to[out=-70,in=-200] (2.3,-0.3);
\draw (1.7,-0.3) to[out=70,in=200] (3.3,1.3);
\draw (5.2,-0.3) to[out=70,in=200] (6.8,1.3);
\draw (6.2,1.3) to[out=-70,in=-200] (7.8,-0.3);
\end{scope}
\begin{scope}[thick]
\clip (2.5,1.5) rectangle (3.2,-0.5);
\draw (2.7,1.3) to[out=-70,in=-200] (4.3,-0.3);
\end{scope}
\begin{scope}[thick]
\clip (5,1.5) rectangle (5.8,-0.5);
\draw (4.2,1.3) to[out=-70,in=-200] (5.8,-0.3);
\end{scope}
\node[scale=2] at (4.25,0.5) { $\cdots$ };
\fill[color=ptmarque] (-0.2,-0.08) circle (1mm);
\node[color=ptmarque,scale=0.8,yshift=1mm,right] at (-0.2,-0.08) { $P_0$ };
\fill[color=ptmarque] (0.8,1.08) circle (1mm);
\node[color=ptmarque,scale=0.8,yshift=1mm,left] at (0.8,1.08) { $P_1$ };
\fill[color=ptmarque] (1.8,-0.08) circle (1mm);
\node[color=ptmarque,scale=0.8,yshift=1mm,right] at (1.8,-0.08) { $P_2$ };
\fill[color=ptmarque] (2.8,1.08) circle (1mm);
\node[color=ptmarque,scale=0.8,yshift=1mm,left] at (2.8,1.08) { $P_3$ };
\fill[color=ptmarque] (5.3,-0.08) circle (1mm);
\node[color=ptmarque,scale=0.8,yshift=1mm,right] at (5.3,-0.08) { $P_{\ell-2}$ };
\fill[color=ptmarque] (6.3,1.08) circle (1mm);
\node[color=ptmarque,scale=0.8,yshift=1mm,left] at (6.3,1.08) { $P_{\ell-1}$ };
\fill[color=ptmarque] (7.3,-0.08) circle (1mm);
\node[color=ptmarque,scale=0.8,yshift=1mm,right] at (7.3,-0.08) { $P_{\ell}$ };
\end{tikzpicture}
\end{center}

\noindent
(voir également \S \ref{exchaineP1}).
Nous supposons en outre que la stratification sur $\calV$ est ainsi 
donnée : le genre du point $P_i$ est $\I \times \cdots \times \I \times 
\II \times \I \times \cdots \times \I$ (le $\II$ étant à la $i$-ième 
position), tandis que le genre des autres points est $\I \times \cdots 
\times \I$.
Un candidat naturel pour le facteur $D$ correspondant est alors 
l'espace rigide associé à l'anneau :
$$\left(\frac{\oE[[T_1, \ldots, T_\ell, U, V]]}{(UV + p^{\ell+1})}\right)
\!\left[p^{-1}\right]$$
En effet, l'espace $D$ ci-dessus est isomorphe au produit $B^\ell \times 
C$ où $B$ est une boule ouverte et $C$ est la couronne ouverte d'\og 
épaisseur $(\ell+1)$ \fg\ comprise entre les cercles de rayon $1$ et 
$p^{-\ell-1}$ qui peut s'écrire\footnote{Cette écriture est abusive
car les $B_i$ ne sont pas des espaces rigides.} :
$$C = A_0 \sqcup B_1 \sqcup A_1 \sqcup B_2 \sqcup \,\cdots\, \sqcup
A_{\ell-1} \sqcup B_\ell \sqcup A_\ell$$
où $A_i$ est la couronne ouverte d'\og épaisseur $1$ \fg\ comprise
entre les cercles de rayon $p^{-i}$ et $p^{-i-1}$ et $B_i$ est le
cercle de rayon $p^{-i}$ qui est lui-même recouvert par une union
de boules ouvertes indexée par $\P^1_{k_E} \backslash \{0,\infty\}$.
Sur la variété de Kisin stratifiée, chaque $A_i$ correspond donc au
point $P_i$ tandis que $B_i$ correspond au $i$-ième $\P^1_{k_E}$ privé
des points $P_{i-1}$ et $P_i$.

\subsection{Des candidats pour les espaces de déformations}
\label{subsec:candidats}

Considérons une représentation $\rhobar$ non dégénérée ainsi qu'un type 
$\ttt = \eta \oplus \eta'$ non dégénéré. La variété de Kisin 
$\vK^\psi(\vv,\ttt,\rhobar)$ munie de sa stratification s'écrit alors 
comme un produit de variétés stratifiées facteur, conformément aux 
résultats du \S \ref{sssec:geomstrat} :
$$\vK^\psi(\vv,\ttt,\rhobar) = \calV_1 \times \cdots \times \calV_r.$$
D'après la conjecture \ref{conj:D2}, à chaque facteur $\calV_j$ devrait 
être associé un espace rigide analytique $D_j$, de façon à ce que 
$D^\psi(\vv,\ttt,\rhobar) = D_1 \times \cdots \times D_r$. Le but 
de ce paragraphe est d'\emph{esquisser} la construction d'un candidat 
possible pour les espaces $D_j$.

\subsubsection*{Une construction préliminaire}

Rappelons que si $a$ et $b$ sont deux éléments d'un ensemble, nous avons 
défini le symbole $\delta(a,b)$ par $\delta(a,b) = \II$ si $a = b$ et 
$\delta(a,b) = \I$ sinon. Rappelons également que si $k$ est un corps et 
si $x = [u:v] \in \P^1(k)$, nous notons $x^{-1} = [v:u]$. Étant donnés 
un entier $d \geq 0$ et un élément $x_0 \in \P^1(k_E)$, nous allons 
construire un schéma formel $\calX_{d,x_0}$, un sous-schéma fermé $\bar 
\calZ_{d,x_0}$ de la fibre spéciale $\bar \calX_{d,x_0}$ de 
$\calX_{d,x_0}$ ainsi qu'un morphisme $\alpha_{d,x_0} : \calX_{d,x_0} 
\to (\P^1_{\oE})^d$ vérifiant les deux propriétés suivantes :
\begin{itemize}
\item[\emph{i)}] le morphisme $\alpha_{d,x_0}$ induit un isomorphisme 
$\bar \calZ_{d,x_0} \to (\P^1_{k_E})^d$, 
\item[\emph{ii)}] pour tout point fermé $x \in \bar \calZ_{d,x_0}$, l'image 
inverse de $x$ par le morphisme de spécialisation 
est isomorphe à $\Spm(W_{\oE}(k_E(x)) \otimes_{\oE} R_{g(x)}[1/p])$ où 
$$\begin{array}{l}
g(x) = \big(\delta(x_0, x_1^{-1}), \delta(x_1, x_2^{-1}),
\ldots, \delta(x_{d-1}, x_d^{-1})\big) \in \mathcal G^d \medskip \\
\hspace{3cm} \text{avec }
\alpha_{d,x_0}(x) = (x_1, x_2, \ldots, x_d) \in (\P^1_{k_E})^d
\end{array}$$
et nous rappelons que $R_{g(x)}$ est défini par la formule \eqref{eq:defRg},
page \pageref{eq:defRg}.
\end{itemize}

\medskip

Nous procédons par récurrence sur $d$. 
Pour $d = 0$, nous posons $\calX_{0,x_0} = \Spf\oE$, $\bar \calZ_{0,x_0} 
= \bar \calX_{0,x_0}$ et $\alpha_0 = \text{id}$.
Supposons à présent que $\calX_{d-1,x_0}$, $\bar \calZ_{d-1,x_0}$ et 
$\alpha_{d-1,x_0}$ sont construits. Nous posons $\calY = 
\calX_{d-1,x_0} \times_{\oE} \P^1_{\oE}$ et notons $\bar \calY$ sa fibre 
spéciale. Nous avons donc $\bar \calY = \bar \calX_{d-1,x_0} 
\times_{k_E} \P^1_{k_E}$. Soit $\bar \calD$ le diviseur de $\bar \calY$ 
défini comme l'image du morphisme :
$$\text{id} \times (\text{inv} \circ \varpi_{d-1} \circ
\alpha_{d-1}) : \bar \calX_{d-1} \to \bar \calY$$
où $\varpi_{d-1} : (\P^1_{k_E})^{d-1} \to \P^1_{k_E}$ est la projection 
sur la dernière composante et $\text{inv}$ est le morphisme $x \mapsto 
x^{-1}$. Nous définissons $\calX_{d,x_0}$ comme l'éclaté de $\calY$ par 
rapport au sous-schéma fermé $\bar \calD$ et $\bar \calZ_{d,x_0}$ comme 
le transformé strict de $\bar \calZ_{d-1,x_0}$ dans $\calX_{d,x_0}$.
Enfin, le morphisme $\alpha_{d,x_0}$ est la composée $\calX_{d,x_0} \to 
\calY \to (\P^1_{\oE})^d$ où la première flèche est le morphisme 
canonique et la seconde est $\alpha_{d-1,x_0} \times \text{id}$.

Nous pouvons alors vérifier que le triplet $(\calX_{d,x_0}, \bar 
\calZ_{d,x_0}, \alpha_{d,x_0})$ ainsi construit satisfait aux propriétés 
\emph{i)} et \emph{ii)} sus-mentionnées.

\subsubsection*{Un candidat pour $D_j$}

Dans ce paragraphe, nous reprenons les notations $\pr_i$, $S$, $i_j$, 
$S_j$ du \S \ref{sssec:geomstrat}. Sans restreindre la généralité, nous 
supposons également comme dans le \S \ref{sssec:geomstrat} que le 
plongement $\tau_0$ est choisi de sorte que $0 \in S$.

Fixons un entier $j \in \llbracket 1, r \rrbracket$. Si $i_{j+1} = i_j + 
1$, la variété facteur $\calV_j$ est réduite à un point. Le candidat que 
nous proposons pour $D_j$ est alors naturellement 
$\Spm R_{g_{i_j}(x)}[1/p]$ où 
$x$ est l'unique point de $\calV_j$. Supposons donc, à présent, que 
$i_{j+1} > i_j + 1$. Dans un premier temps, nous supposons également 
qu'il n'y a pas de croix dans les décorations du gène entre les
positions $i_j$ et $i_{j+1}$.
Par définition, les fonctions $\pr_{i_j}$ et $\pr_{i_{j+1}}$ sont 
constantes sur la variété de Kisin $\vK^\psi(\vv,\ttt,\rhobar)$. Notons 
$s$ (resp. $t$) la valeur prise par $\pr_{i_j}$ (resp. $\pr_{i_{j+1}}$) 
; elle vaut nécessairement $[0:1]$ ou $[1:0]$. Posons $\ell = i_{j+1} - 
i_j = \text{Card } S_j$ et, pour simplifier les écritures, notons
$\calX = \calX_{\ell,s}$ et $\alpha = \alpha_{\ell,s}$.
Définissons également $\beta$ comme la composée :
$$\textstyle
\beta : \calV_j \longrightarrow \prod_{i\in S_j} \P^1_{k_E}
\longrightarrow (\P^1_{k_E})^\ell$$
où la première flèche est l'inclusion déduite des $\pr_i$ ($i \in
S_j$) et où la deuxième flèche est le morphisme
$(x_{i_j}, x_{i_j+1}, \ldots, x_{i_{j+1}-1})
\mapsto (x_{i_j+1}, \ldots, x_{i_{j+1}-1}, t)$.
Du fait que la fonction $\pr_{i_j}$ est constante, nous déduisons
que $\beta$ est une immersion fermée et identifie ainsi $\calV_j$
à un sous-schéma fermé de $(\P^1_{k_E})^\ell$, ce dernier étant
lui-même un sous-schéma fermé de $\bar \calX$ (la fibre spéciale de 
$\calX$) \emph{via} $\alpha$.
Ainsi, cela a du sens de considérer le tube $]\calV_j[_{\calX}$. En 
outre, il est facile de se convaincre en remontant les définitions et en 
utilisant la propriété \emph{ii)} vérifiée par les variétés 
$\calX_{d,x_0}$ que l'image 
inverse par le morphisme de spécialisation $]\calV_j[_{\calX} \to 
\calV_j$ d'un point fermé $x \in \calV_j$ est isomorphe à 
$\Spm(W_{\oE}(k_E(x)) \otimes_{\oE} R_{g(x)}[1/p])$ où $g(x) = 
(g_i(x))_{i \in S_j}$.
Autrement dit, $]\calV_j[_{\calX}$ vérifie les propriétés attendues
de $D_j$ et, de fait, c'est le candidat que nous proposons pour être 
$D_j$.

Pour conclure, mentionnons rapidement comment la construction précédente 
s'étend au cas où les décorations du gène font apparaître des croix. 
Supposons donc qu'il y ait au moins une croix dans la portion de gène 
correspondant aux indices de $S_j$. Nous utilisons les notations 
$\calV'_j$ pour la variété stratifiée facteur associée à la portion
de gène ci-dessus dans laquelle toutes les croix ont été contractées.
Concrètement, la variété stratifiée $\calV'_j$ n'est autre que $\calV_j$
mais le genre d'un point fermé $x \in \calV'_j$ n'est pas $g(x)$ mais
le sous-uplet $g'(x) = (g_i(x))_{x \in S'_j}$ où $S'_j$ est le
sous-ensemble de $S_j$ où nous avons retiré les indices $i$ pour lesquels
la décoration entre $i$ et $(i+1)$ est une croix. Notons que ceci ne
modifie pas la stratification car, d'après la proposition 
\ref{prop:stratgene}, les fonctions $g_i$ pour $i \in S_j \backslash
S'_j$ sont constantes, égales à $\I$.
Cette dernière remarque justifie également le fait que nous nous 
attendions à ce que l'espace rigide $D_j$ associé à $\calV_j$ soit 
$D'_j \times B^c$ où $D'_j$ est l'espace associé à $\calV'_j$, $B$ 
est une boule ouverte et $c = \text{Card}(S_j \backslash S'_j)$ est 
le nombre de croix qui ont été contractées.

\end{document}